\documentclass[a4paper, 11pt]{amsart}
\usepackage{amssymb,verbatim,epsf,hyperref,placeins,graphicx}
\usepackage[all]{xy}

\newtheorem{theorem}{Theorem}
\newtheorem{definition}[theorem]{Definition}
\newtheorem{conjecture}[theorem]{Conjecture}

\newtheorem{proposition}[theorem]{Proposition}
\newtheorem{lemma}[theorem]{Lemma}
\newtheorem{corollary}[theorem]{Corollary}
\newtheorem{theorem-question}[theorem]{Theorem-Question}

\theoremstyle{definition}
\newtheorem{remark}[theorem]{Remark}
\newtheorem{example}[theorem]{Example}

\def\bbz{\mathbb{Z}}
\def\bbr{\mathbb{R}}
\def\CC{\mathcal{C}}
\def\CX{\mathcal{X}}
\def\calx{\mathcal{X}}
\def\calr{\mathcal{R}}
\def\CR{\mathcal{R}}

\def\polytopalX{\mathcal{C}_{\mathcal{X}}}
\def\polytopalXideal{\mathcal{C}_{\mathcal{X}^-}}

\def\umv{\circ}
\def\mv{\bullet}
\def\smv{\rule[.2ex]{1ex}{1ex}}

\def\iso{\buildrel \sim\over\to}

\def\mMod{\operatorname{\!-mod}\nolimits} 
 
\def\Mod{\operatorname{\!-Mod}\nolimits}  
\def\Nil{\operatorname{\!-nil}\nolimits} 
\def\QMod{\operatorname{\!-Qmod}\nolimits} 
\def\qMod{\operatorname{\!-qmod}\nolimits} 
\def\perf{\operatorname{\!-perf}\nolimits}

\begin{document}

\baselineskip = 15pt

\title{Cubist algebras} 

\author{Joseph Chuang}
\author{Will Turner}

\begin{abstract}
We construct algebras from rhombohedral tilings of Euclidean space obtained as projections of certain
cubical complexes. 
We show that these `Cubist algebras' satisfy strong homological properties, 
such as Koszulity and quasi-heredity, 
reflecting the combinatorics of the tilings. 
We construct derived equivalences between Cubist algebras associated to local mutations in tilings. 
We recover as a special case the Rhombal algebras of Michael Peach and make a precise connection to weight $2$ 
blocks of symmetric groups.
\end{abstract}

\maketitle
\tableofcontents

\section {Introduction}

\subsection{Homological algebras}
How many algebras are there which possess every strong homological property known to mathkind, 
yet are not semisimple ?

We define algebras $U_{\mathcal X}$, which generalise the Rhombal algebras of M. Peach, as well
as the Brauer tree algebra associated to an infinite line. 
We prove that these algebras are 
Koszul, symmetric, super-symmetric algebras, 
whose projective modules are of identical odd Loewy length $2r-1$, for some natural number $r$.
We reveal $r !$ different highest weight structures on $U_{\mathcal X}\mMod$.
Standard modules are Koszul, for any highest weight structure.
Given a highest weight structure, there
is a canonical choice of homogeneous cellular basis for $U_{\mathcal X}$.

The Koszul duals $V_{\mathcal X}$
of the $U_{\mathcal X}$ generalise the preprojective algebra associated to an infinite line.
They have global dimension $2r-2$. 
Thinking of $V_{\mathcal X}\mMod$ as the category of quasi-coherent sheaves on 
a chimeric noncommutative affine algebraic variety of dimension $2r-2$, 
we may define a category of sheaves on the corresponding projective variety, 
which has dimension $2r-3$,
and obeys Serre duality, with trivial canonical bundle.
We define $r !$ different highest weight structures on $V_{\mathcal X}\mMod$,
dual to those on $U_{\mathcal X}\mMod$.
Again, standard modules are all Koszul.
Given a highest weight structure, there
is a canonical choice of homogeneous cellular basis for $V_{\mathcal X}$.

The only finite-dimensional algebras which enjoy these potent combinations of properties 
are semisimple algebras. Our examples are therefore necessarily infinite dimensional.

The combinatorics of our algebras is governed by collections of cubes in $r$-dimensional space,
viewed from $(r-1)$-dimensional space. In homage to P. Picasso and G. Braque, who explored manifold possibilities
of this geometric attitude when $r=3$, we call them \emph{Cubist algebras}.

We expect there to be further examples of algebras ${\mathcal A}_\tau$ which satisfy many of the listed properties of the 
$U_{\mathcal X}$'s, and upon a suitable localisation, 
describe the Morita type of blocks of symmetric groups / Schur algebras.
The Koszul duals of the algebras ${\mathcal A}_\tau$, upon localisation, will describe the Koszul duals of
blocks of Schur algebras.

There are further similarities between the Cubist algebras $U_{\mathcal X}$, and blocks of symmetric groups.
Large collections of them are derived equivalent, the tilting bimodule complexes being obtained by 
composing two-term complexes. 
Symmetric group blocks of weight two are largely similar to certain $U_{\mathcal X}$'s, in case $r=3$.

Differences between the algebras $U_{\mathcal X}$, and ${\mathcal A}_\tau$ are soon visible.  
On $U_{\mathcal X}$, there are $r !$ possible highest weight structures.
However, ${\mathcal A}_\tau$ has 
only two alternative
orderings on its simple objects, corresponding to  the dominance ordering of partitions and its opposite.  
The combinatorics surrounding ${\mathcal A}_\tau$ is complicated, and mysterious,
whilst numerical properties of $U_{\mathcal X}$ are elegant, and transparent.

Relations between matrices of composition multiplicities for $U_{\mathcal X}$ and $V_{\mathcal X}$ 
release beautiful combinatorial formulae, associated to Cubist views of Euclidean space.

Our proofs begin with the combinatorics of Cubist diagrams. 
Pursuing these combinatorics allows us to prove the existence of 
highest weight structure on $V_{\mathcal X}\mMod$. 
Standard modules for $V_{\mathcal X}$ have linear projective resolutions, implying that
$V_{\mathcal X}$ is standard Koszul, in the sense of Agoston, Dlab and Lukacs.
This allows us to deduce Koszulity, and the existence of highest
weight structures on $U_{\mathcal X}\mMod$.
The existence of cellular structures on $U_{\mathcal X}, V_{\mathcal X}$ are then apparent. 
Eventually, symmetry of $U_{\mathcal X}$ is won.

We proceed to prove that the derived categories of 
$U_{\mathcal X}, U_{\mathcal X'}$ are equivalent, whenever ${\mathcal X'}$ is obtained from ${\mathcal X}$ 
by a local flip. This equivalence of categories allows us identify portions of 
symmetric group blocks of weight two with portions of certain Cubist algebras, in case $r=3$, 
following a path down from the Rouquier block.

\subsection{The case $r=3$:  Peach's rhombal algebras.}
\label{subsection:introrhombal}

\begin{figure}[h]
\[
\includegraphics[height=3in]{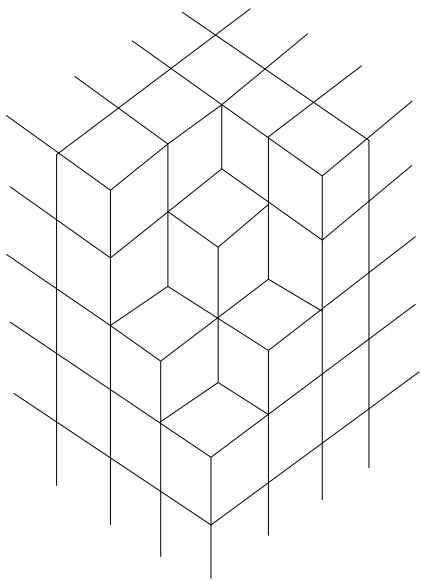}
\]
\caption{Part of a tiling $\Gamma$.}
\label{ExampleTiling}
\end{figure}
\FloatBarrier

Let $\Gamma$ be a tiling of the plane by congruent rhombi affixed to a hexagonal grid (see Figure~\ref{ExampleTiling}). The rhombal algebra $U_\Gamma$ associated to $\Gamma$ is defined to be the path algebra of the quiver obtained from $\Gamma$ by replacing every edge by two arrows in opposite directions, modulo the following quadratic relations:
\begin{itemize}
\item {\bf Two rhombuses relation.} Any path of length two not bordering a single rhombus is zero.
\item {\bf Mirror relation.} The sum of the two paths of length two from one vertex of a rhombus to the opposite vertex is zero.
\item {\bf Star relation.} At each vertex $x$, there are six possible paths of length two from $x$ to itself, which we label as follows:
\begin{figure}[h]
\[
\xymatrix@M=0.1EX@R=2.8ex{
&&  \ar@/^0.3pc/[dddd]^{f}  && \\
&& && \\
 \ar@/^0.3pc/[ddrr]^{a}  && &&   \ar@/^0.3pc/[ddll]^{b}  \\
&& && \\
&& \bullet \ar@{-}@/^0.3pc/[uuuu] \ar@{-}@/^0.3pc/[uurr] \ar@{-}@/^0.3pc/[uull] \ar@{-}@/^0.3pc/[dddd] \ar@{-}@/^0.3pc/[ddrr] \ar@{-}@/^0.3pc/[ddll]   && \\
&& && \\
 \ar@/^0.3pc/[uurr]^{e} &&  &&  \ar@/^0.3pc/[uull]^{d}   \\
&& && \\
&&  \ar@/^0.3pc/[uuuu]^{c} && 
}
\]
\end{figure}
\FloatBarrier 
\noindent We impose the relation $$a+d=b+e=c+f,$$ where we read any of  $a,b,c,d,e,f$ as $0$ if it is not present at the vertex $x$.
\end{itemize} 
The choice of signs in this presentation is due to Turner \cite[Definition 15]{Turner-OSFOA}.  The relationship with Peach's original presentation
\cite{Peach} is described in Remark~\ref{Peachsigns} below.

We now list the good properties of $U_\Gamma$, together with references to results (proved in the generality of Cubist algebras) in the main body of the paper. Some of the properties were originally established by Peach; in those cases we also provide the reference in Peach's thesis.
\begin{itemize}
\item $U_\Gamma$ is a locally finite dimensional graded algebra; each indecomposable projective module has radical length five 
(\cite[Corollary 2.4.2]{Peach}, Corollary~\ref{gldim})
\item $U_\Gamma$ is a symmetric algebra, i.e. it possesses an invariant inner product.
(\cite[Theorem 2.5.1]{Peach}, Theorem~\ref{symmetric})
\item $U_\Gamma$ is Koszul (Theorem~\ref{UKoszul})
\item $U_\Gamma$ is quasihereditary (Theorem~\ref{Uqh})
\end{itemize}
Let $V_\Gamma$ be the quadratic dual of $U_\Gamma$. Then
\begin{itemize}
\item $V_\Gamma$ is Koszul (Theorem~\ref{VKoszul})
\item $V_\Gamma$ is quasihereditary (Theorem~\ref{Vhwc})
\item $V_\Gamma$ has global dimension 4 (Corollary~\ref{gldim})
\end{itemize}
The quasihereditary structure on $U_\Gamma$ comes with a partial order $\preceq$ on simple modules and therefore on the vertices of $\Gamma$; the partial order for $V_\Gamma$ is the opposite order.   The partial order
$\preceq$ may be described as follows. There exists a bijection $\lambda$ from the set of vertices of $\Gamma$ to the set of rhombi in $\Gamma$, given by Figure~\ref{vertexfacetfigure}. Then $y\preceq x$ if $y$ is a vertex in the rhombus $\lambda x$, and this relation generates the partial order.
 The standard modules are pictured in Figure~\ref{StandardModuleFigures}. Note that the standard module for $U_\Gamma$ has four composition factors, while the standard module for $V_\Gamma$ has infinitely many.
\begin{figure}[h]
\[
\xymatrix@M=0EX@R=2.8ex{
&  \ar@{-}[dl] \ar@{-}[dd] &&\\
\ar@{-}[dd] & & \ar@{-}[dd] \ar@{-}[dr] & \\
& \bullet \ar@{-}[dl] && \ar@{-}[dd] \\
&& \ar@{-}[dr]& \\
& \ar@{-}[dl] \ar@{-}[dr] && \bullet \\
\ar@{-}[dr] && \ar@{-}[dl] & \\
& \bullet && \\}
\]
\caption{Bijection between vertices and rhombi.}
\label{vertexfacetfigure}
\end{figure}
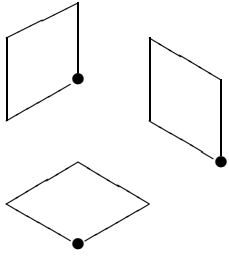
\FloatBarrier
\begin{figure}[h]
\[
\includegraphics[height=3in]{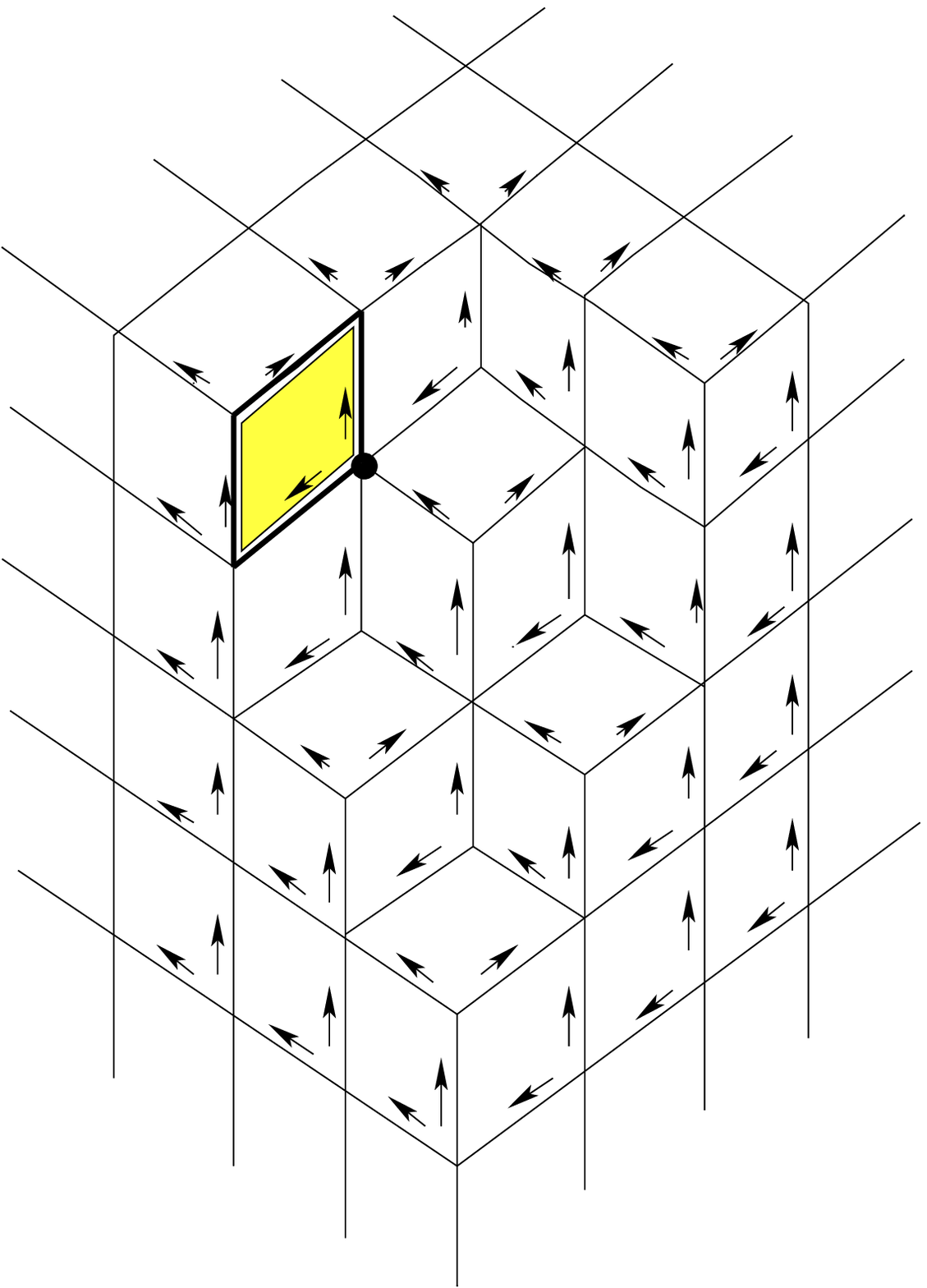}
\hspace{1in}
\includegraphics[height=3in]{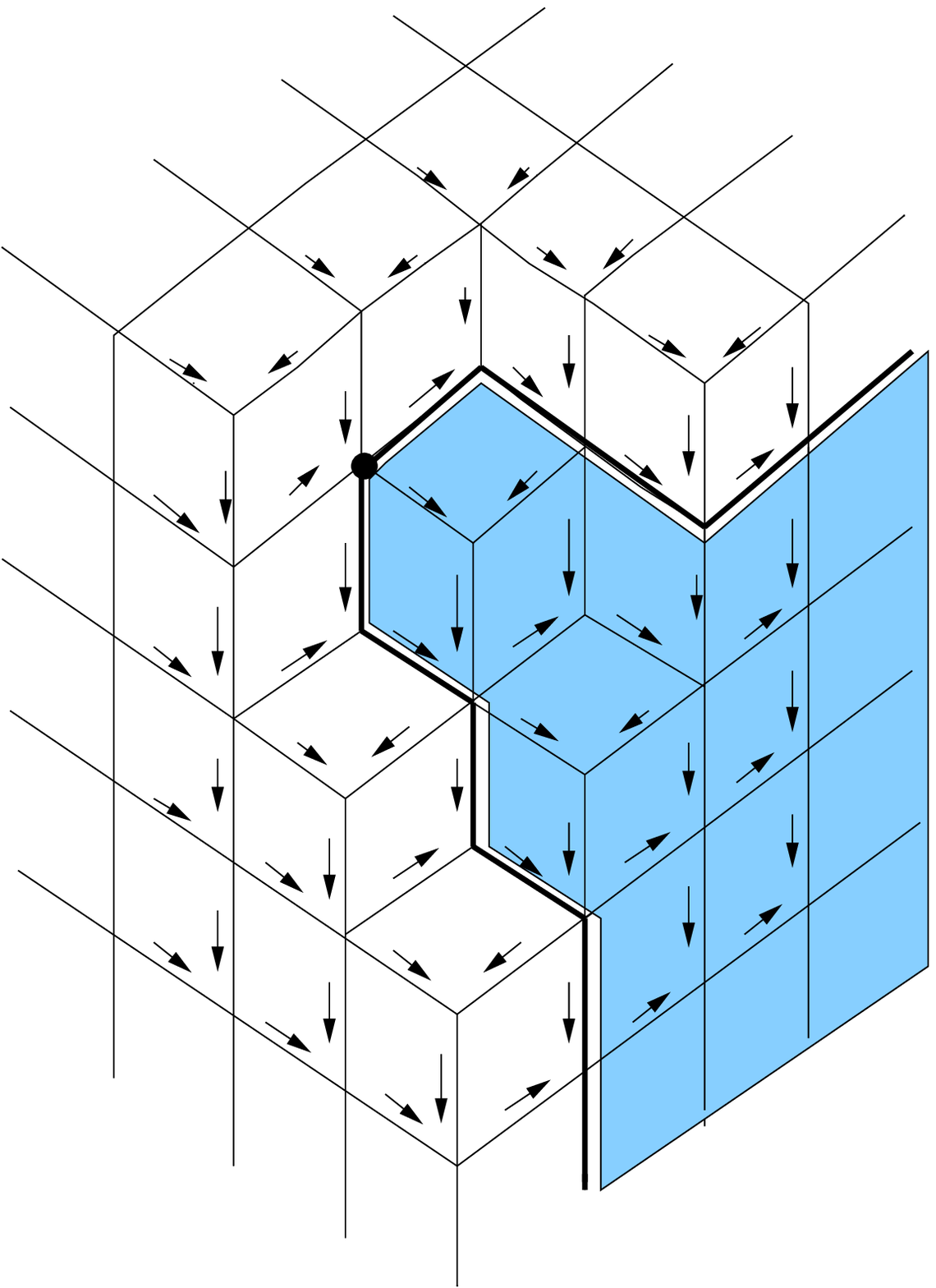}
\]
\caption{Standard modules in $U_\Gamma$ and $V_\Gamma.$}
\label{StandardModuleFigures}
\end{figure}
\FloatBarrier

Rotating Figure~\ref{vertexfacetfigure} by a multiple of $\pi/3$, we obtain $6$ different partial orders on the vertices
of $\Gamma$ and therefore $6$ different quasi-hereditary structures on $U_\Gamma$ and $V_\Gamma$.

The quasi-hereditary and Koszul properties give rise to curious formulae for the inverse of a 
matrix recording distances in $\Gamma$.  
Indeed, 
given $x,y \in \Gamma$, let $d(x,y)$ be the length of the shortest path from $x$ to $y$ along edges in $\Gamma$.
Let $\operatorname{Dist}(q)$ be the infinite matrix with rows and columns indexed by the vertices of $\Gamma$,  
whose $(x,y)$-entry is $q^{d(x,y)}$.
Then
$$\operatorname{Dist}(q)\operatorname{Loc}(-q)=(1-q^2)^2 I,$$ 
where $I$ is the identity matrix, 
and $\operatorname{Loc}(q)$ is a matrix with a number of equivalent descriptions, each of which
describe certain local configurations in $\Gamma$. 
The
$(x,y)$-entry of $\operatorname{Loc}(q)$ can be written $\sum q^{d(x,z)+d(z,y)}$, 
where the sum is over all vertices $z$ in $\Gamma$ such that $x,y\in\lambda z$ (Corollary~\ref{preformula}). 

An alternative formula for the $(x,y)$-entry  of $\operatorname{Loc}(q)$, 
which does not depend on a choice of highest weight structure is
$\sum_{z \in I_\Gamma(x) \cap I_\Gamma(y)} q^2[3-d(z,x)-d(z,y)]_q$, where $I_\Gamma(x)$ is a set 
describing a local configuration about $x$ (Proposition~\ref{alternativeformula}), and $[n]_q=(q^n-q^{-n})/(q-q^{-1})$.
Let $H$ be a tiled hexagon, in which the six internal lines/rhombi are placed in correspondence 
with the six external vertices, as in Figure~\ref{Hex}.

\begin{figure}[h]
\[
\includegraphics[height=1in]{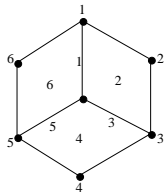}
\]
\caption{The hexagon $H.$}
\label{Hex}
\end{figure}
\FloatBarrier

Let $H(x)$ be the hexagon $H$, centered at $x$,
and dilated so its external vertices coincide with the six vertices of $\Gamma$ closest to $x$.
Note that the rhombic tiles in $H(x)$ cannot necessarily be thought of as tiles in $\Gamma$. 
The set $I_\Gamma(x)$ consists of the vertex $x$, as well as those vertices $1,...,6$ in $H(x)$ 
corresponding to 
lines/rhombi in $H(x)$ which can also be thought of as lines/rhombi in $\Gamma$.

\subsection{Acknowledgements}
We thank Volodymyr Mazorchuk for helpful discussions on
standard Koszul algebras, Claus Ringel for his kind encouragement and 
Michael Peach for volunteering his pictures of rhombus tilings.  Joe Chuang thanks the
EPSRC for its support (grant GR/T00924/01).

\section{Cubist combinatorics}

We define and study certain subsets of integer lattices that correspond to tilings of Euclidean space by rhombohedra.

\subsection{Cubist subsets}

Given $x,y\in\mathbb{R}^r$, we write $x\leq y$ if $y-x\in\mathbb{R}^r_{\geq 0}$.
This defines a partial order on $\mathbb{R}^r$.
We denote by $\epsilon_{1},\ldots,\epsilon_{r}$
the standard basis of $\mathbb{R}^{r}$.
For $x\in\mathbb{R}^r$ and $\zeta\in\mathbb{R}$, let $x[\zeta] = x+\zeta(\epsilon_1+\cdots+\epsilon_r)\in \mathbb{R}^r$.

\begin{definition} 
A subset ${\mathcal X} \subset \mathbb{Z}^r$ is \emph{Cubist}, if  
${\mathcal X} = {\mathcal X}^- \backslash {\mathcal X}^-[-1]$, where ${\mathcal X}^-$ is a nonempty proper ideal of $\mathbb{Z}^r$ (with respect to the partial order $\leq$).
\end{definition}

Note that ${\mathcal X}^-$ is uniquely determined by $\mathcal{X}$; it is the ideal of $\mathbb{Z}^r$ generated by $\mathcal{X}$. An ideal of $\mathbb{Z}^2$ is an infinite version of (the Ferrers diagram of) a partition, and of $\mathbb{Z}^3$ is an infinite version of a plane partition (see, e.g., \cite[p.371]{Stanley}). So Cubist subsets may be regarded as higher-dimensional generalisations of infinite partitions.
 
We have the following easy inductive characterisation of ideals in $\mathbb{Z}^r$.

\begin{lemma} \label{divide}
${\mathcal X}^-$ is an ideal of $\mathbb{Z}^r$ if, and only if, 
\begin{enumerate}
\item ${\mathcal X}^-_i = 
\{ x \in \mathbb{Z}^{r-1} \mid  (x , i) \in {\mathcal X} \}$
is an ideal of $\mathbb{Z}^{r-1}$ for $i \in \mathbb{Z}$, and 
\item ${\mathcal X}^- _{i+1} \subset {\mathcal X}^- _{i}$, 
for all $i \in \mathbb{Z}$. $\Box$
\end{enumerate}
\end{lemma}
This leads to an useful inductive description of Cubist subsets.
\begin{lemma} \label{ind}
If ${\mathcal X} \subset \mathbb{Z}^r$ is Cubist, then ${\mathcal X}_i = 
\{ x \in \mathbb{Z}^{r-1} \mid  (x , i),(x, i-1)  \in {\mathcal X} \}$
is either a Cubist subset of $\mathbb{Z}^{r-1}$ or empty, for $i \in \mathbb{Z}$. $\Box$
\end{lemma}

\begin{proof}
By Lemma~\ref{divide},
$\mathcal{X}_{i}^{-}=\{x\in\mathbb{Z}^{r-1}\mid (x,i)\in\mathcal{X}^{-}\}$ is an ideal in
$\mathbb{Z}^{r-1}$, and
\begin{align*}
\mathcal{X}_{i}^{-}\setminus\mathcal{X}_{i}^{-}[-1] 
 &= \{x\in\mathbb{Z}^{r-1} \mid   (x,i)\in\mathcal{X}^{-},\:(x[1],i)\notin\mathcal{X}^{-}\} \\
& = \{x\in\mathbb{Z}^{r-1} \mid   (x,i),\,\left(x,i-1\right)  \in\mathcal{X}^{-}, 
 \;(x[1],i),\,(x[1],i+1)\notin\mathcal{X}^{-}\}\\
&= \{x\in\mathbb{Z}^{r-1} \mid   (x,i)\in\mathcal{X},\:(x,i-1)\in\mathcal{X}\}.
\end{align*}
\end{proof}

The set of Cubist subsets of $\bbz^r$ is invariant under translations and under the action of the symmetric group $\Sigma_r$ permuting coordinates, as well as under the involution $x\mapsto -x$. The latter is true because of the following easy lemma.

\begin{lemma}
A subset ${\mathcal X}$ of $\mathbb{Z}^r$ is Cubist if, and only if, 
${\mathcal X} = {\mathcal X}^+ \backslash {\mathcal X}^+[1]$, where ${\mathcal X}^+$ is a nonempty proper 
coideal of $\mathbb{Z}^r$.  
$\Box$
\end{lemma}

\subsection{Rhombohedral tilings}
\label{subsection:tilings}
In case $r\leq 3$, it is easy to understand Cubist subsets: one only needs
to draw a picture.  In higher dimensions, Cubist combinatorics are not so
easy, but a topological perspective can be helpful.  Here we associate
to any Cubist subset $\CX$ of $\bbz^r$ a polytopal complex $\CC_{\CX}$
of dimension $r-1$ inside $\bbr^r$, whose faces are all cubes. This
cubical complex projects homeomorphically onto a hyperplane, inducing a tiling of Euclidean $(r-1)$-space by rhombohedra.

Let us spell this out it detail.
We define a {\em cube} in $\mathbb{Z}^{r}$ to be a subset of the form
$$
C = x+\left\{  \sum_{j\in S}a_{j}\epsilon_{j},\quad a_{j}=0,1\right\},
$$
where $x\in\mathbb{Z}^{r}$ and $S$ is a subset of $\{1,\ldots,r\}$. We  say
$C$ is a {\em $d$-cube} if $S$ has size $d$. These cubes define a polytopal decomposition of $\mathbb{R}^r$ in which the $d$-dimension faces are
the convex hulls of the $d$-cubes in $\mathbb{Z}^r$. For $x=(x_1,\ldots,x_r)\in\bbr^r$, the unique face of smallest dimension containing $x$
is $\left\{y\in\bbr^r\mid \lfloor x \rfloor \leq y \leq
\lceil x \rceil \right\}$, where
$\lfloor x \rfloor = \left(\lfloor x_1 \rfloor,\ldots,\lfloor x_r \rfloor\right)$ and
$\lceil x \rceil = \left(\lceil x_1 \rceil,\ldots,\lceil x_r \rceil\right)$.

Fix a Cubist subset $\calx$ of $\bbz^r$.
\begin{definition} We define
$\polytopalX$ to be the smallest subcomplex of $\mathbb{R}^r$ containing $\mathcal{X}$. Equivalently $\polytopalX$ is the subcomplex consisting of all
faces which are convex hulls of cubes contained in $\calx$.
\end{definition}

 For all $x\in\mathbb{R}^r$, we have $x\in\polytopalX$ if and only if 
$\lfloor x \rfloor\in\calx$ and $\lceil x \rceil\in\calx$, i.e., if and only if
$\lceil x \rceil\in\calx^-$ and $\lfloor x \rfloor[1]\notin\calx^-$.
We define $\polytopalXideal$ much as we did $\polytopalX$:
as the smallest
subcomplex of $\mathbb{Z}^r$ containing $\mathcal{X}^-$.
We note that if $x$ is in the ideal of $\bbr^r$ generated by $\calx^-$, then $\lceil x \rceil$ is as well. It follows that this ideal is equal to
$\polytopalXideal$.


Let $L$ be an affine line in $\bbr^r$ parallel to the vector
 $\epsilon_1+\ldots+\epsilon_r$. Because $\calx^-$ is a nonempty proper subset
of $\bbz^r$, the intersection of $L$ with $\polytopalXideal$ is a half-line: there exists $x_L\in L$ such that for all $y\in L$, $y\in
\polytopalXideal$ if and only if $y\leq x_L$. Clearly, $x_L$ is on the boundary of $\polytopalXideal$. Conversely any point on the boundary is
of the form $x_L$ for some $L$. Indeed, if
$x\in\polytopalXideal$ and $x\neq x_L$, where $L$ is the affine line
parallel to $\epsilon_1+\ldots+\epsilon_r$ containing $x$, then
$\left\{y\in\bbr^r\mid y<x_L\right\}$ is a neighborhood of $x$ contained in 
$\polytopalXideal$. 

We claim that the boundary of $\polytopalXideal$ is
$\polytopalX$. Suppose that $x\in\polytopalX$ is not
on the boundary of $\polytopalXideal$. By the
discussion above, $x[\varepsilon]\in\polytopalXideal$ for some
$\varepsilon>0$. Then $\lfloor x \rfloor \in \calx$ and 
$\lfloor x \rfloor[1] \leq \lceil x[\varepsilon] \rceil \in \calx^-$,
contradicting the assumption that $\calx$ is Cubist. On the other hand
if $x$ is on the boundary of $\polytopalXideal$, then
$\lceil x \rceil \in \calx^-$ and
 $\lfloor x \rfloor [1]\notin\calx^-$, the latter because
 $x[\varepsilon]<\lfloor x \rfloor[1]$ for some $\varepsilon>0$.

\begin{proposition}\label{homeomorphism}
Let $\mathbb{R}^r_0$ be the hyperplane $\{ (x_1,\ldots,x_r) | \sum x_i = 0 \}$ of $\mathbb{R}^r$.
Let $\pi$ be the orthogonal projection of $\bbr^r$ onto $\mathbb{R}^r_0$. Then the restriction of $\pi$
to $\polytopalX$ is a homeomorphism onto $\mathbb{R}^r_0$.
\end{proposition}
\begin{proof}
In the discussion above we proved that $\polytopalX$ is the boundary of $\polytopalXideal$ and thus meets each fiber of $\pi$ in exactly one point.
So the restriction of $\pi$ to $\polytopalX$ is a continuous bijection. In fact it is a homeomorphism, since the restriction to any cell of
$\polytopalX$ is a homeomorphism onto a closed subset of $\mathbb{R}^r_0$.
\end{proof}

The homeomorphism between $\polytopalX$ and $\mathbb{R}^r_0$ induces a rhombotopal decomposition of $\mathbb{R}^r_0$, which
Linde, Moore and Nordahl call `a configuration of the  $(r-1)$-dimensional tiling' \cite[\S2]{LMN}. When $r=3$ we obtain the rhombus tilings in $\mathbb{R}^2$ considered in
\S~\ref{subsection:introrhombal}.

From Proposition~\ref{homeomorphism} we deduce some useful properties of Cubist subsets.

\begin{definition}
A subset ${\mathcal S}$ of $\mathbb{Z}^r$ is \emph{connected}, 
if for any $x,y \in {\mathcal S}$, there exists a sequence
$x= x^0, x^1, \ldots, x^l =y$, such that $x^m \in {\mathcal S}$ and for each $m \geq 1$, we have 
$x^{m+1} = x^m \pm \epsilon_n$, for some $n=n(m)$.

Let ${\mathcal S}$ be a connected subset of $\mathbb{Z}^r$. 
For $x,y \in {\mathcal S}$, 
let $d_{\mathcal S}(x,y)$ be the smallest number $l$ for which there exists a sequence
$x= x^0, x^1, \ldots, x^l =y$, such that $x^m \in {\mathcal S}$ and for each $m \geq 1$, we have 
$x^{m+1} = x^m \pm \epsilon_n$, for some $n=n(m)$.

We write $d(x,y) = d_{\mathbb{Z}^r}(x,y).$
\end{definition}

\begin{corollary} \label{def}
Let ${\mathcal X}$ be a Cubist subset of $\mathbb{Z}^r$.
Then  
\begin{enumerate}
\item
${\mathcal X}$ is a connected subset of $\mathbb{Z}^r$,
 such that 
$\mathbb{Z}^r = \amalg_{\mathbb{Z}} {\mathcal X}[m]$.
\item No $r$-cube is contained in $\calx$.
\item For any $x\in\calx$ the intersection of the $(r-1)$-cubes in $\calx$ containing $x$
is $\{x\}$.
\end{enumerate}
\end{corollary}

Given a polytopal complex $\mathcal{C}$ which is everywhere locally homeomorphic to $\mathbb{R}^w$, for some $w$,
one can form its dual complex ${\mathcal C'}$. The dual ${\mathcal C'}$ 
is a polytopal complex homeomorphic to ${\mathcal C}$,
whose $d$-dimensional faces are in bijection with the $(w-d)$-dimensional faces of ${\mathcal C}$.
The poset of faces of ${\mathcal C'}$ is opposite to the poset of faces of ${\mathcal C}$.

\begin{definition}
Let $x$ be an element of a cubist subset ${\mathcal X}$ of $\mathbb{Z}^r$. 
We define ${\mathcal P}_x$ to be the face of $\mathcal{C'_X}$
which corresponds to the vertex $x \in \mathcal{C_X}$. 
\end{definition}

Thus, ${\mathcal P}_x$ is an $r-1$-dimensional polytope, which describes the configuration of ${\mathcal X}$ about $x$.

In case $r=3$, the polytope ${\mathcal P}_x$ can be a triangle, a square, a pentagon, or a hexagon.
For general $r$,
the number of $d_2$-dimensional faces containing any given $d_1$-dimensional face of ${\mathcal P}_x$
is $\binom{r-1-d_1}{d_2-d_1}$, whenever $d_1 \leq d_2$.

\subsection{Vertex-facet bijection}

While we have defined $d$-cubes in $\mathbb{Z}^r$ for all $d$, Proposition~\ref{homeomorphism} shows that
$(r-1)$-cubes are particularly relevant.  We call them {\em facets}. Let
$\mathcal{F}$ be the set of facets in $\mathbb{Z}^{r}$ and $\mathcal{F}
_{\mathcal{X}}$ the set of facets contained in a Cubist subset $\mathcal{X}$.
Any $F\in\mathcal{F}$ can be written as
\[
F=x+F_{i},
\]
where
\[
F_{i}=\left\{  \sum_{j<i}a_{j}\epsilon_{j}-\sum_{j>i}a_{j}\epsilon_{j},\quad
a_{j}=0,1\right\},
\]
for a unique choice of $x\in\mathbb{Z}^{r}$ and $i\in\{1,\ldots,r\}$.

\begin{proposition}
\label{vertex-facet bijection}For each $x\in\mathcal{X}$, there is a unique
$i$ such that $x+F_{i}\subseteq\mathcal{X}$. We have
$$i=\max\{j\mid  x+\epsilon_{1}+\ldots+\epsilon_{j-1}\in\mathcal{X\}},$$
and putting $\lambda
x=x+F_{i}$, we obtain a bijection
\[
\lambda=\lambda_{\mathcal{X}}:\mathcal{X}\longrightarrow\mathcal{F}%
_{\mathcal{X}}.
\]

\end{proposition}
\begin{proof}

Let $i=\max\{j\mid  x+\epsilon_{1}+\ldots+\epsilon_{j-1}\in\mathcal{X\}}$.
Then $x+F_{i}\subseteq\mathcal{X}$. Indeed if $y\in x+F_{i}$, then $y\leq
x+\epsilon_{1}+\ldots+\epsilon_{i-1}\in\mathcal{\mathcal{X}}^{-}$ and
$y[1]\geq x+\epsilon_{1}+\ldots+\epsilon_{i}\notin\mathcal{X}^{-}$. On the
other hand, if $j>i$ (resp. $j<i$) then $x+\epsilon_{1}+\ldots+\epsilon_{i}$
(resp. $x-\epsilon_{i}-\ldots-\epsilon_{r}$) is in $x+F_{j}$ but not in
$\mathcal{X}$. Once the map $\lambda$ is defined it is clearly a bijection.
\end{proof}

\begin{definition}
For $x\in\mathcal{X}$, such that 
$\lambda x=x+F_{i}$, let $C_{i}=\mathbb{Z}_{\leq
0}^{i-1}\times\mathbb{Z\times\mathbb{Z}}_{\geq0}^{r-i}$.
Let $\mu x=x+C_{i}$.
\end{definition}

The set $\lambda x$ can be thought of as a cube emanating from $x$.
From this perspective, $\mu x$ can be seen as a convex polyhedral cone in $\mathbb{Z}^r$ emanating from $x$, 
in all directions opposite to $\lambda x$.

\subsection{Basic examples}

Let $1\leq j \leq r$. Then the lattice
 $\mathcal{H}_j=\left\{(x_1,\ldots,x_n)\mid x_j=0\right\}$
is a Cubist subset of $\bbz^r$. The corresponding tiling of 
$\bbr^r_0$ is a regular tiling composed of translates of a fixed $(r-1)$-dimensional rhombohedron.
A Cubist subset $\calx\subset \bbz^r$ may look locally like one of these `flat' Cubist subsets $H_j$.

\begin{definition} Let ${\mathcal X} \subset \mathbb{Z}^r$ be Cubist. The subset 
${\mathcal X}_{flat}$ of \emph{flat elements}
is defined to be the subset of elements $x \in {\mathcal X}$ 
for which there exists $i = i(x)$ such that $x + \epsilon_i, x - \epsilon_i \notin {\mathcal X}$.
The subset ${\mathcal X}_{crooked}$ of \emph{crooked elements} is the complement 
${\mathcal X} \backslash {\mathcal X}_{flat}$.
\end{definition}

\begin{lemma}
Let ${\mathcal X} \subset \mathbb{Z}^r$ be a Cubist subset.

1. If $x \in {\mathcal X}$ is crooked, 
then for all $j$, 
either $x+\epsilon_j \in {\mathcal X}$, or $x-\epsilon_j \in {\mathcal X}$. 

2. If $x \in {\mathcal X}$ is flat, then for all $j \neq i(x)$, we have 
$x + \epsilon_j, x - \epsilon_j \in {\mathcal X}$.
$\Box$ 
\end{lemma}
\begin{proof}

We proceed by induction on $r$. The lemma is obvious for small $r$. 
Suppose the lemma holds in dimensions $<r$, and ${\mathcal X}$ is a cubist subset of $\mathbb{Z}^r$.
We may assume that $x=0$. If $x \in {\mathcal X}_0$, then $x-\epsilon_r \in {\mathcal X}$, 
and the statement of the lemma holds
by the inductive hypothesis. Similarly, if $x \in {\mathcal X}_1 -\epsilon_r$, then $x+\epsilon_r \in {\mathcal X}$,
and the statement of the lemma holds by the inductive hypothesis.
Otherwise, $x \pm \epsilon_r \notin {\mathcal X}$, in which case $x[-1], x[1] \notin {\mathcal X}$. 
Therefore, $x + \sum_{j<r} \epsilon_j, x - \sum_{j<r} \epsilon_j \in {\mathcal X}$,
and $x \pm \epsilon_j \in {\mathcal X}$,
for $1 \leq j \leq r-1$.
\end{proof}

Another example is the `corner configuration'.

\begin{definition}
The \emph{Corner configuration} is the Cubist subset 
\begin{align*}
{\mathcal X}_{CC} 
& = \mathbb{Z}^r_{\leq 0} \backslash \mathbb{Z}^r_{\leq 0}[-1] \\
& =\left\{(x_1,\ldots,x_n)\in\bbz_{\leq 0}^r\mid x_i=0 \textrm{ for some } i
\right\}
\end{align*}
of $\mathbb{Z}^r$.
\end{definition}

\begin{lemma} \label{corner approximation}
Any Cubist subset ${\mathcal X}$ can be approximated in an arbitrarily large finite region, 
by removing a finite number of $r$-cubes from the Corner configuration.

Precisely, given $x \in {\mathcal X}$, and $N \geq 0$,
there exists $z \in \mathbb{Z}^r$, and a Cubist set ${\mathcal X}(x, N) \subset \mathbb{Z}^r$,
such that ${\mathcal X}(x,N)^-$ is obtained from $(z + {\mathcal X}_{CC})^-$ by removing a finite number of elements,
and 
$$y \in {\mathcal X} \Leftrightarrow y \in {\mathcal X}(x,N),$$
for all $y \in \mathbb{Z}^r$ such that $d(y,x) \leq N$.
\end{lemma}
\begin{proof}

Let ${\mathcal X}(x,N)^-$ be the ideal of $(\mathbb{Z}^r, \leq)$ generated by 
$$({\mathcal X} \cap (x + [-N,N]^r)) \cup \{ x[N] - 2N \epsilon_i \mid  i=1,\ldots,r \}.$$ 
Let $z = x[N]$. The statement of the lemma is satisfied for this pair $(z, {\mathcal X}(x,N)^-)$.
\end{proof}

\section{Some algebraic preliminaries}

Let $k$ be a field. We shall be working with associative $k$-algebras $A$
graded over the integers. So $A=\oplus_{i\in\mathbb{Z}}A_{i}$ and $A_{i}
A_{j}\subset A_{i+j}$. While not assuming the existence of a unit, we require
$A$ to be equipped with a set of mutually orthogonal idempotents
$\{e_{s}\mid  s\in\mathcal{S}\}\subset A_{0}$ such that $A=\oplus
_{s,s'\in\mathcal{S}}e_{s}Ae_{s'}$. It will be useful to allow some of the idempotents to be zero. Unless stated otherwise, all $A$-modules $M$
are assumed to be \emph{graded} left modules, so that $M=\oplus_{i\in
\mathbb{Z}}M_{i}$ and $A_{i}M_{j}\subset M_{i+j}$, and to be
quasi-unital, i.e.,
$M=\oplus_{s\in S}e_{s}M$.  Given $n\in\mathbb{Z}$, we let $M\langle n\rangle$
be the $A$-module obtained by shifting the grading by $n$, so that
$M\langle n\rangle_{i}=M\langle n-i\rangle$.

Suppose that $A_{i}=0$ for all but finitely many negative integers $i$, and that $e_{s}A_{i}e_{s'}$ is finite
dimensional for all $s,s'\in\mathcal{S}$ and $i\in\mathbb{Z}$. 
We define
the graded Cartan matrix $C_A(q)$ to be the matrix with rows and columns labelled by 
$\mathcal{S}$ and entries
\[
C_{A}(q)_{s,s'}=\sum_{i\in\mathbb{Z}}\left(\dim e_{s}A_{i}e_{s'}\right)q^i
\]
in the ring of Laurent power series in an indeterminate $q$.
If $e_s=0$ then the entries in the row and column labelled by $s$ are zero. So we will often regard $C_A(q)$ 
as an $\calr\times\calr$ matrix,
where $\calr=\left\{s\in S\mid e_s\neq 0\right\}$.

Now suppose that $A$ is positively graded, i.e. $A_{i}=0$ for $i<0$, and that
$\{e_{s}\mid  s\in\mathcal{R}\}$ is a basis for $A_0$.  Let us also impose the
finiteness condition $\dim e_s A_i < \infty$ for all $s\in\CR$ and $i\in\bbz$.
Let $A\Mod$ be the
category of all graded $A$-modules, where the space of morphisms
between graded modules $M$ and $N$, which we denote $\operatorname*{Hom}%
\nolimits_{A}(M,N)$, consists of $A$-module homomophisms preserving degree.
We denote by $A\mMod$ the full subcategory consisting of modules $M$ such that
$\dim e_s M_i < \infty$ for all $s\in\CR$ and $i\in\bbz$, and that $M_i=0$ for 
$i<<0$.
 
Then $Ae_{s}$ is a projective $A$-module (=projective indecomposable object in $A\mMod$) 
for each $s\in\calr$, 
and every projective indecomposable $A$-module is isomorphic to 
$Ae_{s}\langle n\rangle$, for a unique $s\in\mathcal{R}$ and $n\in\mathbb{Z}$. 
Similarly, every simple $A$-module is isomorphic to $L_A(s)\langle n\rangle$ for a unique $s\in\mathcal{R}$ and $n\in\mathbb{Z}$, where $L_A(s)=Ae_{s}/A_{> 0}e_{s}$.
The category $A\mMod$ contains enough projective objects.

Fix a partial order $\preceq$ on $\calr$. For each $s\in\calr$, the standard module $\Delta_A(s)=\frac{Ae_s}{\sum_{t\succ s}Ae_tAe_s}$ is the largest 
quotient of $Ae_s$ which does not contain $L(t)\langle n \rangle$ as a composition factor for $t\succ s$.
We define the graded decomposition matrix $D_A(q)$ of $A$ to be the $\calr\times\calr$ matrix with entries
$$D_A(q)_{st}=\sum_{i\in\bbz}\left(\dim e_t\Delta_A(s)_i\right)q^i.$$
If $A$ has an antiautomorphism fixing each $e_s$, then 
$D_A(q)$=$D_{A^{\operatorname{op}}}(q)$, where $A^{\operatorname{op}}$ is the 
opposite algebra.

 We say that $A$ is a graded quasi-hereditary algebra, or that $A\mMod$ is a
 graded highest weight category, if for all $s\in\calr$,
 \begin{itemize}
 \item
 $D_A(q)_{st}=0$ for all $t\npreceq s$,
 \item
 $ker(Ae_s \twoheadrightarrow \Delta_A(s))$  has a filtration in which each section is isomorphic
 to $\Delta_A(t)\langle n \rangle$ for some $t\succ s$ and $n\in\bbz$.
\end{itemize}

This differs from the original notion of quasiheredity introduced by Cline, Parshall and Scott
in that $A$ is allowed to be infinite dimensional; in particular $A$ may have infinite global 
dimension.  We are also  using a slightly different notion of (graded) highest weight category than that introduced by
Cline, Parshall and Scott. We filter projective objects by standard modules, rather than
injectives by costandards. Furthermore, we do not assume the finite interval property holds with respect
to our partial order $\preceq$.

If $A\mMod$ is a highest weight category, then according to
\cite[Theorem 3.1.11]{MR961165}
we have Brauer-Humphreys reciprocity:
$$C_A(q) = D_A(q)^T D_{A^{op}}(q).$$ 
We will need a version of Rickard's Morita theorem for derived categories \cite{Rickard-MTDC}
adapted to our graded algebras (see, e.g., \cite[\S 2]{Drozd-Mazorchuk-KDEASM}). Let us assume that $Ae_{s}$ is
finite dimensional for each $s\in\mathcal{S}$. Let $\{\Gamma_{t}\mid  t\in\mathcal{T}\}$ be a collection of
bounded complexes of projective $A$-modules.  Denote by $D^{b}(A\mMod)$
the derived category of bounded complexes in $A\mMod$.
Our formulation is rather clumsy; for the general theory it
is better to think in terms of functor categories.
 
\begin{theorem}[Rickard]
\label{Rickard} Suppose that
\begin{itemize}
\item for $t,t'\in\mathcal{T}$ and $m,n\in\mathbb{Z}$ with $m\neq0$,%
\[
\operatorname*{Hom}\nolimits_{D^{b}(A\mMod)}(\Gamma_{t}\langle
n\rangle,\Gamma_{t'}[m])=0.
\]

\item The triangulated subcategory of $D^{b}(A\mMod)$ generated
by all summands of $\Gamma_{t}\langle n\rangle$, $t\in\mathcal{T}$,
$n\in\mathbb{Z}$, contains $Ae_{s}$ for all $s\in\mathcal{S}$.
\end{itemize}

Then, the graded endomorphism ring $E=\oplus_{n\in\mathbb{Z}}E_{n}$ with
components
\[
E_{n}=\oplus_{t,t'\in\mathcal{X}}\operatorname*{Hom}\nolimits_{D^{b}%
(A\mMod)}(\Gamma_{t}\langle n\rangle,\Gamma_{t'}),
\]
comes equipped with idempotents $e_{t}=\operatorname*{id}\nolimits_{\Gamma
_{t}}\in E_{0}$, and there exists an equivalence%
\[
F:D^{b}(E^{op}\mMod)\iso D^{b}(A\mMod)
\]
such that $F(E^{op}e_{t})\cong\Gamma_{t}$ for all $t\in\mathcal{T}$, and
$F(X\langle n \rangle) \cong F(X)\langle n \rangle$ for $X\in D^b(A\mMod)$ and $n\in\mathbb{Z}$.
\end{theorem}

\section{Definitions}

Let $r$ be a natural number.
In this section, we define algebras $U_r$, and $V_r$ by quiver and relations.
The Cubist algebras are defined to be quotients, or subalgebras of these.

\bigskip

{\bf Motivation.}
Before stating the generators and relations which define $U_r$ and $V_r$, 
we explore their conception.

Indeed, let $J$ be a $2r$-dimensional orthogonal vector space over an algebraically closed field $\bar{k}$,
with non-degenerate bilinear form $\langle,\rangle$. 

Let $H$ be the Heisenberg Lie
superalgebra of dimension $2r+1$ associated to $J$, with $H_0 = \bar{k}, H_1 = J$,
and bracket $$[(\lambda,x), (\mu, y)] = (\langle x,y \rangle, 0).$$
Let $T$ be a maximal torus in $SO(J)$.
Then $T$ acts on $H$ as automorphisms, via $(\lambda,x)^t = (\lambda, x^t)$.

In the sequel, we define an algebra $U_r$ over an arbitrary field $k$.
Over $\bar{k}$, this algebra has the same finite-dimensional graded complex
representations as the crossed product $U(H) \rtimes T$
of the universal enveloping algebra of $H$ with $T$.

The Koszul dual of this crossed product is
$S(J^*)/\delta \rtimes T$, 
where $S(J^*)$ is the symmetric algebra on $J^*$, and
$\delta$ is the quadratic form on $J$, identified as an element of 
$S^2(J^*)$. 
This Koszul dual algebra has the same
finite-dimensional graded representations as the algebra $V_r$. 

\bigskip

\begin{definition}
We define a graded associative algebra $U_r$ by quiver and relations, over any field $k$.
The quiver $Q$ has vertices
$$\{ e_x \mid  x \in \mathbb{Z}^r \},$$
and arrows
$$\{ a_{x,i}, b_{x,i} \mid  x \in \mathbb{Z}^r, 1 \leq i \leq r \}.$$
The arrow $a_{x,i}$ is directed from $e_x$ to $e_{x+\epsilon_i}$, and
$b_{x,i}$ is directed from $e_x$ to $e_{x-\epsilon_i}$.  
$U_r$ is defined to be the path algebra $kQ$ of $Q$,
modulo \emph{square relations},
\begin{equation}\label{U1}\tag*{(U0)}
\begin{split}
a_{x,i} a_{x + \epsilon_i,i} = 0,\\
b_{x,i} b_{x - \epsilon_i,i} = 0,
\end{split}
\end{equation}
for $x \in \mathbb{Z}^r, \ 1 \leq i \leq r$, as well as \emph{supercommutation relations},
\begin{equation}\label{U2}\tag*{(U1)}
\begin{split}
a_{x,i} a_{x + \epsilon_i,j} + a_{x,j} a_{x + \epsilon_j,i} = 0,\\
b_{x,i} b_{x - \epsilon_i,j} + b_{x,j} b_{x - \epsilon_j,i} = 0,\\
a_{x,i} b_{x + \epsilon_i,j} + b_{x,j} a_{x - \epsilon_j,i} = 0,
\end{split}
\end{equation}
for $x \in \mathbb{Z}^r, \ 1 \leq i,j\leq r, \ i\neq j$, and \emph{Heisenberg relations},
\begin{equation}\label{U3}\tag*{(U2)}
b_{x,i} a_{x- \epsilon_i,i} + a_{x,i} b_{x + \epsilon_i,i} =
b_{x,i+1} a_{x- \epsilon_{i+1},i+1} + a_{x,i+1} b_{x + \epsilon_{i+1},i+1},
\end{equation}
for $x \in \mathbb{Z}^r, \ 1 \leq i < r$.
\end{definition}

\begin{remark}\label{Peachsigns}
Applying the automorphism $\tau$ of $kQ$ defined by
$\tau(e_x)=e_x$,\,
$\tau(a_{x,i}) =(-1)^{\sum_{\zeta=1}^{i}x_{\zeta}}a_{x,i}$
and $\tau(b_{x,i})=(-1)^{\sum_{\zeta=1}^{i-1}x_{\zeta}}b_{x,i}$,
we obtain an alternative presentation for $U_r$, in which the relations
\ref{U2} and \ref{U3} are replaced by
\begin{equation}\label{U2'}\tag*{(U1')}
\begin{split}
a_{x,i} a_{x + \epsilon_i,j} - a_{x,j} a_{x + \epsilon_j,i} = 0,\\
b_{x,i} b_{x - \epsilon_i,j} - b_{x,j} b_{x - \epsilon_j,i} = 0,\\
a_{x,i} b_{x + \epsilon_i,j} - b_{x,j} a_{x - \epsilon_j,i} = 0,
\end{split}
\end{equation}
for $x \in \mathbb{Z}^r, \ 1 \leq i,j\leq r, \ i\neq j$, and
\begin{equation}\label{U3'}\tag*{(U2')}
(-1)^{x_i}\left(b_{x,i} a_{x- \epsilon_i,i}
 - a_{x,i} b_{x + \epsilon_i,i}\right) = (-1)^{x_{i+1}}
\left(b_{x,i+1} a_{x- \epsilon_{i+1},i+1} - a_{x,i+1} b_{x + \epsilon_{i+1},i+1}\right),
\end{equation}
for $x \in \mathbb{Z}^r, \ 1 \leq i < r$.
This presentation coincides with that used by Peach for his
rhombal algebras (the $r=3$ case).

\end{remark}

\begin{remark}
Let $R_r$ be the algebra generated by indeterminates $a_i,b_i, i = 1,\ldots,r$, 
modulo relations  
$$a_i^2= b_i^2= 0,$$
$$a_i a_j + a_j a_i = b_i b_j + b_j b_i = 0,$$
$$a_i b_j + b_j a_i = 0, \ i \neq j,$$
$$a_ib_i + b_ia_i = a_{j}b_{j} + b_{j}a_{j}, \ i \neq j,$$
for $i,j = 1,\ldots,r$.
The algebra $R_r$ acts on the right of $U_r$, via
$$u \circ a_i = \sum_{x \in \mathbb{Z}^r}
ua_{x,i},$$
$$u \circ b_i = \sum_{x \in \mathbb{Z}^r}
ub_{x,i},$$
for $u \in U_r$. 
Similarly, $R_r$ acts on the left of $U_r$.
Let $c$ be the element $a_1b_1 + b_1a_1$ of $R_r$.
By the supercommutation relations, the left and right actions of $c$ commute: 
$c \circ u = u \circ c$, for $u \in U_r$.
Therefore, by a combination of the left and right actions, $U_r$ attains the structure
of an $R_r \otimes_{k[c]} R_r^{op}$-module.
\end{remark}

\begin{definition}
We define $V_r$ to be the quadratic dual of $U_r$. It is the path algebra of the quiver $Q'$ with vertices
$$\{ f_x \mid  x \in \mathbb{Z}^r \},$$
and arrows,
$$\{ \alpha_{x,i}, \beta_{x,i} \mid  x \in \mathbb{Z}^r, 1 \leq i \leq r \},$$
modulo \emph{commutation relations},
$$\alpha_{x,i} \alpha_{x + \epsilon_i,j} - \alpha_{x,j} \alpha_{x + \epsilon_j,i} = 0,$$
$$\beta_{x,i} \beta_{x - \epsilon_i,j} - \beta_{x,j} \beta_{x - \epsilon_j,i} = 0,$$
$$x \in \mathbb{Z}^r, 1 \leq j < i \leq r,$$
$$\alpha_{x,i} \beta_{x + \epsilon_i,j} - \beta_{x,j} \alpha_{x - \epsilon_j,i} = 0,$$ 
for $x \in \mathbb{Z}^r, 1 \leq j \leq i \leq r$,
and the \emph{Milnor relation},
$$\sum_{i=1} ^r  \beta_{x,i} \alpha_{x- \epsilon_i,i} =0,$$
for $x \in \mathbb{Z}^r$.
\end{definition}

\begin{remark}
Let $\alpha_i, \beta_i$ ($i=1,\ldots,r$) be indeterminates. Let $\gamma_i = \alpha_i \beta_i$.
Let $$\Lambda_r = k[\alpha_1,\ldots,\alpha_r,\beta_1,\ldots,\beta_r]/(\sum_{i=1}^r \gamma_i).$$
The algebra $\Lambda_r$ acts on the right of $V_r$, via
$$v \circ \alpha_i = \sum_{x \in \mathbb{Z}^r} v \alpha_{x,i},$$
$$v \circ \beta_i = \sum_{x \in \mathbb{Z}^r} v \beta_{x,i},$$
for $v \in V_r$.
Similarly, $\Lambda_r$ acts on the left of $V_r$.
Let $\Gamma_r$ be the subalgebra $k[\gamma_1,\ldots,\gamma_r]/(\sum \gamma_i)$ of $\Lambda_r$.
By the commutation relations, the right and left actions of $\Gamma_r$ commute: 
$\gamma_i \circ v = v \circ \gamma_i$ for $v \in V_r$.
Therefore, by a combination of the left and right actions, $V_r$ attains the structure of a 
$\Lambda_r \otimes_{\Gamma_r} \Lambda_r^{op}$-module.
\end{remark}

\begin{remark}
Let $x,y \in \mathbb{Z}^r$.
Any two paths in $Q'$ from $x$ to $y$, of length $d(x,y)$, 
represent the same element of $V_r$, by the commutation relations. 
We define $p_{xy}$ to be the element of $V_r$ 
representing a path in $Q'$ of length $d(x,y)$.
\end{remark}

\begin{lemma} \label{monomial}
Let $1 \leq i \leq r$.
The set 
$$B_i = \{ p_{xy} \circ m \ \mid  \ x, y \in \mathbb{Z}^r,
\ m \textrm{ is a monomial in } \gamma_1,\ldots,\gamma_{i-1},\gamma_{i+1},\ldots,\gamma_r \},$$
is a basis for $V_r$.
\end{lemma}

\begin{proof}
We first demonstrate that $B_i$ is a spanning set of $V_r$.

The commutation relations for $V_r$ reduce any path from $x$ to $y$ in $Q'$ to the form
$p_{xy} \circ q_y$, 
where $q_y$ is a monomial in $\gamma_1,\ldots,\gamma_r$.
The Milnor relation reduces $q_y$ to a polynomial in 
$\gamma_1,\ldots,\gamma_{i-1},\gamma_{i+1},\ldots,\gamma_r$.

\bigskip
We now show that $B_i$ is linearly independent.

Let $W_r$ be the vector space with basis
$$\{ (x,y,m) \ \mid  \ x,y \in \mathbb{Z}^r,
\ m \textrm{ is a monomial in } \zeta_1,\ldots,\zeta_{i-1},\zeta_{i+1},\ldots,\zeta_r \}.$$
We define an action of $V_r$ on $W_r$, via

$$f_{x} \circ (y,z,m) = \delta_{xy} (y,z,m),$$

$$ 
\alpha_{x,j} \circ (y,z,m) =
\begin{cases}
(x,z,m), & \textrm{ if }  x+\epsilon_j = y, d(x,z) = d(y,z)+1, \\
(x,z, \gamma_j m), & \textrm{ if } x+\epsilon_j = y, d(x,z) = d(y,z)-1, j \neq i, \\
-\sum_{l \neq i} (x,z, \gamma_l m), & \textrm{ if }  x+\epsilon_j = y, d(x,z) = d(y,z)-1, j = i, \\
0, & \textrm{ otherwise.}
\end{cases}
$$
$$\beta_{x,j} \circ (y,z,m) =
\begin{cases}
(x,z,m), & \textrm{ if }  x-\epsilon_j = y, d(x,z) = d(y,z)+1, \\
(x,z, \gamma_j m), & \textrm{ if } x-\epsilon_j = y, d(x,z) = d(y,z)-1, j \neq i, \\
-\sum_{l \neq i} (x,z, \gamma_l m), & \textrm{ if }  x+\epsilon_j = y, d(x,z) = d(y,z)-1, j = i, \\
0, & \textrm{ otherwise.}
\end{cases}
$$

This does in fact define an action.
Indeed, we defined this action precisely 
in such a way that the defining relations for $V_r$ are forced to hold.


Now observe that the $V_r$-module $W_r$ is generated by $\{ (x,x,1), x \in \mathbb{Z}^r \}$. 
In fact, the image of $B_i$ under the map
$$V_r\rightarrow W_r:\quad v \mapsto \sum_{x\in\mathbb{Z}^r}v\circ(x,x,1)$$
is the defining basis for $W_r$.
Therefore, $B_i$ is linearly independent.
\end{proof}

\begin{corollary} \label{Cartan1}
$C_{V_r}(q)_{xy} = (1- q^2)^{1-r} q^{d(x,y)}$, for $x,y \in \mathbb{Z}^r$. 
$\Box$
\end{corollary} 

\begin{lemma} \label{free}
The actions of
$\Lambda_r$ on $V_r$ are free.
\end{lemma}
\begin{proof}

We look at the right action.
For $x \in \mathbb{Z}^r$, the map 
$$\Lambda_r \rightarrow f_x V_r,$$
$$v \mapsto f_x \circ v,$$
is clearly surjective, and degree preserving. To see it is an isomorphism, we observe that the Hilbert
polynomials of the two sides agree. Indeed, 
by corollary~\ref{Cartan1}, summing over all $y$, we see the Hilbert series
of the right hand side is
$$(1-q^2)^{1-r}(1 + 2q + 2q^2 +\ldots)^r = (1-q^2)^{1-r}\left(\frac{1+q}{1-q}\right)^r = (1-q^2)(1-q)^{-2r},$$
which is the Hilbert series of the left hand side. \end{proof}   

\bigskip

Let ${\mathcal X} \subset \mathbb{Z}^r$ be a Cubist subset.
We now define our main objects of study.


\begin{definition}
The Cubist algebras associated to ${\mathcal X}$ are  
\begin{align*}
U_{\mathcal X} & =\: U_r/\sum_{x\in\mathbb{Z}^r\backslash\mathcal{X}}U_re_x U_r \\
\intertext{and}
V_{\mathcal X} & = \sum_{x,y\in\mathcal{X}}f_x V_r f_y. 
\end{align*}
\end{definition}

\begin{remark} \label{CartanV}
$C_{V_{\mathcal X}}(q)_{xy} = (1- q^2)^{1-r} q^{d(x,y)}$, for $x,y \in {\mathcal X}$. 
\end{remark}

\begin{remark}
The algebras $U_r, V_r$ each have an anti-involution $\omega$, which swaps 
$a_{x,i}$ and $b_{x+\epsilon_i,i}$ (respectively $\alpha_{x,i}$ and $\beta_{x+\epsilon_i,i}$).
These anti-involutions descend to $U_{\mathcal X}, V_{\mathcal X}$.
\end{remark}  

\begin{remark} \label{smallrank}
When $r=1$, the algebras $U_{\mathcal X}$ and $V_{\mathcal X}$ are all isomorphic to the field $k$.

When $r=2$, the algebras $U_{\mathcal X}$ are all isomorphic to the 
Brauer tree algebra of an infinite line, and
the algebras $V_{\mathcal X}$ are all isomorphic to the 
preprojective algebra on an infinite line. 

\begin{figure}[h]
\[
\includegraphics[height=3in]{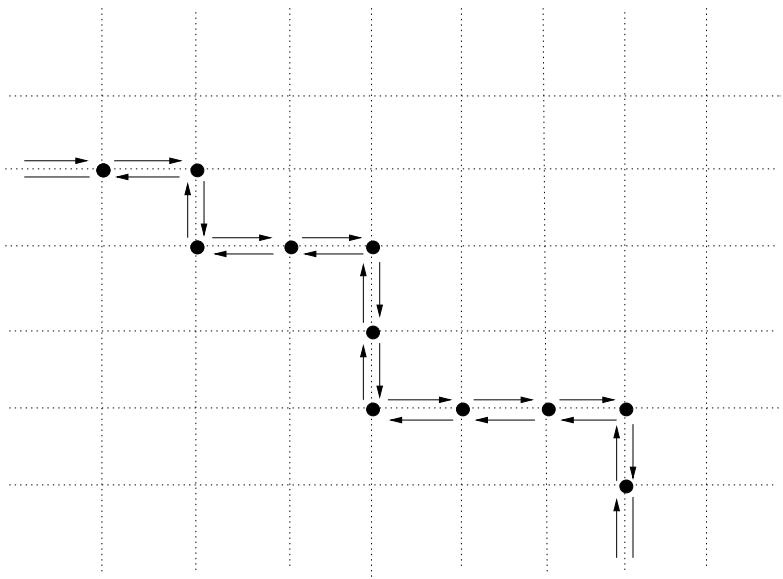}
\]
\caption{Cubist algebras when $r=2$.}
\label{rankone}
\end{figure}
\FloatBarrier

The commutation relations for the preprojective algebras come from the commutation relations for $V_2$ at flat elements $x\in\calx$ and the Milnor relations at crooked elements. The commutation relations for the Brauer tree algebras come from the supercommutation relations for $U_2$ at flat elements $x\in\calx$ and the Heisenberg relations at crooked elements; the square relations for the Brauer tree algebras come from the square relations for $U_2$ at flat elements and the supercommutation relations at crooked elements.  

When $r=3$, the algebras $U_{\mathcal X}$ are isomorphic to the rhombal
algebras of Peach introduced in \S~\ref{subsection:introrhombal}. The rhombus tilings are obtained from the Cubist subsets $\mathcal{X}$ by projecting the cubical
complex $\mathcal{C}_\mathcal{X}$ onto a hyperplane, as described in 
\S~\ref{subsection:tilings}.  The star relations come from the Heisenberg relations for $U_3$, and the mirror relations from the supercommutation relations. The two rhombuses relations come in two varieties:  straight paths of length two in the rhombus tiling are zero (a consequence of the square relations), and nonstraight paths of lenth two not bordering a single rhombus are zero (a consequence
of the supercommutation relations).

To obtain the original presentation of the rhombal algebras given by Peach, 
one has to use a different presentation of $U_3$. This alternative choice of signs is
described in remark~\ref{Peachsigns}.
\end{remark}

\section {Highest weight categories} \label{hw}

In this section and the following one, we demonstrate that $U_{\mathcal X}\mMod, V_{\mathcal X}\mMod$ 
are graded highest weight categories, in the sense of Cline, Parshall and Scott. 
As throughout the rest of the paper, we frequently forget the word graded, 
and use the term "highest weight category'' as an abbreviation for ``graded highest weight category''.

Let $\mathcal{X} \subset \mathbb{Z}^r$ be a Cubist subset.  
The bijection $\lambda$ between $\mathcal{X}$ and its set of facets $\mathcal{F}_\mathcal{X}$ established in
Proposition~\ref{vertex-facet bijection} gives rise to a partial order on $\mathcal{X}$ which usually does not coincide with the restriction of the partial order on $\mathbf{Z}^r$ that we have been employing.

\begin{proposition} \label{poset}
${\mathcal X}$ possesses a partial order $\succeq$, generated by the relations
$x \succeq y$, for $y \in \lambda x.$
\end{proposition}
\begin{proof}

We proceed by induction on $r$.

When $r=1$, the set ${\mathcal X}$ has only one element, and the lemma is trivial.

Now assume $r>1$. Let
$$x=x^0 \succ x^1 \succ \ldots \succ x^l = x$$
be a loop in ${\mathcal X}$. 
To prove the lemma, we show any such loop has length zero.
Without loss of generality, we may assume that $x = 0$.

\bigskip

Let us first observe that the loop must lie in the hyperplane 
$\mathbb{Z}^{r-1} \times 0 \subset \mathbb{Z}^r$.
This is because by definition $x^{i+1}\in \lambda x^i + F_j$ for some $j$, which implies that the last
coordinate of $x^{i+1}$ is less than or equal to the last coordinate of $x^i$
Let us write ${\mathcal Z}_0 = {\mathcal X} \cap (\mathbb{Z}^{r-1} \times 0)$.
We have proved the loop lies in ${\mathcal Z}_0$.

\bigskip

Secondly, we prove that
the loop must lie either in the subset ${\mathcal X}_0$ of ${\mathcal Z}_0$, or else in 
${\mathcal Z}_0 \backslash {\mathcal X}_0$
where
\begin{align*} 
{\mathcal X}_0 &= \{ x \in {\mathcal Z}_0 \mid  x-\epsilon_r \in {\mathcal X} \} \\
 &= \{ x \in {\mathcal Z}_0 \mid  \lambda x = x + F_j, j=1,..,r-1 \}.
 \end{align*}
Indeed, the loop cannot pass from ${\mathcal X}_0$ onto ${\mathcal Z}_0 \backslash {\mathcal X}_0$, 
since if $\lambda(x^i)= x^i + F_j$, and $j=1,..,r-1$ then
$x^{i+1} = x^i + h$, where $h \in F_j \cap (\mathbb{Z}^{r-1} \times 0)$, and so
$h - \epsilon_r \in F_j, x^i +h - \epsilon_r \in {\mathcal X}$, which implies 
that $x^{i+1} \in {\mathcal X}_0$.
  
The relation $\succeq$ on ${\mathcal Z}_0 \backslash {\mathcal X}_0$ is precisely the restriction of the relation $\geq$ on 
$\mathbb{Z}^{r-1} \times 0$, so there are no loops in ${\mathcal Z}_0 \backslash {\mathcal X}_0$.

Therefore, the loop must lie in ${\mathcal X}_0$.
By lemma~\ref{ind}, ${\mathcal X}_0$ is a Cubist subset of $\mathbb{Z}^{r-1}$.
Therefore, by the inductive hypothesis, our loop has length zero.
\end{proof}

\begin{remark} 
We call the partial order $\succeq$ the \emph{standard} partial order.
Permuting indices $1 \leq i \leq r$ with an element of the symmetric group $\Sigma_r$, we obtain
an alternative partial order on ${\mathcal X}$. Indeed, conjugating $\succeq$ by elements
of the symmetric group, we obtain $r!$ different partial orders on ${\mathcal X}$.
The theorems of this paper hold for any such partial order. 
\end{remark}

\begin{example}\label{Cubistexamples}
\begin{enumerate}
\item Consider the flat Cubist subset
$\mathcal{H}_j=\left\{x=(x_1,\ldots,x_n)\mid x_j=0\right\}$ in $\bbz^r$.
For each $x\in\mathcal{H}_j$ we have $\lambda x=x+F_j$. So for all
$x,y\in\mathcal{H}_j$.
$x\succeq y$ if and only if $x_i\leq y_i$ for $i<j$ and
$x_i\geq y_i$ for $i>j$.
\item The partial order on the corner configuration $\calx_{CC}$ is more subtle. Given $x\in\calx_{CC}$, we have $\lambda x = x+F_{m(x)}$ where
$m(x)=\min\left\{i\mid x_i=0\right\}$.  We claim that $x\succeq y$ in
$\calx_{CC}$
if and only if the following hold:
\begin{itemize}
\item $m(x)\geq m(y)$,
\item $x_i \leq y_i$, if $1\leq i \leq m(y)$,
\item $x_i \geq y_i$, if $m(x) \leq i \leq r$.
\end{itemize}
These conditions define a transitive relation on $\calx_{CC}$ which clearly
holds when $y\in \lambda x$ and therefore when $x\succeq y$. Conversely suppose
that the conditions are satisfied for some $x,y\in\calx_{CC}$. Then
$x\succeq x^{(1)}\succeq\ldots\succeq x^{(m(x)-m(y))}\succeq y$, where
$x^{(j)}\in\calx_{CC}$ is defined by
\[
x^{(j)}_i=
\begin{cases}
0 
& \text{if } m(x)-j \leq i \leq m(x), \\
x_i   & \text{otherwise}.
\end{cases}
\]
\end{enumerate}
\end{example}

Our proofs of the strong homological properties of the Cubist algebras,
rely on the following combinatorial observation. 

\begin{proposition} \label{combinatorial}
For $x,y \in {\mathcal X}$, the intersection $\lambda x \cap \mu y$ is a $d$-cube, for
some dimension $d$.

If $x \neq y$, and $\lambda x \cap \mu y \neq \emptyset$, then
there exist some $k \in \{ 1,\ldots,r \}, \sigma \in \{ \pm 1 \}$ such that 
$x + \sigma \epsilon_k \in \lambda x \cap \mu y$,
and $d(x + \sigma \epsilon_k, y) = d(x, y) -1$.
\end{proposition}
\begin{proof}

The intersection of $x + F_i$ with $y + C_j$ is certainly a $d$-cube for some $d$. 
Indeed, if it is nonempty, this subset of the $(r-1)$-cube $x + F_i$ is carved out by a number of 
inequalities on coordinates.

\bigskip

The second statement of the proposition is trivial when $r=1$. Let us assume that
the proposition is true for all Cubist subsets of $\mathbb{Z}^{r-1}$.
Let ${\mathcal X}$ be a Cubist subset of $\mathbb{Z}^r$. We prove that the proposition is true for ${\mathcal X}$, 
and thereby for any Cubist subset, by induction. 
We assume that $\lambda x \cap \mu y \neq \emptyset$ for some $x \neq y$, 
and demonstrate the existence of some coordinate $k$
so that the proposition is satisfied.
Without loss of generality, $y = 0$.

\bigskip

Case 1. 
$\lambda y = F_r, \lambda x = x + F_r$. 
Therefore, $\mu y = \mathbb{Z}_{\leq 0}^{r-1} \times \mathbb{Z} \cap {\mathcal X}$.
Then $x \in \mu y$, and $x+ \epsilon_i \in \mu y$, for some $i < r$,
and the proposition holds for $k=i$.

\bigskip

Case 2. 
$\lambda y = F_r, \lambda x = x + F_i$, for some $i<r$.
Again, $\mu y = (\mathbb{Z}_{\leq 0}^{r-1} \times \mathbb{Z}) \cap {\mathcal X}$.
We claim that the intersection of $\mu y$ with ${\mathcal X}_0$ is empty. 
Indeed, suppose for a contradiction that this set is non-empty, and there exists 
$z \in {\mathcal X}_0 \cap \mu y \subset (\mathbb{Z}_{\leq 0}^{r-1} \times 0) \cap {\mathcal X}$.
Then $z - \epsilon_r \in {\mathcal X}$. Since $z \leq 0$, we have 
$z - \epsilon_r \leq - \epsilon_r$.
Therefore $- \epsilon_r \in {\mathcal X}^+$, 
and $\epsilon_1 +\ldots+ \epsilon_{r-1} \in {\mathcal X}^+[1]$.
Thus $\epsilon_1 +\ldots+ \epsilon_{r-1} \notin {\mathcal X}$. However, $\lambda y = F_r$ by 
assumption, so $\epsilon_1 +\ldots+ \epsilon_{r-1} \in {\mathcal X}$, realising our contradiction.

We may now observe that the existence of a nontrivial intersection between 
$\mu y$ and $\lambda x$ implies the $r^{th}$ coordinate
of $x$ must be strictly greater than zero, and the proposition holds for $k=r$. 

\bigskip

Case 3. $\lambda y = F_i$ for some $i<r$, $\lambda x = x+F_r$.
Suppose there is a non trivial intersection $\lambda x \cap \mu y$.
Then $\mu y = \mathbb{Z}^{i-1} _{\leq 0} \times \mathbb{Z} \times
\mathbb{Z}^{r-i} _{\geq 0}$. If the $j^{th}$ coordinate of $x$ is strictly less than zero 
for some $j<i$, then the proposition holds for $k=j$.

If the $1^{st}, \ldots,i-1^{th}$ coordinates of $x$ are equal to zero, let us observe 
the $i^{th}$ coordinate 
of $x$ is strictly less than zero. 
Indeed, suppose for a contradiction that 
the $1^{st}, \ldots,i-1^{th}$ coordinates of $x$ are equal to zero, and 
the $i^{th}$ coordinate of $x$ is greater than, or equal to zero.
Let $z \in \lambda x \cap \mu y$.
We may assume that $z \in x + (0^i \times \{0,1\}^{r-i})$.
Thus, $z \in \mathbb{Z}_{\geq 0}^r \cap {\mathcal X}$, and 
$z + \epsilon_1 +\ldots + \epsilon_i \in {\mathcal X}$. Thus, 
$\epsilon_1+\ldots+\epsilon_i \in {\mathcal X}^-$, 
so $-\epsilon_{i+1}-\ldots-\epsilon_r \in {\mathcal X}^-[1]$,
and $-\epsilon_{i+1}-\ldots-\epsilon_r \notin {\mathcal X}$.
However, $\lambda y = F_i$, and so $-\epsilon_{i+1}-\ldots-\epsilon_r \in {\mathcal X}$,
giving a contradiction.
 
When the $1^{st}, \ldots,i-1^{th}$ coordinates of $x$ are equal to zero,
we may now observe the truth of the proposition for $k=i$.

\bigskip

Case 4. $\lambda y = F_i, \lambda x = x+ F_j$ for some $i,j<r$.
Then either the $r^{th}$ coordinate of $x$ is strictly greater than zero, in which case the proposition
holds for $k=r$, or else $x \in {\mathcal X}_0$, in which case
the induction hypothesis gives the result.
\end{proof}

\begin{corollary}
If $x \neq y \in {\mathcal X}$, and $\lambda x \cap \mu y$ is non-empty,
then there exists a $(d-1)$-cube ${\mathcal C}$, an integer $k \in \{ 1,\ldots,r\}$, and 
$\sigma \in \{ \pm 1 \}$, 
such that
$$\lambda x \cap \mu y = \mathcal{C}+\{0,\sigma\epsilon_k\},$$
and such that for all $c \in {\mathcal C}$,
$$d(x, c+ \sigma \epsilon_k) = d(x,c) + 1,$$
$$d(c+ \sigma \epsilon_k,y) = d(c,y) - 1.$$
\end{corollary}
\begin{proof}

By proposition~\ref{combinatorial}, there exists a $(r-2)$-cube ${\mathcal C}'$ such that 
$\lambda x = {\mathcal C}' + \{ 0, \sigma \epsilon_k \}$, and for all $c \in {\mathcal C}'$,
we have $d(x + \sigma \epsilon_k, y) = d(x, y) -1$, 
$d(y, c + \sigma \epsilon_k)= d(y,c) -1$, and $c \in \mu y$ exactly when 
$c + \sigma \epsilon_k \in \mu y$. The corollary follows upon putting 
${\mathcal C} = {\mathcal C}' \cap \mu y$. 
\end{proof}

\bigskip

Let $\tilde{D}_{V_{\mathcal X}}(q)$ be the ${\mathcal X} \times {\mathcal X}$ matrix, whose $xy$ entry is
$q^{d(x,y)}$, if $y \in \mu x$, and zero otherwise. 

Let $\tilde{D}_{U_{\mathcal X}}(q)$ be the ${\mathcal X} \times {\mathcal X}$ matrix, whose $xy$ entry is
$q^{d(x,y)}$, if $y \in \lambda x$, and zero otherwise.

(We shall eventually show that these are equal to the decomposition matrices 
${D}_{V_{\mathcal X}}(q)$
and ${D}_{U_{\mathcal X}}(q)$.)

\begin{lemma} \label{formula}
$$\tilde{D}_{U_{\mathcal X}}(q) \tilde{D}_{V_{\mathcal X}}(-q)^T = 1.$$
\end{lemma}
\begin{proof}

For $x,y \in {\mathcal X}$, the $xy$ entry is equal to
$$\sum_{z \in \lambda x \cap \mu y} q^{d(x,z)} (-q)^{d(z,y)}.$$

If $\lambda x \cap \mu y = \emptyset$, then this sum is equal to zero.

If $x \neq y$, and $\lambda x \cap \mu y$ is non-empty, the previous corollary shows that this sum is equal to
$$\sum_{c \in {\mathcal C}} (-1)^{d(c + \sigma \epsilon_k,y)} 
\left( q^{d(x,c + \sigma \epsilon_k) + d(c+ \sigma \epsilon_k,y)} - q^{d(x,c) + d(c,y)}
\right) = 0.$$
If $x=y$, then $\lambda x \cap \mu y = \{ x \}$, and the sum is equal to $1$.
\end{proof}

\begin{corollary}
Let $x,y \in {\mathcal X}$. If $x \in \mu y$, then $x \succeq y$.
\end{corollary}
\begin{proof}

The matrix $\tilde{D}_{U_{\mathcal X}}(q)$ is lower unitriangular with respect to $\succeq$. Therefore its inverse is also
lower unitriangular, with respect to $\succeq^{op}$.
\end{proof}

\bigskip

Let $1 \leq i \leq r$.
Consider the subalgebra $P_i$ of $V_r$ generated by elements 
$$\{ f_x, \beta_{x,1},\ldots,\beta_{x,i-1}, \alpha_{x,i+1},\ldots, \alpha_{x,r} \mid  
x \in {\mathbb{Z}^r} \}.$$
Let $L(x)$ be the simple $P_i$-module corresponding to $x \in \mathbb{Z}^r$. Let
$$\Delta_{V,i}(x) = V_r \otimes_{P_i} L(x),$$
a $V_r$-module. 
Let $\Omega_i = k[\beta_{1},\ldots,\beta_{i-1}, \alpha_{i+1},\ldots, \alpha_{r}]$ 
a polynomial
subalgebra of $\Lambda_r$ in $r-1$ variables.

\begin{lemma} The algebra $P_i$ is free over $\Omega_i$ 
with basis $\{ f_x \mid  x \in {\mathbb{Z}^r} \}$.
The algebra $V_r$ is free over $P_i$ with a basis
$\{b\circ 1\mid  b\in\mathcal{B}_i\}$, where
\begin{align*}
{\mathcal B}_i = & 
\left\{ \textrm{monomials in } \alpha_1,\ldots,\alpha_i, \beta_{i+1},\ldots,\beta_r \right\} \\
& \quad\cup\left\{ \textrm{monomials in } \alpha_1,\ldots,\alpha_{i-1}, \beta_i,\ldots,\beta_r 
\right\}.
\end{align*}
\end{lemma}

\begin{proof}
By Lemma~\ref{free}, the action of $\Lambda_r$ on $V_r$ is free,
with basis $\{ f_x \mid  x \in {\mathbb{Z}^r} \}$. The commutative algebra 
$\Omega_i$ acts freely on $\Lambda_r$, with basis 
${\mathcal B}_i$.
Hence

$$V_{r} \; = \;\:\bigoplus_{x\in\mathbb{Z}^{r}}\Lambda_{r}\circ f_{x}
      \;\: =\;\: \bigoplus_{x\in\mathbb{Z}^{r}}\bigoplus_{b\in\mathcal{B}_{i}}
             b\Omega_{i}\circ f_{x}
      \;\:=\;\: \bigoplus_{b\in\mathcal{B}_{i}}b\circ P_{i},$$
the action of $\Omega_i$ on $P_i$ is free with basis $\{ f_x \}$
and the action of $P_i$ on $V_r$ is free with basis
$\{b\circ 1\mid  b\in\mathcal{B}_i\}$.
\end{proof}

Let $\mathcal{X}\subset\mathbb{Z}^r$ be a Cubist subset. Given $x\in\mathcal{X}$, we have
$\lambda x = x+F_{i_x}$ for some $i_x$.

\begin{corollary} \label{Vdec}
Let $x \in {\mathcal X}, 1 \leq i \leq r$.
$$[\Delta_{V,i}(x) : L(y)]_q =
\sum_{z \in x + C_i} 
q^{d(x,z)},$$
$$[f_{\CX}\Delta_{V,i_x}(x) : L(y)]_q = \tilde{D}_{V_{\mathcal X}}(q)_{xy}. \Box$$ 
\end{corollary}

\begin{lemma} \label{lin}
Let $x \in {\mathcal X}$.
Then 
$\Delta_{V_{\mathcal X}}(x)$, the standard $V_{\CX}$-module corresponding to $X$ with
respect to $\succeq^{op}$,
 possesses a linear projective resolution:
$$\ldots\rightarrow \bigoplus_{\substack {y \in \lambda x \\ d(x,y) =2}} 
\negthickspace
V_{\mathcal X} f_y \langle 2 \rangle
\rightarrow \bigoplus_{\substack{y \in \lambda x \\ d(x,y) =1}} 
\negthickspace 
V_{\mathcal X} f_y \langle 1 \rangle
\; \rightarrow\;  V_{\mathcal X} f_x
\; \rightarrow \; \Delta_{V_{\mathcal X}}(x).$$
Furthermore, $\Delta_{V_{\mathcal X}}(x) \cong f_{\CX}\Delta_{V,i_x}(x)$,
and consequently ${D}_{V_{\mathcal X}}(q)=\tilde{D}_{V_{\mathcal X}}(q)$.
\end{lemma}
\begin{proof}

We first prove the existence of a linear projective resolution for $\Delta_{V_{\mathcal X}}(x)$, 
before deducing the standard property.

In fact, we first reveal a linear projective resolution of
$$\Delta_{V,i} = V_r \otimes_{\Omega_i} k,$$
for $1 \leq i \leq r$, where $k$ is the unique graded simple $\Omega_i$-module.
Note that 
$$\Delta_{V,i} \;\cong\;\: V_r \otimes_{P_i} \left( \bigoplus_{x\in\mathbb{Z}^r} f_x \Omega_i\right)
\otimes_{\Omega_i} k
\;\:\cong\;\: V_r \otimes_{P_i} \bigoplus_{x \in \mathbb{Z}^r} L(x)
\;\:\cong\;\: \bigoplus_{x \in \mathbb{Z}^r} \Delta_{V,i}(x)$$
as $V_r$-modules,
and therefore $\Delta_{V,i}(x)$ possesses a linear projective resolution.
In case $i= i_x$, applying the exact functor $\oplus_{x\in\mathcal{X}} Hom(V_r f_x,-)$, 
we obtain a linear projective resolution of $\Delta_{V_{\mathcal X}}(x)$.

How to obtain the linear projective resolution of
$\Delta_{V,i}$ ? 
Recall that $\Omega_i$ is a Koszul algebra, whose Koszul complex
$$k[\beta_1,\ldots,\beta_{i-1}, \alpha_{i+1},\ldots,\alpha_r] 
\otimes_k \bigvee(\beta_1,\ldots,\beta_{i-1}, \alpha_{i+1},\ldots,\alpha_r)
\twoheadrightarrow k$$ 
defines a linear projective resolution of $_{\Omega_i}k$.
Here, we write $\bigvee(W)$ for the vector space dual of $\bigwedge(W^*)$. Note that
$\bigwedge(W^*)$ is Koszul dual to $S(W) \cong k[W^*]$.

Recall that $P_i$ 
acts freely on $V_r$.
Furthermore, $\Omega_i$ acts freely on $P_i$. Therefore,
tensoring the Koszul complex for $\Omega_i$ with $V_r$ over 
$\Omega_i$, we obtain a linear projective resolution of $V_r$-modules,
$$V_r \bigotimes_{\Omega_i} k[\beta_1,\ldots,\beta_{i-1}, \alpha_{i+1},\ldots,\alpha_r] 
\otimes_k 
\bigvee(\beta_1,\ldots,\beta_{i-1}, \alpha_{i+1},\ldots,\alpha_r)$$ 
$$\twoheadrightarrow V_r \bigotimes_{\Omega_i} k = \Delta_{V,i}.$$
Let $x \in {\mathcal X}$. Taking a direct summand of this complex, in case $i = i_x$, 
we obtain a linear projective resolution of 
$\Delta_{V,i_x}(x)$, whose term in differential degree $d$ is 
$$\bigoplus_{h \in F_i, \mid h\mid =d} V_r f_{x + h} \otimes 
k \xi_{h},$$
where $\xi_h = \xi_h^1 \vee \ldots \vee \xi_h^r$, and 
$$\xi_h^j = 
\begin{cases} 
\beta_j &\textrm{ if the coefficient of }\epsilon_j \textrm{ in }h \textrm{ is } 1 \\
1 & \textrm{ if the coefficient of }\epsilon_j \textrm{ in }l \textrm{ is } 0 \\
\alpha_j &\textrm{ if the coefficient of }\epsilon_j \textrm{ in }h \textrm{ is } -1 \\
\end{cases}
$$
for $h \in F_{i_x}$. 
Note that all the projective indecomposable terms in 
this complex are indexed by elements of ${\mathcal X}$. 
Therefore, applying the exact functor
$\oplus_{x\in\mathcal{X}} Hom(V_r f_x,-)$,
we obtain a projective linear resolution
of the $V_{\mathcal X}$-module $f_{\CX}\Delta_{V,i_x}(x)$, as described in the
statement of the lemma.

Looking at the first two terms in our resolution, and observing that 
$x + \epsilon_j \prec x$ for $j=1,\ldots,i_x-1$, and
$x - \epsilon_j \prec x$ for $j=i_x+1,\ldots,r$,  
we perceive that $f_{\CX}\Delta_{V,i_x}(x)$ surjects onto the standard module at $x$. 
However, we also know that $y \succeq x$, for $y \in \mu x$, and so every composition factor
$L(y)$ of $f_{\CX}\Delta_{V,i_x}(x)$ satisfies $y \succeq x$. 
Therefore, $f_{\CX}\Delta_{V,i_x}(x)$ is a standard module $\Delta_{V_{CX}}(x)$ for $V_{\mathcal X}$.
\end{proof}

\begin{theorem} \label{Vhwc}
$V_{\mathcal X}\mMod$ is a highest weight category, with respect to $\succeq^{op}$.
\end{theorem}
\begin{proof}

Thanks to the linear resolution of standard modules,
we have the formula
$$\tilde{D}_{U_{\mathcal X}}(-q) C_{V_{\mathcal X}}(q) = {D}_{V_{\mathcal X}}(q).$$
Together with proposition~\ref{formula} and 
the identification ${D}_{V_{\mathcal X}}(q)=\tilde{D}_{V_{\mathcal X}}(q)$,
this implies that
$$C_{V_{\mathcal X}}(q) = D_{V_{\mathcal X}}(q)^T D_{V_{\mathcal X}}(q).$$

Now that this numerical manifestation of the highest weight property is evident, we may appeal
to a standard argument due to Dlab \cite{Dlab-QHAR}.
Let $A = V_{\mathcal X}$.
The existence of the (graded) highest weight structure is equivalent to the surjective multiplication map
$$\frac{Af_{x}}{\sum_{y\succ x}Af_{y}
Af_{x}}\otimes_k \frac{f_{x}A}{\sum_{y \succ x}f_{x}Af_{y}A}\;\longrightarrow\;\frac
{\sum_{y\succeq x}Af_{y}A}{\sum_{y \succ x}Af_{y}A}$$
being an isomorphism, for all $x \in {\mathcal X}$.
Keeping in mind that we have an anti-automorphism $\omega$ of $A$ fixing each $f_x$, we see that 
this is equivalent to the sum over $x\in\mathcal{X}$ of the Hilbert series of 
$f_z \Delta(x)\otimes_{k} f_{z'}\Delta(x)$ 
being equal to the sum over $x$ of the Hilbert series of
$f_z \frac{\sum_{y\succeq x}Af_{y}A}{\sum_{y\succ x}Af_{y}A}f_{z'}$, for all $z,z' \in {\mathcal X}$.
This is precisely the formula,
$$C_{V_{\mathcal X}}(q)_{zz'} = \left( D_{V_{\mathcal X}}(q)^T D_{V_{\mathcal X}}(q) \right)_{zz'}.$$
\end{proof}

\begin{remark} We are using a slightly different notion of highest weight category than that introduced by
Cline, Parshall and Scott. We filter projective objects by standard modules, rather than
injectives by costandards. Furthermore, we do not assume the finite interval property holds with respect
to our partial order $\preceq$. 
In other words, we do not assume that
$\{ z \mid  x \preceq z \preceq y \}$
to be finite, for all $x, y \in {\mathcal X}$.
\end{remark}

The following theorem can be proved by the dual of an argument given by Cline, Parshall and Scott  
(\cite{MR961165}, Theorem 3.9(a)).

\begin{theorem} \label{recollements}
Suppose that ${\mathcal X}$ possesses the finite interval property.
Let ${\mathcal T}_1 \subset {\mathcal X}$ be an ideal relative to $\preceq$.
Let $V_{{\mathcal T}_1}$ be the quotient of $V_{\mathcal X}$ by the ideal generated by 
$f_x, x \in {\mathcal X} \backslash {\mathcal T}_1$.

There is a full embedding of derived categories,
$$D^b(V_{{\mathcal T}_1}\mMod)  \hspace{0.2cm}  \hookrightarrow  \hspace{0.2cm} D^b(V_{\mathcal X}\mMod).$$
\end{theorem}

We conclude this section by observing the finite interval property does hold for those
Cubist sets which are obtained from the corner configuration ${\mathcal X}_{CC}$ by removing finitely many boxes.

\begin{lemma} \label{finite interval property}
Suppose that ${\mathcal X}^-$ is obtained from the corner configuration ${\mathcal X}_{CC}^-$ 
by removing finitely many elements.
Then set $\{ z \in {\mathcal X} \mid  x \preceq z \preceq y \}$ is finite for all $x,y \in {\mathcal X}$.
\end{lemma}
\begin{proof} 
The finite interval property holds for ${\mathcal X}_{CC}$, by example~\ref{Cubistexamples}(2).
\end{proof}

\section {Standard Koszulity}

\begin{theorem} $U_r$ and $V_r$ are Koszul dual. 
\end{theorem}
\begin{proof}
Note that $\Lambda_r$ is a Koszul algebra, whose Koszul complex 
$$\Lambda_r \otimes_k (\Lambda_r^!)^* \twoheadrightarrow k,$$
defines a linear projective resolution of $_{\Lambda_r}k$. (More generally, any commutative complete
intersection with quadratic regular sequence is Koszul, see \cite[\S3.1]{Froberg}.)
Tensoring over with $V_r$ over $\Lambda_r$, we obtain a linear projective resolution,
$$V_r \bigotimes_{\Lambda_r} \left( \Lambda_r \otimes_k (\Lambda_r^!)^* \right) \rightarrow 
V_r \bigotimes_{\Lambda_r} k = V_r^0,$$
of the degree zero part of $V_r$. 
The Koszul dual of $V_r$ is equal to its quadratic dual, namely $U_r$.
\end{proof}

\bigskip

Let ${\mathcal T}$ be a finite truncation of the poset $({\mathcal X}, \succeq^{op})$. Thus ${\mathcal T}$ is the intersection
of an ideal ${\mathcal T}_1$, and a coideal ${\mathcal T}_2$ in ${\mathcal X}$, and ${\mathcal T}$ has finitely many elements.
Since $V_{\mathcal X}\mMod$ has a highest weight module category, 
it has a finite dimensional subquotient $V_{\mathcal T}$,
which is quasi-hereditary, and whose simple modules are indexed by ${\mathcal T}$ (\cite{MR961165}, Theorem 3.9).

\begin{definition}
(\'Agoston, Dlab, and Luk\'acs, \cite{Agoston-Dlab-Lukacs}) 
A \emph{Standard Koszul algebra} is a
quasi-hereditary algebra, whose standard modules all possess linear resolutions.
\end{definition}

\begin{proposition}
Standard modules for $V_{\mathcal T}$ have linear projective resolutions.
$V_{\mathcal T}$ is Koszul.

The Koszul dual $V_{\mathcal T}^!$ of $V_{\mathcal T}$ is quasi-hereditary, with respect to $\succeq$.
Standard modules for $V_{\mathcal T}^!$ have linear projective resolutions.
\end{proposition}
\begin{proof}
Let $t \in {\mathcal T}$.
Let ${\mathcal X}(t,N)$ be a Cubist subset of $\mathbb{Z}^r$ defined as in Lemma~\ref{corner approximation}. Thus
${\mathcal X}(t,N)$ is identical to $\mathcal{X}$ in the region of radius $N$ about $t$, and is
obtained by removing finitely many boxes from a shift of the Corner configuration. We know by
Lemma~\ref{finite interval property} that ${\mathcal X}(t,N)$ satisfies the finite interval property. 
Note that for $N>>0$, the finite truncation $V_{\mathcal T}$ is also a finite truncation of $V_{\mathcal X}(t,N)$.
Replacing ${\mathcal X}$ by ${\mathcal X}(t,N)$ for some $N>>0$, if necessary,
we may now assume that ${\mathcal X}$ possesses the finite interval property.

Let $\Delta(s)$ be a standard $V_{\mathcal X}$-module.
By Lemma~\ref{lin}, we have a linear projective resolution,
$$\bigoplus_{t \in \lambda s} V_{\mathcal X}f_t \twoheadrightarrow \Delta(s),$$
of $\Delta(s)$. The term $V_{{\mathcal X}}f_t$ rests in homological degree $d(s,t)$.

We first prove that standard modules for $V_{{\mathcal T}_1}$ have linear projective resolutions.
Indeed, let $L$ be a simple $V_{{\mathcal T}_1}$-module, and assume $s \in {\mathcal T}_1$.
We have, by Lemma~\ref{recollements},
$$Ext_{V_{{\mathcal T}_1}\mMod}^*(L \langle j \rangle, \Delta(s)) 
\cong Ext_{V_{\mathcal X}\mMod}^*(L \langle j \rangle, \Delta(s)).$$
Therefore, our linear projective resolution of $\Delta(s)$ in $V_{\mathcal X}\mMod$ 
descends to a linear projective resolution of $\Delta(s)$ in $V_{{\mathcal T}_1}\mMod$. 

Let $s \in {\mathcal T}$. Let $f_{\mathcal T} = \sum_{t \in {\mathcal T}} f_t$.
A standard module for $V_{\mathcal T}$ is obtained by 
applying the functor $Hom_{V_{{\mathcal T}_1}}(V_{{\mathcal T}_1}f_{\mathcal T},-)$
to the standard $V_{{\mathcal T}_1}$-module $\Delta(s)$. 
Applying 
this functor
to our resolution of $\Delta(s)$, 
we obtain a linear resolution of $f_{\mathcal T} \Delta(s)$, as required. 
The terms in this resolution are projective, because they are sums of modules $f_{\mathcal T}V_{{\mathcal T}_1}f_t$
such that $t \in \lambda s \cap {\mathcal T}_1$; and therefore 
$t \succeq^{op} s$, and $t \in {\mathcal T}_1 \cap {\mathcal T}_2 = {\mathcal T}$.

We have thus proved that the algebra $V_{\mathcal T}$ is a standard Koszul algebra, in the sense of 
\'Agoston, Dlab, and Luk\'acs \cite{Agoston-Dlab-Lukacs}.
In other words, $V_{\mathcal T}$ is a quasi-hereditary algebra, 
all of whose standard modules have linear projective resolutions. 
These authors have proved that such algebras are always Koszul, and that their Koszul duals are standard Koszul,
with respect to the opposite partially ordered set.
Therefore $V_{\mathcal T}$ is Koszul, and $V_{\mathcal T}^!$ is standard Koszul with respect to $\succeq$. \end{proof}

\begin{theorem} \label{VKoszul}
$V_{\mathcal X}$ is Koszul. 
\end{theorem}
\begin{proof}

Let $x \in {\mathcal X}$.
Let ${\mathcal X}(x,N)$ be a Cubist subset of $\mathbb{Z}^r$ defined in Lemma~\ref{corner approximation}. Thus
${\mathcal X}(x,N)$ is identical to $\mathcal{X}$ in the region of radius $N$ about $x$, and is
obtained by removing finitely many boxes from a shift of the Corner configuration. We know by
Lemma~\ref{finite interval property} that ${\mathcal X}(x,N)$ satisfies the finite interval property. Consequently
there exists a finite subset ${\mathcal T}(x,N)$ of ${\mathcal X}(x,N)$, which is the intersection of
an ideal and a coideal, and contains the region of radius $N$ about $x$.
The algebra $V_{{\mathcal T}(x,N)}$ is therefore Koszul by the previous theorem.
In particular, $V_{{\mathcal T}(x,N)}$ is a quadratic algebra, and as this is true for all $x,N$, the algebra 
$V_{\mathcal X}$ is quadratic. Let $K$ be the Koszul complex associated to the quadratic algebra $V_{\mathcal X}$.
Thus, $K = \bigoplus_{N \geq 0} K_N$ is the sum of complexes 
$$K_N = \bigoplus_{i+j = N} (V_{\mathcal X})_i \otimes (V_{\mathcal X}^!)_j^*.$$
To prove that $V_{\mathcal X}$ is Koszul, it suffices to show that $f_x K_N$ is exact for all $x$, and all $N \geq 1$.
This is true, however, 
because we can identify $f_x K_N$ with the corresponding summand of the Koszul complex of 
$V_{{\mathcal T}(x,N)}$.
\end{proof}

\begin{proposition} 
$V_{\mathcal X}$ is isomorphic to the path algebra of the quiver with vertices 
$$\{ f_x \mid  x \in {\mathcal X} \},$$
and arrows
$$\{ \alpha_{x,i} \mid  x, x+\epsilon_i \in {\mathcal X}  \} \cup
\{ \beta_{x,i} \mid  x, x-\epsilon_i \in {\mathcal X}  \},$$
modulo the ideal generated by quadratic relations,
$$\alpha_{x,i} \alpha_{x + \epsilon_i,j} - \alpha_{x,j} \alpha_{x + \epsilon_j,i} = 0 \
(x,x+\epsilon_i,x+\epsilon_j,x+\epsilon_i+\epsilon_j \in {\mathcal X}),$$
$$\beta_{x,i} \beta_{x - \epsilon_i,j} - \beta_{x,j} \beta_{x - \epsilon_j,i} = 0 \
(x,x-\epsilon_i,x-\epsilon_j,x-\epsilon_i-\epsilon_j \in {\mathcal X}),$$
$$\alpha_{x,i} \beta_{x + \epsilon_i,j} - \beta_{x,j} \alpha_{x - \epsilon_j,i} = 0 \ 
(x,x+\epsilon_i,x-\epsilon_j,x+\epsilon_i-\epsilon_j \in {\mathcal X}),$$ 
$$1 \leq i,j \leq r, \quad i\neq j,$$
$$\sum _{i} \xi_{x,i} \eta_{x,i} = 0 \ (x \in {\mathcal X}_{crooked}),$$
where 
 $$(\xi_{x,i}, \eta_{x,i}) =
 \begin{cases} 
(\beta_{x,i}, \alpha_{x-\epsilon_i}), & \textrm{ if } 
x-\epsilon_i \in {\mathcal X} \\
(\alpha_{x,i}, \beta_{x+\epsilon_i,i}), &
\textrm{ if } x-\epsilon_i \notin {\mathcal X}.
\end{cases}$$
\end{proposition}
\begin{proof} 

Since $V_{\mathcal X}$ is Koszul, it is generated in degrees zero and one, 
modulo the ideal generated by quadratic relations.
In degrees zero and one, by definition $V_{\mathcal X}$ has a basis
as described in the proposition. It remains to check the quadratic relations between these generators.

The commutation relations between generators of $V_{\mathcal X}$ are
visible as the first three families of relations given in the proposition. 
The Milnor relation at $x$ is inherited from $V_r$ if $x$ is crooked.
However, the degree two part of $f_x V_{\mathcal X} f_x$ 
has dimension $r-1$, and when $x$ is flat, the Milnor relation 
need not be invoked to demonstrate that 
the elements $\{ \beta_{x,j}\alpha_{x-\epsilon_j}, j \neq i(x) \}$ 
form a basis for this space.
\end{proof}

\begin{corollary}
If $x,y \in {\mathcal X}$, then $d_{\mathcal X}(x,y) = d_{\mathbb{Z}^r}(x,y)$. $\Box$
\end{corollary}

The Koszul dual of $V_{\mathcal X}$ is equal to the quadratic dual $V_{\mathcal X}^!$.
The quadratic presentation of $V^!_{\mathcal X}$ is by the quiver with vertices,
$$\{ e_x \mid  x \in {\mathcal X} \}$$
arrows,
$$\{ a_{x,i} \mid  x, x+\epsilon_i \in {\mathcal X}  \} \cup
\{ b_{x,i} \mid  x, x-\epsilon_i \in {\mathcal X}  \},$$
and relations,
$$a_{x,i} a_{x + \epsilon_i,i} = 0 \ 
(x \in {\mathcal X}, x + \epsilon_i \in {\mathcal X}, x + 2\epsilon_i \in {\mathcal X}),$$ 
$$b_{x,i} b_{x - \epsilon_i,i} = 0 \ 
(x \in {\mathcal X}, x - \epsilon_i \in {\mathcal X}, x - 2\epsilon_i \in {\mathcal X}),$$
$$ 1 \leq i \leq r.$$
$$a_{x,i} a_{x + \epsilon_i,j} + a_{x,j} a_{x + \epsilon_j,i} = 0 \
(x,x+\epsilon_i+\epsilon_j \in {\mathcal X}),$$
$$b_{x,i} b_{x - \epsilon_i,j} + b_{x,j} b_{x - \epsilon_j,i} = 0 \
(x,x-\epsilon_i-\epsilon_j \in {\mathcal X}),$$
$$a_{x,i} b_{x + \epsilon_i,j} + b_{x,j} a_{x - \epsilon_j,i} = 0 \ 
(x,x+\epsilon_i-\epsilon_j \in {\mathcal X}),$$
$$1 \leq i,j \leq r,\quad i\neq j,$$
$$b_{x,i} a_{x- \epsilon_i,i} + a_{x,i} b_{x + \epsilon_i,i} = 
b_{x,i+1} a_{x- \epsilon_{i+1,i+1}} + a_{x,i+1} b_{x + \epsilon_{i+1,i+1}},$$ 
$$(x \in {\mathcal X}), 1 \leq i < r.$$
Here, the term $a_{x,i} a_{x+\epsilon_i,j}$ is defined to be zero
if $x+\epsilon_i$ is not an element of $\mathcal{X}$.  The same convention
applies to any term in the last four relations. 

\begin{lemma}
$V^!_{\mathcal X}$ is a locally finite dimensional algebra. 
\end{lemma}
\begin{proof}

The relations allow an element of degree $2r$ to be written as a sum of elements,
$$c^1_{\sigma 1}\ldots c^{r}_{\sigma r} d^1_{\sigma 1}\ldots d^{r}_{\sigma r},  \sigma \in \Sigma_r, 
\{c^i,d^i \} = \{a,b \}.$$ 
The term $c^1_{\sigma 1}\ldots c^{r}_{\sigma r}$ represents a path of length $r$ across an $r$-cube.
All such paths are equal up to sign, by the supercommutation relations, 
and there exists such a path through each vertex
of the $r$-cube. However, no $r$-cube is a subset of ${\mathcal X}$, so this term is equal to zero
in $V^! _{\mathcal X}$.
\end{proof}

\bigskip

Because $V^!_{\mathcal X} = Ext^*_{V_{\mathcal X}}(V^0_{\mathcal X}, V^0_{\mathcal X})$ 
is locally finite dimensional, we have the following fact:

\begin{corollary}
$V_{\mathcal X}$ has finite global dimension. $\Box$
\end{corollary}

Let $D^b(V_r\mMod)_{\mathcal X}$ denote the subcategory of $D^b(V_r \mMod)$ of complexes of modules, 
whose homology is given by simple modules outside ${\mathcal X}$.

\begin{corollary}
There is a recollement of derived categories,
$$D^b(V_r\mMod)_{\mathcal X} \hspace{0.2cm} \overleftarrow{\rightleftarrows}  \hspace{0.2cm} D^b(V_r\mMod)  \hspace{0.2cm}
\overleftarrow{\rightleftarrows}  \hspace{0.2cm} D^b(V_{\mathcal X}\mMod).$$
\end{corollary}
\begin{proof}

Since $V_{\mathcal X}$ has finite global dimension,
a theorem of Cline, Parshall, and Scott implies that  
the map $D^b(V_r\mMod) \rightarrow D^b(V_{\mathcal X}\mMod)$ extends to a recollement of derived categories
(\cite{MR961165}, Theorem 2.3).
\end{proof}

\begin{theorem} \label{UKoszul}
We have an isomorphism, $V_{\mathcal X}^! \cong U_{\mathcal X}$.
In other words, $U_{\mathcal X}$ is Koszul dual to $V_{\mathcal X}$.
\end{theorem}
\begin{proof}

The relations for $V^!_{\mathcal X}$ are precisely the relations for $U_r$, modulo the relation $e_{\mathcal X}=0$.
Therefore, there exists a surjection 
$V^!_{\mathcal X} \rightarrow U_r / U_r e_{\mathcal X} U_r = U_{\mathcal X}$
of graded algebras.
 
Thanks to the aforementioned recollement, there exists a surjection,
$$U_r = Ext^*_{V_r}(V_r^0, V_r^0) \rightarrow 
Ext_{V_{\mathcal X}}^*(V^0_{\mathcal X}, V^0_{\mathcal X}) = V^!_{\mathcal X},$$
in the kernel of which lies $U_r e_{\mathcal X} U_r$.  
Thus, we have a surjection $U_{\mathcal X} \rightarrow V^!_{\mathcal X}$ of graded algebras.

We have proved the existence of graded surjections from $V^!_{\mathcal X}$ to $U_{\mathcal X}$, and back. 
Such maps preserve homogeneous spaces of projective indecomposable modules, which are finite dimensional. 
Each of these surjections is therefore an isomorphism. \end{proof}

\begin{corollary} \label{UVeq}
There is an equivalence of derived categories,
$$D^b(U_{\mathcal X} \mMod) \cong D^b(V_{\mathcal X} \mMod).$$
\end{corollary}
\begin{proof}

By a theorem of Beilinson, Ginzburg, and Soergel, 
such an equivalence holds for a general pair of Koszul dual algebras,
one of which is locally finite dimensional (\cite{MR1322847}, Theorem 2.12.6).
\end{proof}

Let $D^b(U_r\mMod)^{\mathcal X}$ denote the quotient category of $D^b(U_r\mMod)$ 
by the smallest thick subcategory containing $U_{\mathcal X}\mMod$.

\begin{corollary}
There is a recollement of derived categories,
$$D^b(U_{\mathcal X}\mMod)  \hspace{0.2cm} \overleftarrow{\rightleftarrows}  \hspace{0.2cm} D^b(U_r\mMod)  \hspace{0.2cm}
\overleftarrow{\rightleftarrows}  \hspace{0.2cm} D^b(U_r\mMod)^{\mathcal X}.$$
\end{corollary}
\begin{proof}

Since $V_{\mathcal X}$ is equal to $f_{\mathcal X} V f_{\mathcal X}$,
we know that $Ext_{U_r}^i(S,T) = Ext_{U_{\mathcal X}}^i(S,T)$, for all simple $U_{\mathcal X}$-modules $S,T$.
By functorality, $Ext_{U_r}^i(M,T) = Ext_{U_{\mathcal X}}^i(M,T)$, for all finite dimensional 
$U_{\mathcal X}$-modules $M$, and all simple $U_{\mathcal X}$-modules $T$.
Again by functorality, we find that
$Ext_{U_r}^i(M,N) = Ext_{U_{\mathcal X}}^i(M,N)$, for all finite dimensional 
$U_{\mathcal X}$-modules $M,N$.

A theorem of Cline, Parshall and Scott allows us to deduce
that the map $D^b(U_{\mathcal X}\mMod) \rightarrow D^b(U_r\mMod)$ extends to a recollement
(\cite{MR957457}, Theorem 3.1).  \end{proof}

\begin{corollary}\label{Uqh}
$U_{\mathcal X}\mMod$ is a highest weight category, with respect to $\succeq$.
Standard modules possess linear projective resolutions.
\end{corollary}
\begin{proof}

For $x \in {\mathcal X},N \geq 0$, let ${\mathcal X}(x,N)$ be a Cubist subset
which can be identified with ${\mathcal X}$ in a box of diameter $N$ around $x$, 
such that ${\mathcal X}(x,N)$ is obtained by removing
boxes from a translate of the corner configuration.
Such an ${\mathcal X}(x,N)$ has the finite interval property, and therefore
for all finite truncations ${\mathcal T}(x,N)$,
the algebra $U_{{\mathcal T}(x,N)}$ Koszul dual to $V_{{\mathcal T}(x,N)}$ is standard Koszul.
Because $U_{\mathcal X}$ is locally finite dimensional, $U_{\mathcal X}$ can be identified with $U_{{\mathcal T}(x,N)}$ in
a large region around $x$, so long as $N$ is large enough.
For this reason, the regular $U_{\mathcal X}$-module possesses a $\Delta$-filtration. 
As the Koszulity of $V_{{\mathcal T}(x,N)}$ implied the Koszulity of $V_{{\mathcal X}}$ 
in the proof of theorem~\ref{VKoszul}, now the
existence of linear projective resolutions for standard $U_{{\mathcal T}(x,N)}$-modules 
imply the existence of linear projective resolutions for standard $U_{\mathcal X}$-modules.
\end{proof}

\begin{lemma}\label{Ustandard}
The standard module $\Delta_{U_{\mathcal X}}(x)$ of $U_{\mathcal X}$
has a basis $\{ q_{y} \mid  y \in \lambda x \}$, with $q_{y}$ in degree $d(y,x)$.
If $y,y'\in\lambda x$ and $d(y',x)=d(y',y)+d(y,x)$, then $\gamma_{y'y}q_y=\pm q_{y'}$.  In particular,  
$\Delta_{U_{\mathcal X}}(x)$ has simple socle
$L(x^{op})\langle w \rangle$.
\end{lemma}
\begin{proof}
Let $K^{-1}: D^{b}(V_\mathcal{X}\mMod)\rightarrow D^{b}(U_\mathcal{X}\mMod)$ be the inverse Koszul duality functor (see \cite[Theorem 1.2.6]{MR1322847}). Then $K^{-1}$ is a triangulated functor such that $K^{-1}(M\langle n \rangle)= K^{-1}(M)[-n]\langle -n \rangle$, $K^{-1}(V_\mathcal{X}e_x)=L(x)$, and $K^{-1}(L(x))=U^*_\mathcal{X}e_x$. By Lemma~\ref{lin}, we deduce that
$K^{-1}(\Delta_{V_{\mathcal{X}}}(x))$ is quasiisomorphic to a module $M$ whose composition factors are described by the matrix $D_{U_{\mathcal{X}}}(q)$. Moreover, by Lemma ~\ref{Vdec}, $M$ has an injective resolution 
$U^*_\mathcal{X}e_x\rightarrow\oplus_{y\in\mu x, d(y,x)=1} U^*_\mathcal{X}e_y\langle -1 \rangle \rightarrow\ldots$. Hence $M$ is the costandard module of $U_{\mathcal{X}}$ associated to the simple module $L(x)$, and
$\Delta_{U_{\mathcal{X}}}=M^*$, the corresponding standard module, also has composition factors given by $D_{U_{\mathcal{X}}}(q)$.  By comparing $D_{U_{\mathcal{X}}}(q)$ with $C_{U_{\mathcal{X}}}(q)$, we deduce that the images  $q_{y}$, $y\in\lambda x$ of $\gamma_{y,x}$, $y\in\lambda x$ under a surjective homomorphism
$U_{\mathcal{X}}e_x\rightarrow \Delta_{\mathcal{X}}(x)$ form a basis for $\Delta_{\mathcal{X}}(x)$, and furthermore we have 
$\gamma_{y'y} q_{y} = \pm q_{y}$ whenever $d(y',x)=d(y',y)+d(y,x)$.
\end{proof}

\begin{corollary} \label{preformula}
$$C_{U_{\mathcal X}}(q) = D_{U_{\mathcal X}}(q)^TD_{U_{\mathcal X}}(q).$$
\end{corollary}

\begin{definition}
Let $x \in {\mathcal X}$.
We call the standard $U_{\mathcal X}$-module $\Delta_{U_{\mathcal X}}(x)$ the \emph{facetious module} 
corresponding to $\lambda x$. It is a graded $U_{\mathcal X}$-module whose
head is $L(x)$, and whose Hilbert series is $\sum_{y \in \lambda x} q^{d(x,y)} L(y)$.
\end{definition}

\begin{theorem}
$U_{\mathcal X}$, and $V_{\mathcal X}$ have homogeneous cellular bases. For either algebra,
there is a canonical choice of such basis, with respect to our fixed generators.
\end{theorem}
\begin{proof}
Let ${\mathcal X} \subset \mathbb{Z}^r$ be a Cubist subset.
The cellularity is immediate from the definition of S. K{\"o}nig and C. Xi \cite[Corollary 4.2]{Konig-Xi-OTSOCA}:
a quasi-hereditary algebra which has a decomposition by primitive idempotents each fixed by
an anti-involution is cellular. Our anti-involution is $\omega$, which swaps 
$a_{x,i}$ and $b_{x+\epsilon_i,i}$ (respectively $\alpha_{x,i}$ and $\beta_{x+\epsilon_i,i}$). 
The grading on our algebras is compatible with the highest weight structure, and therefore with the cellular
structure. 
Cellular bases can be canonically defined with respect to the generators of $U_{\mathcal X}, V_{\mathcal X}$, 
because the $q$-decomposition numbers are all monomials. 
Indeed, the basis for $U_{\mathcal X}$ consists of products 
$q_{yx}.q_{xz}$, where $y \in \lambda x, x \in \lambda z$.
The analogous basis for $V_{\mathcal X}$ consists of products 
$p_{yx}.p_{xz}$, where $y \in \mu x, x \in \mu z$.
\end{proof}

\section{Symmetry}

Before proving an algebraic property of the Cubist algebras, we must always do some combinatorics.
Let us prove some lemmas, before we deduce the symmetry of $U_{\mathcal X}$\ldots

\begin{lemma} \label{opp}
Fix a partial order $\succeq$ on ${\mathcal X}$.
The map $x \mapsto x^{op}$ which takes $x$ to its opposite in $\lambda x$
is bijective.
\end{lemma}
\begin{proof}

We prove this for the standard partial order, explicitly written down in section \ref{hw}.
Let $\lambda: {\mathcal X} \rightarrow {\mathcal F_X}$ be the map defined by this partial order $\succeq$.
Let $w_0$ be the longest element of $\Sigma_r$.
Let $\lambda': {\mathcal X} \rightarrow {\mathcal F_X}$ be the map defined by the partial order $\succeq^{w_0}$.
Since $x^{op} = \lambda'^{-1} \lambda(x)$, and $\lambda, \lambda'$ are bijective, the lemma holds. \end{proof}

\begin{lemma} \label{suitfacet}
Let $x,y$ be distinct elements of ${\mathcal X}$.
Then there exists a facet $F$ of $\mathcal{X}$ such that $x\in F$ and $y \notin F$.
\end{lemma}
\begin{proof}
Because $\mathcal{C_X}$ is homeomorphic to $\mathbb{R}^{r-1}$, any $d$-cube $C$ in ${\mathcal X}$ can be characterised
as the intersection of facets containing $C$.
For this reason, ${\mathcal P}_x$ can be characterised as the largest subcomplex of $\mathcal{C_X}$
whose vertices are all contained in ${\mathcal P}_x$. 
However, since $x,y$ are distinct, 
we have ${\mathcal P}_x \neq {\mathcal P}_y$. Therefore, some vertex of ${\mathcal P}_x$
is not a vertex of ${\mathcal P}_y$. This completes the proof of the lemma.
\end{proof}

\begin{definition}
Let $x \in {\mathcal X}$, and let $\xi$ be a facet of ${\mathcal X}$ containing $x$.
Let
$$\xi_i =
\begin{cases}
a_i & \textrm{ if } x+\epsilon_i \in \xi, \\
b_i & \textrm{ if } x-\epsilon_i \in \xi, \\
1 &  \textrm{otherwise,}
\end{cases}$$
for $i=1,\ldots,r$. 
Let $s(\xi)$ to be the order of the set $\{ \xi_i \mid  \xi_i = b_i \}$.
\end{definition} 

\begin{definition}
Let ${\mathcal F}_x = \{ \xi \in \mathcal{F_X} \mid  x \in \xi \}$, for $x \in {\mathcal X}$
\end{definition}

\begin{lemma} \label{transpos}
If $\xi, \xi' \in {\mathcal F}_x$, then there is a sequence 
$\xi = \xi^0,\xi^1\ldots,\xi^l = \xi'$ in ${\mathcal F}_x$, such that $\xi^i \cap \xi^{i+1}$ is an $r-2$-cube.
\end{lemma} 
\begin{proof} 
The polytope ${\mathcal P}_x$ has a $1$-skeleton whose vertices correspond to elements of ${\mathcal F}_x$,
and whose edges correspond to $r-2$-cubes containing $x$. The poset of faces of ${\mathcal P}_x$ ordered by inclusion
is the opposite of the poset of cubes containing $x$, ordered by inculsion.
The $1$-skeleton of any polytope is connected. Therefore, the lemma holds.
\end{proof}
 
\begin{lemma} \label{cases}
Suppose that $\xi, \eta \in {\mathcal F_X}$, and that $\xi \cap \eta$ is an $r-2$-cube.
Then one of the following holds:

(a) There exists $i \in [1,r]$, such that $\xi_i = a_i, \eta_i = b_i$, and $\xi_k = \eta_k$ for 
$k \in [1,r], k \neq i$.

(b) There exists $i \in [1,r]$, such that $\xi_i = b_i, \eta_i = a_i$, and $\xi_k = \eta_k$ for 
$k \in [1,r], k \neq i$.

(c) There exist distinct $i,j \in [1,r]$, such that $(\xi_i,\xi_j) = (a_i,1)$, $(\eta_i, \eta_j) = (1, a_j)$,
and $\xi_k = \eta_k$ for $k \in [1,r], k \neq i,j$.

(d) There exist distinct $i,j \in [1,r]$, such that $(\xi_i,\xi_j) = (b_i,1)$, $(\eta_i, \eta_j) = (1, b_j)$,
and $\xi_k = \eta_k$ for $k \in [1,r], k \neq i,j$.
\end{lemma}
\begin{proof} 

We only need eliminate a couple of possibilities.
The first is the existence of distinct
$i,j \in [1,r]$, such that $(\xi_i,\xi_j) = (a_i,1)$, $(\eta_i, \eta_j) = (1, b_j)$,
and $\xi_k = \eta_k$ for $k \in [1,r], k \neq i,j$.
However, such an arrangement implies that 
$x_\xi = x + \sum_{k, x+\epsilon_k \in \xi} \epsilon_k \in \xi \subset {\mathcal X}$,
as well as
$x_\eta = x - \sum_{k, x-\epsilon_k \in \eta} \epsilon_k \in \eta \subset {\mathcal X}$.
Thus $x_\eta, x_\eta[1] = x_\xi \in {\mathcal X}$, which cannot happen.

The remaining possibility is   
the existence of distinct
$i,j \in [1,r]$, such that $(\xi_i,\xi_j) = (a_i,1)$, $(\eta_i, \eta_j) = (1, b_j)$,
and $\xi_k = \eta_k$ for $k \in [1,r], k \neq i,j$. 
This we can eliminate for identical reasons. \end{proof}

\begin{lemma} \label{Frob}
Let $x \in {\mathcal X}$, and let $\xi, \eta$ be facets of ${\mathcal X}$ containing $x$.
Then 
$$(-1)^{s(\xi)} e_x \circ \xi_1\ldots\xi_{r-1} \xi_{r-1}^\omega\ldots\xi_{1}^\omega
= (-1)^{s(\eta)} e_x \circ \eta_1\ldots\eta_{r-1} \eta_{r-1}^\omega\ldots\eta_{1}^\omega,$$
and these are non-zero elements of $U_{\mathcal X}$.
\end{lemma}
\begin{proof}

The written elements of $U_{\mathcal X}$ are non-zero, by the cellularity of $U_{\mathcal X}$.
By lemma~\ref{transpos}, we may assume that $\xi \cap \eta$ is an $r-2$-cube.
By lemma~\ref{cases}, we should check cases (a)-(d). 

Case (a):
By the supercommutation relations, the left hand side is equal to 
$$(-1)^{s(\xi)} e_x \circ \xi_1\ldots\xi_{i-1}\xi_{i+1}\ldots\xi_r a_i
b_i \xi_r^\omega\ldots\xi_{i+1}^\omega \xi_i^\omega \ldots \xi_1^\omega,$$
whilst the right hand side is equal to
$$(-1)^{s(\xi)+1} e_x \circ \xi_1\ldots\xi_{i-1}\xi_{i+1}\ldots\xi_r b_i
a_i \xi_r^\omega\ldots\xi_{i+1}^\omega \xi_{i-1}^\omega \ldots \xi_1^\omega.$$
Let $j \neq i$ be that number such that $\xi_j = \eta_j = 1$.
The difference of the left and right hand side is
$$(-1)^{s(\xi)} e_x \circ \xi_1\ldots\xi_{i-1}\xi_{i+1}\ldots\xi_r (a_ib_i + b_ia_i)
\xi_r^\omega\ldots\xi_{i+1}^\omega \xi_i^\omega \ldots \xi_1^\omega,$$
which is equal to 
$$(-1)^{s(\xi)} e_x \circ \xi_1\ldots\xi_{i-1}\xi_{i+1}\ldots\xi_r (a_jb_j + b_ja_j)
\xi_r^\omega\ldots\xi_{i+1}^\omega \xi_i^\omega \ldots \xi_1^\omega,$$
by the Heisenberg relation.
Note that $x - \sum_{k,x - \epsilon_k \in \xi} \epsilon_k \in {\mathcal X}$, and so its shift, 
$x + \epsilon_j + \sum_{k,x + \epsilon_i \in \xi} \epsilon_k \notin {\mathcal X}$, and thus  
$$e_x \circ \xi_1\ldots\xi_{i-1}\xi_{i+1}\ldots\xi_r a_j = \pm 
e_x \circ a_j \prod_{k,x + \epsilon_k \in \xi} \xi_{k}
\prod_{k,x - \epsilon_k \in \xi} \xi_{k} = 0.$$  
Similarly, $e_x \circ \xi_1\ldots\xi_{i-1}\xi_{i+1}\ldots\xi_r b_j = 0$,
and we have proved the difference of left and right hand side is zero.

Case (b) is proved identically to case (a).

\bigskip

Case (c): The left hand side is equal to
$$(-1)^{s(\xi)} e_x \circ \xi_1\ldots\xi_{i-1}\xi_{i+1}\ldots\xi_r a_i
b_i \xi_r^\omega\ldots\xi_{i+1}^\omega \xi_i^\omega \ldots \xi_1^\omega,$$
whilst the right hand side is equal to
$$(-1)^{s(\xi)} e_x \circ \xi_1\ldots\xi_{j-1}\xi_{j+1}\ldots\xi_r a_j
b_j \xi_r^\omega\ldots\xi_{j+1}^\omega \xi_j^\omega \ldots \xi_1^\omega.$$
By the Heisenberg relations, we have
$$(-1)^{s(\xi)} e_x \circ \xi_1\ldots\xi_{i-1}\xi_{i+1}\ldots\xi_r (a_ib_i + b_ia_i)
\xi_r^\omega\ldots\xi_{i+1}^\omega \xi_i^\omega \ldots \xi_1^\omega =$$
$$(-1)^{s(\xi)} e_x \circ \xi_1\ldots\xi_{j-1}\xi_{j+1}\ldots\xi_r (a_jb_j + b_ja_j)
\xi_r^\omega\ldots\xi_{j+1}^\omega \xi_j^\omega \ldots \xi_1^\omega.$$
However, 
$x -\epsilon_j = \sum_{k,x - \epsilon_k \in \xi} \epsilon_k \in {\mathcal X}$, and so its shift, 
$x + \sum_{k,x + \epsilon_i \in \xi} \epsilon_k \notin {\mathcal X}$,
and thus
$$e_x \circ \xi_1\ldots\xi_{i-1}\xi_{i+1}\ldots\xi_r b_i = \pm 
e_x \circ b_i \prod_{k,x - \epsilon_k \in \xi} \xi_{k}
\prod_{k,x + \epsilon_k \in \xi} \xi_{k} = 0.$$  
Similarly, $e_x \circ \xi_1\ldots\xi_{j-1}\xi_{j+1}\ldots\xi_r b_j = 0$,
and we have proved the equality of left and right hand side.

Case (d) is proved identically to case (c). \end{proof}

\begin{lemma} \label{conjelt}
Let $F$ be a facet of a cubist set ${\mathcal X}$, containing an element $x$.
Then there exists $\sigma \in \Sigma_r$, such that $F = x + F_i^\sigma$.
\end{lemma}
\begin{proof}

We have
\[
F=\left\{    x  +\sum_{j\in S}a_{j}\epsilon_{j}+\sum_{j\notin S\cup\{i\}}%
a_{j}\epsilon_{j},\quad a_{j}=0,1\right\}
\]
for some $S\subseteq\{1,\ldots,r\}$ and $i\notin S$.
Let $\sigma$ be some element of $\Sigma_r$, such that $\sigma (\{1,\ldots,|S|\}) = S$.
\end{proof}

There are number of equivalent definitions of a symmetric finite dimensional algebra, not all of which pass 
to infinite dimensions in a satisfactory way.

In this paper, we say an infinite dimensional algebra is symmetric, 
if there exists an associative, symmetric, non-degenerate bilinear form on $A$. 
A more functorial definition, which follows from the existence of such a form, 
is an equivalence
$Hom(P, M) \cong Hom(M, P)^*$, natural in $A$-modules $M$, and in projective $A$-modules $P$ 
\cite{Rickard-Symmetric}.

\begin{theorem} \label{symmetric}
$U_{\mathcal X}$ is symmetric.
\end{theorem}
\begin{proof}

Let us define a bilinear form on $U_{\mathcal X}$ by the formula 
$$(u_1, u_2) = \sum_{x \in {\mathcal X}} c(u_1.u_2),$$ 
where $c(u)$ the coefficient of the element 
$(-1)^{s(\xi)} \xi_1\ldots\xi_{1} \xi_{r-1}^\omega\ldots\xi_{1}^\omega$
of lemma~\ref{Frob} in $u$. 
The form is clearly associative.
Let us prove its non-degeneracy.

By lemma~\ref{Frob}, we know that the the degree $2r-2$ part of $U_{\mathcal X}$ 
is isomorphic to $k{\mathcal X}$ as a $U_{\mathcal X}$-$U_{\mathcal X}$-bimodule, 
and that $U_{\mathcal X}$ vanishes in degree $2r-1$.
Therefore, let $0 \neq u \in e_x U_{\mathcal X} e_y$ be homogeneous of degree $i <2r-2$. We 
are required to show that $U_{\mathcal X}^{>0}.u \neq 0$.  Suppose not. Then
the socle of $U_{\mathcal X} e_y$ contains a summand isomorphic to $L(x)\langle i \rangle$.  
Hence the same is true
of the socle of one of the factors $\Delta_{U_{\mathcal{X}}}(z)\langle d(z,y) \rangle$, 
$y\in\lambda z$ in a standard filtration of $U_{\mathcal X} e_y$. 
By Lemma~\ref{opp} there is a unique $z\in\mathcal{X}$ such that $z^{op}=x$, 
and by Lemma~\ref{Ustandard} we must have $y\in\lambda z$ and $i=d(z,y)+r-1$. 
In particular we have $y\neq x$.  

Now the above argument remains valid for any conjugate of the partial order on $\mathcal{X}$. 
By Lemma~\ref{suitfacet} there is a facet $F$ of $\mathcal{X}$ such that $x\in F$ and $y\notin F$. 
Let $z'$ be the vertex of $F$ opposite $x$.
By lemma~\ref{conjelt}, we may choose a partial order $\succeq'$
on $\mathcal{X}$, with respect to which $\lambda' z'=F$. 
Then $z'^{op}=x$ with respect to $\lambda'$, but $y\notin \lambda' z'$, which is a contradiction. 

Let us finally observe that $(,)$ is symmetric. 
We need to see that the Nakayama automorphism $N$ of $U_{\mathcal X}$, defined by 
$$(x,y) = (N(y),x),$$
is trivial.
The Nakayama automorphism is a graded homomorphism, 
so it is enough to know that $N(x) = x$, for arrows $x$,
since these generate $U_{\mathcal X}$. 
The explicit formula of lemma~\ref{Frob} makes this clear.
\end{proof}

\begin{corollary} \label{gldim}
Every principal indecomposable $U_{\mathcal X}$-module has radical length $2r-1$. The global dimension of 
$V_{\mathcal X}$ is $2r-2$. 
\end{corollary}

\begin{corollary}
$U_{\mathcal X}$ is Ringel self-dual. 
\end{corollary}
\begin{proof}

Projective modules are also injective, and therefore tilting. \end{proof}

R. Mart{\'{\i}}nez-Villa has characterised Koszul self-injective algebras 
by a noncommutative Gorenstein property \cite{Martinez-Villa-GSAKA}, 
providing a corollary to theorem~\ref{symmetric}:

\begin{corollary}
There is an isomorphism 
of graded $k{\mathcal X}$-$k{\mathcal X}$-bimodules, 
$$Ext^*_{V_{\mathcal X}\mMod}(k{\mathcal X}, V_{\mathcal X}) \cong k{\mathcal X} \langle 2r-2 \rangle. \Box$$
\end{corollary}

Let $A$ be a $\mathbb{Z}_+$-graded algebra.
Let $A \Nil$ be the Serre subcategory of $A \mMod$, 
of modules on which $A_{>0}$ acts nilpotently.
Let $$A \QMod = A \mMod/A \Nil,$$ the non-commutative analogue of coherent sheaves on 
a projective scheme $Proj(A)$ associated to a commutative graded algebra $A$.
Let $A \qMod$ be the subcategory of $A \QMod$ of graded $V_{\mathcal X}$-modules
with projective resolutions whose terms are all finitely generated.  
A theorem of R. Mart{\'{\i}}nez-Villa and A. Martsinkovsky \cite{Martinez-Villa-Martsinkovsky} provides a second
corollary to theorem~\ref{symmetric}:

\begin{corollary} \label{Serre} (Serre duality)
For graded $V_{\mathcal X}$-modules $M,N$, 
we have,
$$Ext^i_{V_{\mathcal X} \qMod}(M,N) \cong Ext^{2r-3-i}_{V_{\mathcal X} \qMod}(N,M\langle-2r+2\rangle)^*. \Box$$
\end{corollary}

Note that in the above duality, we do not twist the module $N$ by an automorphism of $V_{\mathcal X}$. 
The absence of such an automorphism (canonical bundle)
is known to some as the \emph{Calabi-Yau} property. 
The Calabi-Yau dimension of $V_{\mathcal X} \qMod$ is $2r-3$.
In particular, when $r=3$, these categories have Calabi-Yau dimension 3.
String theorists are said to be interested in such things.

Let us mention here that categories of \emph{ungraded} 
representations of the algebras $V_{\mathcal X}$ also possess the Calabi-Yau property, in dimension $2r-2$,
as we record below as corollary \ref{CY}. 

A triangulated category $\mathcal{T}$ over $k$ is said to be Calabi-Yau of dimension $d$, if we have
$Hom(M,N) \cong Hom(N,M[d])^*$, naturally in $M,N \in \mathcal{T}$.

If $B$ is a positively graded algebra, with positive part $B_+$, 
let us write $\widehat{B} = \underleftarrow{\lim} (B/B_+^n)$ 
for the inverse limit of quotients $B/B_+^n$ of $B$, for $n \geq 1$.
We have,

\begin{corollary} \label{CY}
The bounded derived category of ungraded $\widehat{V}_{\mathcal{X}}$-modules 
is a Calabi-Yau category, of dimension $2r-2$.
\end{corollary}

Corollary \ref{CY} is an immediate consequence of the following lemma, 
applied in case $A = U_{\mathcal{X}}$, $B = V_{\mathcal{X}}$.

\begin{proposition}
Suppose $A$ is a Koszul algebra, with Koszul dual $B$.
The bounded derived category of ungraded $\widehat{B}$-modules 
is Calabi-Yau of dimension $d$ if, and only if, 
$A$ possesses a super-symmetric associative bilinear form of degree $-d$.
\end{proposition}
\begin{proof}
Given an associative algebra $A$, we
let $A \perf_0$ denote the category of perfect complexes of ungraded $A$-modules,
that is, the full subcategory of the derived category generated by $_AA$.
Then $A$ is symmetric if it obeys either of the two equivalent conditions: 

(1) $A$ possesses a symmetric, associative, non-degenerate bilinear form.

(2) $A \perf_0$ is a Calabi-Yau triangulated category of dimension $0$.

\bigskip

As we have assumed throughout our paper, $A$ is in fact a graded algebra.
Thinking of $A$ as a differential graded algebra, with zero differential,
we let $A \perf$ denote the full subcategory of the derived category $D_{dg}(A)$ of dg modules generated by $_AA$.
We say $A$ is $d$-symmetric if it obeys either of the two equivalent conditions (see \cite{VandenBergh}, A.5.1): 

(1) $A$ possesses a super-symmetric (= graded symmetric), 
associative, non-degenerate bilinear form of degree $-d$.

(2) $A \perf$ is a $d$-dimensional Calabi-Yau triangulated category.

\bigskip

Suppose now that $A$ is a Koszul algebra, with Koszul dual $B$.
By B. Keller's formulation of Koszul duality, $D_{dg}(A)$ is equivalent to the derived category of 
ungraded $B$-modules \cite{Keller}.
Under this equivalence, $A \perf$ corresponds to that subcategory of the 
derived category of $B$-modules
generated by the collection of finite dimensional $B^0$-modules.
In other words $A \perf$ is equivalent to the bounded derived category of ungraded $\widehat{B}$-modules.
The statement of the lemma is now clear. 
\end{proof}

To end this section, let us give a few definitions concerning the algebra $U_{\mathcal{X}}$ which will
be of some application in the sequel.

\begin{definition}
Let $\omega:kQ\longrightarrow kQ$ be the graded antiautomorphism, defined by
$\omega(e_{x})=e_{x}$, $\omega(a_{x,i})=b_{x+\epsilon_{i},i}$, $\omega
(b_{x,i})=a_{x-\epsilon_{i},i}$. We have $\omega(R)=R$ , so $\omega$ defines
antiautomorphisms on $U_{r}$ and $U_{\mathcal{X}}$ which we still call
$\omega$. 
\end{definition}

Given a graded $U_{\mathcal{X}}$-module $V=\oplus V_{n}$, we make
the dual $V^{\ast}=\oplus(V^{\ast})_{n}=\oplus V_{-n}^{\ast}$ a graded
$U_{\mathcal{X}}$-module via $(u\phi)(v)=\phi(\omega(u)v)$.

\begin{definition}
Let $\kappa$ be the automorphism of $\mathbb{Z}^{r}$ defined by $\kappa(x)=-x$. 
\end{definition}

Then
$G=\mathbb{Z}^{r}\rtimes\langle\kappa\rangle$ acts as a group of automorphisms
on $\mathbb{Z}^{r}$ (the $\mathbb{Z}^{r}$ by translations) and thus as a group
of automorphisms of $kQ$, with $g(e_{x})=e_{g(x)}$; $g(a_{x,i})=a_{g(x),i}$
and $g(b_{x,i})=b_{g(x),i}$ if $g\in\mathbb{Z}^{r}$; and $\kappa
(a_{x,i})=b_{g(x),i}$ and $\kappa(b_{x,i})=a_{g(x),i}$ if $g\notin
\mathbb{Z}^{r}$.

\bigskip

Let $x=(x_{1},\ldots,x_{r})$, $y=(y_{1},\ldots,y_{r})\in\mathbb{Z}^{r}$. All
paths in $Q$ from $x$ to $y$ of shortest length have the same image in $U_{r}%
$, up to sign, by virtue of the anticommutation relations. In order to be
precise, we make the following

\begin{definition}
Let $\gamma_{xy}$ $\in Q$ be the path of shortest length starting
at $x$, visiting $(y_{1},x_{2},\ldots,x_{n})$,
$(y_{1},y_{2},x_{3},\ldots
,x_{n}),~\ldots,~(y_{1},\ldots,y_{n-1},x_{n})$ in succession and ending at $y$.
\end{definition}

\section{Derived equivalences}
\label{section:derived}

In this section we show how certain mutations of Cubist subsets correspond to
derived equivalences of Cubist algebras.  These `flips' play an important role
in the study of rhombus tilings (see, e.g., \cite[\S2]{LMN}).

\begin{lemma} \label{flipable} Let $\mathcal{X}$ be a Cubist subset of $\mathbb{Z}^r$ and
let $z\in\mathcal{X}$. Then
the following statements are equivalent.
\begin{enumerate}
	\item $z$ is a maximal element of $\mathcal{X}$ with respect to the partial order $\leq$.
	\item For all $x\in\mathbb{Z}^r$, if $z[-1]<x\leq z$, then $x\in\mathcal{X}$.
	\item The subset $\mathcal{X}'$ of $\mathbb{Z}^r$ obtained from $\mathcal{X}$ be replacing 
	        $z$ by $z[-1]$ is a Cubist subset.
        \item The polytope ${\mathcal P}_z$ is an $r-1$-simplex.

\end{enumerate}
\end{lemma}

For most of the remainder of this section we fix a Cubist subset $\mathcal{X}\subset\mathbb{Z}^r$
and an element $z\in\mathcal{X}$ satisfying the conditions of
the Lemma. Our aim is the following result

\begin{theorem} \label{mutation theorem}
There exists an equivalence of triangulated categories%
\[
F:D^{b}(U_{\mathcal{X}^{\prime}}\mMod)\iso D^{b}%
(U_{\mathcal{X}}\mMod)
\]
such that $F(X\langle n \rangle)\cong F(X)\langle n \rangle$ for
$X\in D^{b}(U_{\mathcal{X}^{\prime}}\mMod)$ and  $n\in\mathbb{Z}$.
Moreover, for all $x\in\mathbb{Z}^r$, we have 
$F(U_{\mathcal{X}}'e_x)\cong\Gamma_x$, where $\Gamma_x$ is
a complex explicitly described in \S\ref{subsection:tiltingcomplex}.
\end{theorem}

\subsection{Structure around a flip}
 We begin by describing
the structure of $U_{\mathcal{X}}$ near $  z  $.
Given $n\in\bbz$, let
$$[n]_q:=\frac{q^n-q^{-n}}{q-q^{-1}}\in\bbz[q,q^{-1}]$$
be the associated `quantum integer'. Note that $[-n]=-[n]$.
\begin{lemma}
\label{basic}

\begin{enumerate}
\item The element $\zeta_{  z  }=e_{  z  }b_{  z  ,i}a_{  z  -\epsilon_{i},i}e_{  z  }$ in
$U_{\mathcal{X}}$ is independent of $i$.
\item
The elements
\[
\gamma_{ x ,  z  }\zeta_{  z  }^{s}, \qquad z[-1]<x\leq z,\quad 0\leq s\leq r-1-d(  z  , x )
\]
form a basis of $U_{\mathcal{X}}e_{  z  }$, and the elements%
\[
\zeta_{  z  }^{s}\gamma_{  z  , x }, \qquad z[-1]<x\leq z,\quad0\leq s\leq r-1-d(  z  , x )
\]
form a basis of $e_{  z  }U_{\mathcal{X}}$.
\item We have
$C_{U_\calx}(q)_{zy}=q^{-d(z,x)}+q^{-d(z,x)+2}+\ldots+q^{d(z,x)}=
q^{r-1}[r-d(z,x)]_q$.
\item Suppose $z[-1]<x\leq z$ and let
$\rho(\gamma_{  z  , x }) : U_{\mathcal{X}}e_{  z  }\langle d(  z  , x ) \rangle
\rightarrow U_{\mathcal{X}}e_{ x }$
be the homomorphism defined by right multiplication by $\gamma_{z,x}$. Then 
$\operatorname{Hom}_{U_{\mathcal{X}}}(U_{\mathcal{X}}e_{ z }
\langle n\rangle, \operatorname{coker}(\rho(\gamma_{  z  , x })))=0$
 for all
$n\in\mathbb{Z}$.
\end{enumerate}
\end{lemma}


\begin{proof}
Part (1) is a consequence of the Heisenberg relations because $  z  +\epsilon
_{i}\notin\mathcal{X}$ for all $i$. Define a $kQ$-module $W$ with basis
\[
\left\{w_{ x ,s}\mid\: z[-1]<x\leq z,\quad0\leq s\leq
r-1-d(  z  , x )\right\}
\]
with $\deg(w_{ x ,s})= d(z,x) +2s$ and action
\begin{align*}
e_{ y }w_{ x ,s}  &  =%
\begin{cases}
w_{ x ,s} & \text{if } y = x ,\\
0 & \text{otherwise,}%
\end{cases}
\\
a_{ y ,i}w_{ x ,s}  &  =
\begin{cases}
(-1)^{\sigma}w_{ x +\epsilon_{i},s+1} & \text{if } y = x \text{ and }x+\epsilon_i \in {\mathcal X},\\
0 & \text{otherwise},
\end{cases}
\\
b_{ y ,i}w_{ x ,s}  &  =%
\begin{cases}
(-1)^{\sigma}w_{ x -\epsilon_{i},s} & \text{if } y = x \text{, }x-\epsilon_i \in {\mathcal X},
\text{ and }s\neq r-1-d(  z  , x ),\\
0 & \text{otherwise},%
\end{cases}
\end{align*}
where $\sigma=\#\{j\in\{1,\ldots,i-1\}\mid  y_{j}=x_{j}-1\}$. It is easy to
check that the defining
relations of $U_{r}$ hold, and that $W$ is annihilated by $e_x$ for all $ x \notin\mathcal{X}$.
Hence $W$ is a $U_{\mathcal{X}}$-module. There is a unique
homomorphism $\psi:U_{\mathcal{X}}e_{  z  }\longrightarrow W$ such that
$\psi(e_{  z  })=w_{z,0}$. We have $\psi(\gamma_{ x ,  z  }\zeta_{  z  }^{s})=w_{ x ,s}$, so
$\psi$ is surjective.
The dimension of $W$ is
\[
\sum_{z[-1]<x\leq z}(r-d(  z  , x ))=\sum_{j=1}%
^{r}\binom{r}{j}j=r\sum_{j=1}^{r}\binom{r-1}{j-1}=r2^{r-1}.
\]
On the other hand the dimension of $U_{\mathcal{X}}e_{  z  }$ is $\sum
D_{U_{\mathcal{X}}}(1)_{yz}D_{U_{\mathcal{X}}}(1)_{yx}$, where the sum is over
all $ x , y $ such that $  z  , y \in\lambda  x $. This is also $r2^{r-1}$, since $  z  $ is
contained in exactly $r$ facets of $\mathcal{X}$, by Lemma~\ref{flipable}.  We deduce that $\psi$ is an
isomorphism and that the first half of part (2) is true. The second half is
obtained by applying the antiautomorphism $\omega$. Part (3) is follows immediately.
The last part is also a consequence of the second, because
\[
\operatorname*{Hom}\nolimits_{U_{\mathcal{X}}}(U_{\mathcal{X}}e_{  z  }\langle n\rangle
,\operatorname{coker}(\rho(\gamma_{z,x})))\cong\left(  \frac{e_{  z  }U_{\mathcal{X}}e_{ x }}%
{e_{  z  }U_{\mathcal{X}}\gamma_{  z , x  }}\right)  _{n}.
\]
\end{proof}

\subsection{The tilting complex}\label{subsection:tiltingcomplex}
For each $ x \in\mathbb{Z}^{r}$ we define%
\[
\Gamma_{ x }=%
\begin{cases}
\operatorname{cone}(U_{\mathcal{X}}e_{  z  }\langle d(  z  , x ) \rangle\overset
{\rho(\gamma_{  z  , x })}{\longrightarrow}U_{\mathcal{X}}e_{ x }) & \text{if } z[-1]\leq x \leq z,\\
U_{\mathcal{X}}e_{ x } & \text{otherwise}.
\end{cases}
\]
 So each $\Gamma_x$ is a
complex of projective $U_{\mathcal{X}}$-modules concentrated in homological 
degrees $-1$ and $0$. Observe that
 $\Gamma_{  z  }$ is contractible and that
	$\Gamma_{  z  [-1]}$ is isomorphic to $U_{\mathcal{X}}e_{  z  }\langle r\rangle[1]$.
	Hence $\Gamma_z$
is nonzero as an object of $D^{b}(U_{\mathcal{X}}\mMod)$ if and
only if $x \in\mathcal{X}^{\prime}$.

\begin{proposition}
The complexes $\Gamma_x$ satisfy the hypotheses of Theorem \ref{Rickard}.
\end{proposition}

\begin{proof}
The generation condition clearly holds, because $U_{\mathcal{X}}e_{  z  }%
\cong\Gamma_{  z  [-1]}\langle-r\rangle\lbrack-1]$ and, for all $ x \neq   z  $,
$U_{\mathcal{X}}e_{ x }$ is isomorphic to $\Gamma_x$ or to the cone of a morphism from $\Gamma_{ x }$
to $U_{\mathcal{X}}e_{  z  }\langle d(  z  , x ) \rangle\lbrack-1]$. So it remains to prove
that for $ x , x ^{\prime}\in\mathcal{X}$ and $m,n\in\mathbb{Z}$ with $m\neq0$,%
\[
\operatorname*{Hom}\nolimits_{D^{b}(U_{\mathcal{X}}\mMod)}(\Gamma_{ x }\langle
n\rangle,\Gamma_{ x ^{\prime}}[m])=0.
\]
This is clear unless $m=1$ or $m=-1$, since $\Gamma_{ x }$ and $\Gamma
_{ x ^{\prime}}$ are complexes concentrated in degree $-1$ and $0$. Thus it
suffices to show that $\operatorname*{Hom}\nolimits_{U_{\mathcal{X}}}(U_{\mathcal{X}}%
e_{ z }\langle n\rangle,H^{0}(\Gamma_{ x ^{\prime}}))=0$ and $\operatorname*{Hom}%
\nolimits_{U_{\mathcal{X}}}(H^{0}(\Gamma_{ x })\langle n\rangle, U_{\mathcal{X}}e_{  z  })=0$
for all $n\in\mathbb{Z}$.
The first is true by Lemma~\ref{basic} and the second follows because
$U_\mathcal{X}$ is a symmetric algebra (Theorem~\ref{symmetric}).
\end{proof}

We form the graded endomorphism ring
\[
E:=\oplus_{n\in\mathbb{Z}}E_{n},
\qquad
E_{n}=\oplus_{ x , y \in\mathbb{Z}^{r}}\operatorname*{Hom}\nolimits_{D^{b}%
(U_{\mathcal{X}}\mMod)}(\Gamma_{ x }\langle n\rangle,\Gamma_{ y }),
\]
and put
$
e_{ x }^{\prime}=id_{\Gamma_{ x }}$ for each $x \in\mathbb{Z}^{r}$.
By Theorem \ref{Rickard} there exists an equivalence%
\[
F:D^{b}(\operatorname{mod}(E^{op}))\iso D^{b}(\operatorname{mod}%
(U_{\mathcal{X}}))
\]
such that $F(E^{op}e_{ x }^{\prime})\cong\Gamma_{ x }$ for all $ x \in\mathbb{Z}%
^{r}$,
and
$F(X\langle n \rangle) \cong F(X)\langle n \rangle$ for $X\in D^b(E^{op}\mMod)$ and $n\in\mathbb{Z}$.
To complete the proof of Theorem~\ref{mutation theorem} we shall construct
an isomorphism between $U_{\mathcal{X}'}$ and $E^{\operatorname{op}}$.
\subsection{Identification of the endomorphism ring}

We define a graded homomorphism%
\[
\Phi:kQ\longrightarrow E^{\operatorname*{op}}%
\]
as follows:%
\begin{align*}
\Phi(e_{ x })  :  \Gamma_{ x }\rightarrow\Gamma
_{ x } &  =  \operatorname{id}_{\Gamma_{ x }}\quad\\
\Phi(a_{ x ,i})  :  \Gamma_{ x }\langle1\rangle\rightarrow\Gamma
_{ x +\epsilon_{i}} & =%
\begin{cases}
((-1)^{\sigma_{ x ,i}}\rho(\zeta_{  z  }),\rho(a_{ x ,i}))   & \text{if }
z[-1] \,\leq\, x , x +\epsilon_{i}\,\leq\, z,\\
(0,\rho(a_{ x ,i})) &   \text{otherwise},%
\end{cases}
\\
\Phi(b_{ x ,i})   : \Gamma_{ x }\langle1\rangle\rightarrow\Gamma
_{ x -\epsilon_{i}} & =%
\begin{cases}
((-1)^{\sigma_{ x ,i}}\operatorname{id},\rho(b_{ x ,i}))  & \text{if }
z[-1] \,\leq\, x , x -\epsilon_{i}\,\leq\, z,\\
(0,\rho(b_{ x ,i})) &   \text{otherwise},%
\end{cases}
\end{align*}
where $\sigma_{ x ,i}=\#\{j<i\mid  y_{j}\neq x_{j}\}$
and chain maps $(f_{-1},f_0)$ are specified by their components $f_{-1}, f_0$ in degrees $-1$ and $0$.
Using Lemma \ref{basic} it is straightforward to check that these are indeed
chain maps.

\begin{proposition}
$\Phi$ is surjective.
\end{proposition}

\begin{proof}
Let $f:\Gamma_{ x }\langle n\rangle\rightarrow\Gamma_{ y }$ be a chain map. We
want to show that the image of $\Phi$ contains $f$. It
certainly contains a chain map whose degree $0$ component agrees with $f$,
 so by taking their difference and scaling we may assume that 
$z[-1]\,\leq\,x,y\,\leq\,z$,
and that $f=(\rho(\zeta_{  z  }^{s}),0)$ where
$2s=n+d(  z  , x )-d(  z  , y )$. Since $f$ is a chain map, we have $\zeta_{  z  }^{s}%
\gamma_{  z  , y }=0$ which implies by Lemma \ref{basic} that $s\geq r-d(  z  , y )$.
Hence $n\geq2r-d(  z  , x )-d(  z  , y )\geq0$.

If $n=0$ then $f=\Phi(e_{  z  [-1]})$. If $n=1$ then $f=\Phi(a_{  z  [-1],i})$ or
$f=\Phi(b_{  z  [-1]+\epsilon_{i},i})$ for some $i$. Now we assume $n\geq2$ and
argue by induction on $n$. Suppose that $ y =  z  [-1]$. Since $\Gamma_{  z  }$ is
contractible we may assume that  $x \neq   z  $ and choose $i$ such that
$z[-1]\leq x +\epsilon_{i}\leq   z$.
 Then $f$ is the composition of
$\Phi(a_{ x ,i})\langle n-1\rangle$ and $(\rho(\zeta_{  z  }^{s}),0):\Gamma
_{ x +\epsilon_{i}}\langle n-1\rangle\rightarrow\Gamma_{  z  [-1]}$.

Suppose on the other hand that $ y \neq   z  [-1]$. Choose $i$ such that
$z[-1]\leq y -\epsilon_{i}\leq  z$.
 Then $f$ is the composition of $(\rho
(\zeta_{  z  }^{s-1}),0):\Gamma_{ x }\langle n\rangle\rightarrow\Gamma
_{ y -\epsilon_{i}}\langle1\rangle$ and $\Phi(b_{ y -\epsilon_{i},i})$. The former
is a chain map because $s-1\geq r-d(  z  , y -\epsilon_{i})$.
\end{proof}

\begin{proposition}
$\Phi$ factors through the natural homomorphism $kQ\rightarrow
U_{\mathcal{X}^{\prime}}$.
\end{proposition}

\begin{proof}
Since $\Phi(e_{ x })=id_{\Gamma_{ x }}=0$ for $ x \notin\mathcal{X}^{\prime}$, it
suffices to show that $\Phi$ kills the defining relations of $U_{r}$. We shall
show that in fact the image of any relation under $\Phi$ is the zero chain map
(not merely nullhomotopic); note that this is clear in homological degree $0$.

\begin{itemize}
\item Square relations \ref{U1}: For any  $x \in\mathbb{Z}^{r}$, at least one of $ x $,
$ x +\epsilon_{i}$, and $ x +2\epsilon_{i}$ is not in $  z  -\{0,1\}^{r}$ and hence
$\Phi(a_{ x ,i}a_{ x +\epsilon_{i},i})=0$. A similar argument applies to
$\Phi(b_{ x ,i}b_{ x -\epsilon_{i},i})$.

\item Supercommutation relations \ref{U2}: Consider $\Phi(a_{ x ,i}a_{ x +\epsilon_{i}
,j}+a_{ x ,j}a_{ x +\epsilon_{j},i})$. We may assume that both $\Gamma_x$ and
$\Gamma_{ x +\epsilon_{i}+\epsilon_{j}}$are nonzero in degree $-1$, and
therefore that $ x , x +\epsilon_{i}, x +\epsilon_{j}, x +\epsilon_{i}+\epsilon_{j}\in
  z  -\{0,1\}^{r}$. Then the component of $\Phi(a_{ x ,i}a_{ x +\epsilon_{i}%
,j}+a_{ x ,j}a_{ x +\epsilon_{j},i})=0$ in degree $-1$ is right multiplication by%
\[
\left((-1)^{\sigma_{ x ,i}+\sigma_{ x +\epsilon_{i},j}}+(-1)^{\sigma_{ x ,j}%
+\sigma_{ x +\epsilon_{j},i}}\right)\zeta_{  z  }^{2}=0.
\]
The argument that $\Phi(b_{ x ,i}b_{ x +\epsilon_{i},i}+b_{ x ,j}b_{ x +\epsilon
_{j},i})=0$ and $\Phi(a_{ x ,i}b_{ x +\epsilon_{i},j}+b_{ x ,j}a_{ x -\epsilon_{j}%
,i})=0$ is similar.

\item Heisenberg relations \ref{U3}: If $ x \notin   z  -\{0,1\}^{r}$, then the degree $-1$
component of $\Phi(a_{ x ,i}b_{ x +\epsilon_{i},i}+b_{ x ,i}a_{ x -\epsilon_{i},i})$
is $0$. If $ x \in   z  -\{0,1\}^{r}$ then exactly one of  $x +\epsilon_{i}$ and
$ x -\epsilon_{i}$ is in $  z  -\{0,1\}^{r}$, and therefore the degree $-1$
component of $\Phi(a_{ x ,i}b_{ x +\epsilon_{i},i}+b_{ x ,i}a_{ x -\epsilon_{i},i})$
is $\rho(\zeta_{  z  })$, which does not depend on $i$.
\end{itemize}
\end{proof}

By virtue of these two propositions we have a surjective homomorphism
$U_{\mathcal{X}^{\prime}}\rightarrow E^{^{\operatorname{op}}}$. We now
show that this is actually an isomorphism, by demonstrating that
$C_{U_{\mathcal{X}^{\prime}}}(q)=C_{E^{\operatorname{op}}}(q)$.

\begin{lemma}
\label{tiltcartan}We have%
\[
C_{E^{\operatorname{op}}}(q)_{ x , y }=%
\begin{cases}%
C_{U_{\mathcal{X}}}(q)_{ x , y }-q^{r-1}[r-d(  z  , x )-d(  z  , y )]_{q} & \text{if
} z[-1]\leq x \leq      z,\\
C_{U_{\mathcal{X}}}(q)_{ x , y } & \text{otherwise}.%
\end{cases}
\]

\end{lemma}

\begin{proof}
By Lemma~\ref{basic},
\[
C_{U_{\mathcal{X}}}(q)_{  z  , x }=%
\begin{cases}
q^{r-1}[r-d(  z  , x )]_{q}   & \text{if } z[-1]\leq x \leq      z,\\
0 &   \text{otherwise},%
\end{cases}
\]
and since $\omega(e_{ x }U_{\mathcal{X}}e_{  z  })=e_{  z  }U_{\mathcal{X}}e_{ x }$, we
have $C_{U_{\mathcal{X}}}(q)_{ x ,  z  }=C_{U_{\mathcal{X}}}(q)_{  z  , x }$.

Hence%
\begin{align*}
C_{E^{\operatorname{op}}}(q)_{ x , y }  &  =\sum_{n\in\mathbb{Z}}\dim
\operatorname{Hom}_{D^{b}(U_{\mathcal{X}}\mMod
)}(\Gamma_{ x }\langle n\rangle,\Gamma_{ y })~q^{n}\\
&  =\sum_{n\in\mathbb{Z}}\dim\operatorname*{Hom}\nolimits_{U_{\mathcal{X}}%
}(U_{\mathcal{X}}e_{ x }\langle n\rangle,U_{\mathcal{X}}e_{ y })~q^{n}\\
&\qquad  +\sum_{n\in\mathbb{Z}}\dim\operatorname*{Hom}\nolimits_{U_{\mathcal{X}}%
}(U_{\mathcal{X}}e_{  z  }\langle d(  z  , x )+n\rangle,U_{\mathcal{X}}e_{  z  }\langle
d(  z  , y )\rangle)~q^{n}\\
&\qquad  -\sum_{n\in\mathbb{Z}}\dim\operatorname*{Hom}\nolimits_{U_{\mathcal{X}}%
}(U_{\mathcal{X}}e_{ x }\langle n\rangle,U_{\mathcal{X}}e_{  z  }\langle
d(  z  , y )\rangle)~q^{n}\\
&\qquad  -\sum_{n\in\mathbb{Z}}\dim\operatorname*{Hom}\nolimits_{U_{\mathcal{X}}%
}(U_{\mathcal{X}}e_{  z  }\langle d(  z  , x )+n\rangle,U_{\mathcal{X}}e_{ y })~q^{n}.
\end{align*}
So if $z[-1]\,\leq\, x , y \,\leq\,  z$, then
\begin{align*}
C_{E^{\operatorname{op}}}(q)_{ x , y }  &  =C_{U_{\mathcal{X}}}(q)_{ x , y }%
+q^{r-1+d(  z  , y )-d(  z  , x )}[r]_{q}\\
&\qquad -q^{r-1+d(  z  , y )}[r-d(  z  , x )]_{q}-q^{r-1-d(  z  , x )}%
[r-d(  z  , y )]_{q}\\
&  =C_{U_{\mathcal{X}}}(q)_{ x , y }-q^{r-1}[r-d(  z  , x )-d(  z  , y )]_{q},
\end{align*}
and otherwise $C_{E^{\operatorname{op}}}(q)_{ x , y }=C_{U_{\mathcal{X}}}%
(q)_{ x , y }$.
\end{proof}

By using the automorphism $\kappa$ we see that the results of this section
apply to a dual situation in which we specify an element of a Cubist subset
minimal with respect to $\leq$. Taking in particular the Cubist subset
$\mathcal{X}^{\prime}$ and the minimal element $  z  [-1]$, we obtain a graded
endomorphism ring $E^{\prime}$ together with an epimorphism $U_{\mathcal{X}%
}\rightarrow(E^{\prime})^{\operatorname{op}}$, and the formula
\[
C_{(E^{\prime})^{\operatorname{op}}}(q)_{ x , y }=%
\begin{cases}
C_{U_{\mathcal{X}^{\prime}}}(q)_{ x , y }-q^{r-1}[r-d(  z  [-1], x )-d(  z  [-1], y )]_{q} 
& \text{if } z[-1]\,\leq\, x , y \,\leq\, z,\\
C_{U_{\mathcal{X}^{\prime}}}(q)_{ x , y }   & \text{otherwise}.%
\end{cases}
\]
There is an equivalence of categories,
\[
F':D^{b}((E^{\prime})^{op}\mMod))\iso D^{b}%
(U_{\mathcal{X}'}\mMod).
\]
Note that $r-d(  z  [-1], x )-d(  z  [-1], y )=d(  z  , x )+d(  z  , y )-r$.

We are now ready to show that $C_{U_{\mathcal{X}^{\prime}}}(q)_{xy}%
=C_{E^{\operatorname{op}}}(q)_{xy}$ for all $ x,y \in\mathbb{Z}^{r}$, and thus
complete the proof of Theorem~\ref{mutation theorem}. Because
$\Gamma_{  z  }$ is contractible we may assume that   $x \neq   z  $ and $ y \neq   z  $. If
$z[-1]\,\leq\, x , y \,\leq\, z$, then%
\begin{align*}
C_{E^{\operatorname{op}}}(q)_{ x , y }  &  =C_{U_{\mathcal{X}}}(q)_{ x , y }%
+q^{r-1}[r-d(  z  , x )-d(  z  , y )]_{q}\\
&  \geq C_{(E^{\prime})^{\operatorname{op}}}(q)_{ x , y }+q^{r-1}%
[r-d(  z  , x )-d(  z  , y )]_{q}\\
&  =C_{U_{\mathcal{X}^{\prime}}}(q)_{ x , y }\\
&  \geq C_{E^{\operatorname{op}}}(q)_{ x , y },
\end{align*}
where $\geq$ means an inequality holds for each pair of corresponding coefficients. We deduce that the inequalities
are actually equalities.

\begin{remark}
By Koszul duality, the derived categories of $U_{\mathcal X} \mMod, V_{\mathcal X} \mMod$ are equivalent.
Therefore, there also exists an equivalence of triangulated categories%
\[
D^{b}(V_{\mathcal{X}^{\prime}}\mMod)\iso D^{b}%
(V_{\mathcal{X}}\mMod).
\]
\end{remark}

\begin{remark}
The simple $V_{\mathcal X}$-module $L(z)$ 
has extension algebra $e_z U_{\mathcal X} e_z \cong k[\zeta_z]/\zeta_z^r \cong H^*(\mathbb{P}^{r-1})$ in
the category of ungraded $V_{\mathcal X}$-modules.
Objects whose extension algebras are isomorphic to $H^*(\mathbb{P}^n)$
are called $\mathbb{P}^n$-objects by D. Huybrechts and R. Thomas
~\cite{Huybrechts-Thomas}, and give rise to self-equivalences 
whenever they appear in the derived category of a smooth projective variety.

In our setting,  
self-equivalences of $D^b(U_{\mathcal X}\mMod)$ can be obtained by composing equivalences,
$$F\circ F': D^b(U_{\mathcal X}\mMod) \iso D^b(U_{\mathcal X'}\mMod) \iso D^b(U_{\mathcal X}\mMod).$$
\end{remark}

\subsection{Another formula for the entries of the graded Cartan matrix of $U_{\mathcal{X}}$}

We can now derive a
different formula for the entries of $C_{U_{\mathcal{X}}}(q)$, one which
isn't tied to a choice of quasi-hereditary structure on $U_{\mathcal{X}}$.

For any $  x  \in\mathcal{X}$ we define%
\[
I_{\mathcal{X}}(  x  )=\{ z \in\mathcal{X}\mid  x\leq z \leq x[1]\}.
\]

\begin{lemma} \label{Cartan locality}
Let $\mathcal{X}_{1}$ and $\mathcal{X}_{2}$ be Cubist subsets containing $  x  $,
and suppose that $I_{\mathcal{X}_{1}}(  x  )=I_{\mathcal{X}_{2}}(  x  )$. Then for all
$ y \in\mathbb{Z}^{r}$, we have $C_{U_{\mathcal{X}_{1}}}(q)_{  xy  }%
=C_{U_{\mathcal{X}_{2}}}(q)_{  xy  }$.
\end{lemma}

\begin{proof} Let $\calx\subset\bbz^r$
be a Cubist subset.
Using the fact that $U_{\mathcal{X}}$ is quasihereditary (Corollary~\ref{Uqh})
and the formula for its decomposition matrix given in Lemma~\ref{Ustandard},
we have
$$
C_{U_{\mathcal{X}}}(q)_{  xy  }
 = \sum_{\substack{z\in\calx \\ x,y\in\lambda z}}q^{d(x,z)+d(y,z)} 
 = \sum_{\substack{F\in\mathcal{F}_\calx \\ x,y\in F}}
q^{d(x,\lambda^{-1}(F))+d(y,\lambda^{-1}(F))}. 
$$
Hence it suffices to
show that whether or not a facet $F\in\mathcal{F}$ containing $  x  $ is contained in
$\mathcal{X}$ depends only on $I_{\mathcal{X}}(  x  )$.
For some $S\subseteq\{1,\ldots,r\}$ and $i\notin S$, the facet $F$ 
consists of all $x'\in\bbz^r$ such that 
$x'\leq
x  +\sum_{j\in S}a_{j}\epsilon_{j}$
and
$x'\geq x-
\sum_{j\notin S\cup\{i\}}
a_{j}\epsilon_{j}$.
Hence
$F$ is
contained in $\mathcal{X}$ if and only if 
$  x  +\sum_{j\in S}\epsilon_{j}\in\mathcal{X}$ and $  x  +\sum_{j\in S\cup
\{i\}}\epsilon_{j}\notin\mathcal{X}$.
\end{proof}

\begin{proposition}\label{alternativeformula}
Let $\mathcal{X}$ be a Cubist subset. Then for all $  x  , y \in\mathcal{X}$, we
have%
\[
C_{U_{\mathcal{X}}}(q)_{  x   y  }=\sum_{ z \in I_{\mathcal{X}}(x)
\cap I_{\mathcal{X}}(y)}q^{r-1}[r-d( z ,  x  )-d( z , y )]_{q}.
\]

\end{proposition}

\begin{proof}
We induct on $|I_{\mathcal{X}}(  x  )|.$ If $I_{\mathcal{X}}(  x  )=\{  x  \}$, then  the sum on the right hand side of the desired equality contains a
single term $q^{r-1}[r-d(  x  , y )]_{q}$ if $y\leq x\leq y[1]$ and is zero otherwise. This is in agreement with part (3) of Lemma~\ref{basic}.

Now suppose $\left\vert I_{\mathcal{X}}(  x  )\right\vert >1$. By Lemma~\ref{Cartan locality}
we may assume that $\mathcal{X}=\mathcal{X}^{-}\setminus\mathcal{X}^{-}[-1]$,
where $\mathcal{X}^{-}$ is the ideal in $\mathbb{Z}^r$ generated by $I_{\mathcal{X}}(  x  )$. Choose
an element $v\in I_{\mathcal{X}}(  x  )$ maximal with respect to $\leq$. Then $v$ is also
a maximal element of $\mathcal{X}$ with respect to $\leq$. By Lemma~\ref{flipable} the
subset $\mathcal{X}'$ of $\mathbb{Z}^r$ obtained from $\mathcal{X}$ by
replacing $v$ by $v[-1]$ is Cubist.  We have
$I_{\mathcal{X}^{\prime}}(  x  )=I_{\mathcal{X}%
}(  z  )\setminus\{v\}$, so by induction the stated formula holds for
$C_{U_{\mathcal{X}^{\prime}}}(q)_{  xy  }$. Hence for all
 $ y \in\mathcal{X}$ we have
\begin{align*}
C_{U_{\mathcal{X}}}(q)_{  x   y  }  &  =%
\begin{cases}
C_{U_{\mathcal{X}^{\prime}}}(q)_{  x   y  }+q^{r-1}[r-d(v,  x  )-d(v, y )]_{q} &  
\text{if } v\in I_{\mathcal{X}}(x),\\
C_{U_{\mathcal{X}^{\prime}}}(q)_{  x   y  } &   \text{otherwise}%
\end{cases}
\\
&  =\sum_{ z \in I_{\mathcal{X}}(x)
\cap I_{\mathcal{X}}(y)}q^{r-1}
[r-d( z ,  x  )-d( z , y )]_{q},
\end{align*}
where the first equality is by Lemma \ref{tiltcartan}, and the second by induction.
\end{proof}

\section {Formulae}

We assemble nine elegant formulae, which combine to give purely combinatorial relations.
It seems difficult to imagine how such expressions could have been conceived, 
without the Cubist algebras.  

Note that the formulae involving decomposition matrices hold for any of the $r!$ possible highest weight structures,
but depend on the given highest weight structure.
The formulae which do not involve decomposition matrices are independent of highest weight structure.

\begin{theorem}
Combinatorial formulae for decomposition matrices:
$$D_{U_{\mathcal X}}(q)_{xy} = \sum_{z \in \lambda x} \delta_{zy} q^{d(z,x)},$$
$$D_{V_{\mathcal X}}(q)_{xy} = \sum_{z \in \mu x} \delta_{zy} q^{d(z,x)}.$$
Combinatorial formulae for Cartan matrices:
$$C_{U_{\mathcal{X}}}(q)_{  x   y  }=\sum_{ z \in I_{\mathcal{X}}(x)
\cap I_{\mathcal{X}}(y)}q^{r-1}[r-d( z ,  x  )-d( z , y )]_{q}.$$
$$C_{V_{\mathcal X}}(q)_{xy} = (1-q^2)^{1-r} q^{d(x,y)},$$
Brauer formulae for Cartan matrices:
$$C_{U_{\mathcal X}}(q) = D_{U_{\mathcal X}}(q)^T D_{U_{\mathcal X}}(q),$$
$$C_{V_{\mathcal X}}(q) = D_{V_{\mathcal X}}(q)^T D_{V_{\mathcal X}}(q),$$
Transpose formulae:
$$C_{U_{\mathcal X}}(q) = C_{U_{\mathcal X}}(q)^T.$$
$$C_{V_{\mathcal X}}(q) = C_{V_{\mathcal X}}(q)^T.$$
Inverse formulae:
$$D_{U_{\mathcal X}}(q)^T.D_{V_{\mathcal X}}(-q) = 1.$$
$$C_{U_{\mathcal X}}(q).C_{V_{\mathcal X}}(-q) = 1.$$
Symmetry formula:
$$C_{U_{\mathcal X}}(q^{-1}) = q^{2-2r} C_{U_{\mathcal X}}(q).$$
\end{theorem}
\begin{proof}

The combinatorial formulae for the decomposition matrices were proved during our study of highest 
weight categories (corollary~\ref{Vdec}, lemma~\ref{Ustandard}). 
The combinatorial formula for $C_{U_{\mathcal X}}(q)$ was proved in the last section
(proposition~\ref{alternativeformula}). The combinatorial formula for $C_{V_{\mathcal X}}(q)$ is recorded as
remark~\ref{CartanV}.

The Brauer formulae for Cartan matrices in terms are abstract consequences of 
$U_{\mathcal X}\mMod, V_{\mathcal X}\mMod$ being highest weight categories 
with duality \cite[Theorem 3.1.11]{MR961165}.
The transpose formulae follow immediately.

The inverse formula relating the decomposition matrices of $U_{\mathcal X}, V_{\mathcal X}$ was proven as lemma~\ref{formula}.
The inverse formula relating the Cartan matrices of $U_{\mathcal X}, V_{\mathcal X}$ is an abstract consequence 
of $U_{\mathcal X}, V_{\mathcal X}$ being Koszul dual (\cite{MR1322847}, Theorem 2.11.1).

The symmetry formula holds because $U_{\mathcal X}$ is symmetric.
\end{proof}

\section{Interpretation}

The algebras $U_{\mathcal X}$ look like blocks of finite group algebras.
Let us detail this metaphor. 

\bigskip

Consider the diagram,  

$$ \textrm{ Blocks of finite groups }  \longrightarrow \textrm{ Abelian Categories }$$
$$\searrow \hspace{3cm} \swarrow$$
$$\textrm{ Triangulated categories }.$$

Here the horizontal arrow describes a functor $\Phi$, which takes $B$ to its module category $B\mMod$.
The southwest pointing functor $\Psi$ carries an abelian category to its derived category.
The southeast pointing functor $\Upsilon$ takes a block to its derived category.
Here, the categories are considered up to equivalence.
We have the following vague conjectures:

\begin{conjecture} \label{vague}

1. The image of $\Phi$ is small.

2. The image of $\Upsilon$ is very small.
\end{conjecture}
 
We should be more precise. Let $k$ have characteristic $p$.
Let $P$ be a $p$-group, and $B_P$ be the set of blocks with defect group $P$.
Let $b$ be a block of some group in which $P$ is normal, and let $B_b$ be the set of blocks whose 
Brauer correspondent is Morita equivalent to $b$.

\begin{conjecture} \label{cat}

1. (P. Donovan) For any $P$, $| im (\Phi_{B_P}) |  < \infty$.

2. (M. Brou\'e) For abelian $P$, $| im (\Upsilon_{B_b})|  = 1$.
\end{conjecture}

Specialising to symmetric groups, both parts of conjecture~\ref{cat} are theorems.
Let $B_{\Sigma,w}$ be the set of blocks of symmetric
groups of weight $w$. 
 
\begin{theorem} \label{proven}

1. (J. Scopes) $| im(\Phi_{B_{\Sigma,w}})|  < \infty$.

2. (J. Chuang, R. Rouquier) $| im(\Upsilon_{B_{\Sigma,w}})|     = 1.$
\end{theorem}

Various investigations into blocks of symmetric groups suggest the following 
parallel to theorem~\ref{proven}:

\begin{conjecture} \label{sting} Let $w< p$.
$$im(\Phi_{B_{\Sigma,w}}) \subset$$ 
$$\left\{ 
\begin{array}{cc}
e{\mathcal A}e \mMod,
& {\mathcal A} \textrm{ a standard Koszul, symmetric algebra,} \\ 
& \textrm{graded in degrees } 0,1,\ldots,2w, \\
& e \textrm{ an idempotent in } {\mathcal A} \\
\end{array}
\right\}.$$
\end{conjecture}

\bigskip

The relation between the algebras $U_{\mathcal X}$ and blocks of symmetric groups should now be clear. 
We have proved the $U_{\mathcal X}$'s possess all of the
strong properties of the algebras ${\mathcal A}$ of the above conjecture, as well as various refinements
of these properties.
We have also revealed a multitude of derived equivalences between $U_{\mathcal X}$'s, as one expects
to find between blocks of finite groups.

In fact, the similarity between blocks of symmetric groups and Cubist algebras appears to be 
more than merely formal. In the following section, we make a precise connection between the algebras 
$U_{\mathcal X}$ in case $r=3$, and symmetric group blocks of defect $2$.
 
It would be interesting if there were generalisations of conjecture \ref{sting} 
to other families of finite groups. The stated conjecture already appears to be quite deep.

\begin{remark}
Let $G$ be a finite group.
Let $\underline{\chi}$ be the submatrix of the ordinary character table of $G$, whose columns are 
indexed by $p$-regular elements. 
The matrix $\underline{\mathcal B}$ of $p$-Brauer characters of $G$, is related to $\underline{\chi}$ by the formula,
$$\underline{\chi} = D.\underline{\mathcal B},$$
where $D$ is the $p$-decomposition matrix of $G$. Therefore, to
compute $\underline{\mathcal B}$ from $\underline{\chi}$, one uses the formula
$\underline{\mathcal B} = D^{-1} \underline{\chi}$, where $D^{-1}$ is a left inverse of $D$.

It is always easier to find the ordinary character table of $G$ than the table of Brauer characters. 
Therefore, from a computational point of view, $D^{-1}$ is more important
than $D$ itself. 

For the algebra $U_{\mathcal X}$, 
we know exactly what the inverse of the decomposition matrix is:
it is the $q$-decomposition matrix of $V_{\mathcal X}$, evaluated at $q=-1$.

For the algebras ${\mathcal A}$ which we expect to control blocks of symmetric groups of abelian defect,  
the same situation ought to arise. 
The Koszul dual of such an algebra will have a $q$-decomposition matrix. 
It is the evaluation of this matrix at $q=-1$ which should allow for the direct
computation of the Brauer character table $\underline{\mathcal B}$ of a related block from the 
ordinary character table $\underline{\chi}$ of the relevant symmetric group.
\end{remark}

\section{Symmetric group blocks and rhombal algebras}

\subsection{Overview}
Here we establish a direct connection between rhombal algebras and some blocks
of symmetric groups, making more precise and complete the observations of
Michael Peach \cite[\S 4]{Peach}. Let $B$ be a weight $2$ block of a
symmetric group in characteristic $p\neq2$. Then $B$ has $\frac{1}{2}\left(
p-1\right)  \left(  p+2\right)  $ simple modules. In this section we will
prove the following result.

\begin{theorem}
There exists a Cubist subset $\mathcal{X}\subset\mathbb{Z}^{3}$, an idempotent
$e\in U_{\mathcal{X}}$ and an idempotent $f\in B$ such that 
$eU_{\mathcal{X}}e$ and $fBf$ are isomorphic as algebras, each having
$\frac{1}{2}\left(  p-1\right)  p$ simple modules.
\end{theorem}

Actually we shall obtain a more precise result, in which $\mathcal{X}$ is
described explicitly in terms of the combinatorics associated to
$B$. The strategy of the proof is to first construct an isomorphism directly
for a special class of blocks, the Rouquier blocks, which are known to have a
description in terms of wreath products 
\cite{Chuang-Kessar}.
 Then the result is
extended to all blocks using the derived equivalences between Cubist algebras
in \S\ref{section:derived}, together with known equivalences between blocks of symmetric
groups: the Morita equivalences of Scopes \cite{Scopes-CMMEBSG}
and the derived equivalences of Rickard \cite{Rickard-MSRI}.

\subsection{Blocks of symmetric groups} \label{subsection:blocks}

We begin by sketching the combinatorics of the block theory of the
symmetric groups, referring the reader to the standard references \cite{James1978}
and \cite{Mathasbook} for more details. Let $\mathfrak{S}_{n}$ be the
symmetric group of degree $n$ and let $k$ be a field of characteristic $p$.
The simple modules $D^{\lambda}$ of $k\mathfrak{S}_{n}$ are parametrized by
$p$-regular partitions $\lambda$ of $n$.

 We describe a method due to Gordon
James \cite{Jamesabacus} for representing partitions which is useful in this context.
We consider an abacus with $p$ vertical half-infinite runners, with positions
labelled $0,1,\ldots$ from left to right and then top to bottom. Thus the
positions on the $i$-th runner are labelled $i,i+p,i+2p,\ldots$. Given any
partition $\lambda=(\lambda_{1},\lambda_{2},\ldots)$ of $n$, an abacus
representation of $\lambda$ is obtained by placing $N$ beads in positions
$\lambda_{1}+N-1,~\lambda_{2}+N-2,~\ldots,~\lambda_{N}$, where $N$ is any integer at least
as big as the number of parts of $\lambda$. By moving all the beads as far up
their runners as possible one obtains an abacus representation of a partition
of $n-wp$ for some $w\geq0$. This partition is the \emph{core} of $\lambda$
and $w$ is the \emph{weight} of $\lambda$. The simple modules $D^{\lambda}$
and $D^{\mu}$ belong to the same block of $k\mathfrak{S}_{n}$ if and only if
$\lambda$ and $\mu$ have the same core (and therefore the same weight). This
statement, known as Nakayama's Conjecture, allows us to assign a core $\kappa$
and weight $w$ to each block $B$ of $k\mathfrak{S}_{n}$. Any two
blocks of the same weight have the same number of
simple modules. We denote by $\Lambda_{B}$ the set of all partitions (of $n$)
with core $\kappa$ and weight $w$. Then the simple modules $D^{\lambda}$ of
$B$ are indexed by the $p$-regular partitions in $\Lambda_{B}$.

 We choose
mutually orthogonal idempotents $f_{\lambda}$ $\in B$ such that $Bf_{\lambda}$
is a projective cover of $D^{\lambda}$. Then $(\sum f_{\lambda})B(\sum
f_{\lambda})$ is a basic algebra Morita equivalent to $B$.

\subsection{Weight $2$ blocks}\label{subsection:weight2blocks}
We now assume that $p>2$ and restrict our attention to blocks of weight $2$.
This class of blocks has been well studied. Peach and our work on rhombal algebras has been partly inspired by the general results of Scopes \cite{Scopes-SGBDT} and Tan \cite{Tan-thesis},
the determination of decomposition numbers by Richards
\cite{Richards-SDNHAGLG} and the calculation of
quivers and relations by Erdmann and Martin \cite{
Erdmann-Martin-QRPPS} and by
Nebe \cite{Nebe}.

We shall describe a natural parametrization of the simple modules in any block
of weight $2$ by the set%
\[
\mathcal{\mathcal{S}}=\{(u,v)\in\mathbb{Z}^{2}\mid 0\leq u\leq v\leq
p-1,\;(u,v)\neq(0,0)\}.
\]
The simple modules corresponding to the subset%
\[
\mathcal{P}=\{(u,v)\in\mathbb{Z}^{2}\mid 0\leq u<v\leq p-1\}
\]
will survive in a truncation of the block which will be shown to
be isomorphic to a truncation of a rhombal algebra.

Let $B$ be a block of $k\mathfrak{S}_{n}$ of weight $2$. Consider an abacus
representation of the associated $p$-core partition $\kappa$. Let
$q_{0},\ldots,q_{p-1}$ be the first unoccupied positions in each of the $p$
runners, relabelled so that $q_{0}<\ldots<q_{p-1}$, and define the \emph{pyramid} of $B$ to
be%
\[
\mathcal{\mathcal{P}}_{B}=\{(u,v)\in\mathcal{\mathcal{P}}\mid  q_{v}%
-q_{u}<p\}.
\]
This is a corruption of a notion introduced by Matthew Richards \cite{Richards-SDNHAGLG};
his pyramid, defined for any weight, contains the same information in weight
$2$ as ours. Richards proves that if $(u,v)\in\mathcal{P}_{B}$, then
$(u,w),(w,v)\in\mathcal{P}_{B}$ whenever $u<w<v$, and that any subset of
$\mathcal{P}$ with this property is equal to $\mathcal{P}_{B}$ for some block
$B$ of weight $2$. We also define%
\begin{align*}
\mathcal{S}_{B}  &  =\{(u,v)\in\mathcal{\mathcal{S}}\mid  q_{v}-q_{u}<p\}.\\
&  =\mathcal{P}_{B}\cup\{(u,u)\in\mathbb{Z}^{2}\mid 1\leq u\leq p-1\}.
\end{align*}

\begin{example} \label{weight2exampleA}
Let $k$ be a field of characteristic $7$, and let $B$ be the block
of $k\mathfrak{S}_{42}$ of weight $2$ corresponding to the $7$-core partition
$\kappa=(12,6,6,1,1,1,1)$. Choosing $N=7$ we get place beads on the abacus in positions $1,2,3,4,10,12,18$. The first unoccupied positions on each of the runners are, from left to right, $0,8,9,17,25,5,6$. Hence
 $$(q_0,q_1,q_2,q_3,q_4,q_5,q_6)=(0,5,6,8,9,17,25),$$ 
$$\mathcal{P}_B=\left\{(0,1),(0,2),(1,2),(1,3),(1,4),(2,3),(2,4),(3,4)\right\},$$
$$\mathcal{S}_B=\left\{(1,1),(2,2),(3,3),(4,4),(5,5),(6,6),
(0,1),(0,2),(1,2),(1,3),(1,4),(2,3),(2,4),(3,4)\right\}.$$

\end{example}

We shall use a variation of a shorthand, due to Scopes \cite{Scopes-CMMEBSG, Scopes-SGBDT}, for
labelling the partitions in $\Lambda_{B}$:
\begin{itemize}
\item $\left\langle u,v\right\rangle $ for the partition (whose abacus display
is) obtained (from the abacus display of $\kappa$) by moving the beads at
positions $q_{u}-p$ and $q_{v}-p$ down one position, i.e. to positions $q_{u}$
and $q_{v}$. Here $u\neq v$.
\item $\left\langle u\right\rangle $ for the partition obtained by moving the
bead at $q_{u}-p$ to $q_{u}+p$.
\item $\left\langle u,u\right\rangle $ for the partition obtained by moving
the bead at $q_{u}-2p$ to $q_{u}-p$, and the beas at $q_u-p$ to $q_u$.
\end{itemize}
The same set of shorthands labels is used for all blocks of weight $2$. However the subset of labels that corresponds to $p$-regular partitions in $\Lambda_{B}$ depends on $B$.

Scopes \cite{Scopes-CMMEBSG} considers pairs of blocks related to each other by `swapping adjacent
runners'. Even though her results are valid for blocks of arbitrary weight, here we just describe
the weight $2$ case.
Suppose that there exists an abacus display of $\kappa$ and $0\leq
s<t\leq p-1$ such that $q_{t}-q_{s}=mp+1$, where $m>0$. Then by moving the
beads from positions $q_{t}-p,~q_{t}-2p,\ldots,~q_{t}-mp$, to the unoccupied
positions $q_{t}-p-1,~q_{t}-2p-1,\ldots,~q_{t}-mp-1$ we obtain the abacus
display of a $p$-core partition $\bar{\kappa}$. Then $B$ and the block
$\bar{B}$ of weight $2$ with $p$-core $\bar{\kappa}$ are said to form a
$[2:m]$ pair. It is easy to describe the relationship between the pyramids of
$B$ and $\bar{B}$: if $m$ $\geq2$ then $\mathcal{P}_{\bar{B}}=\mathcal{P}_{B}%
$, and if $m=1$ then $\mathcal{P}_{\bar{B}}$ is the disjoint union of
$\mathcal{P}_{B}$ and $\left\{  \left(  s,t\right)  \right\}  $. For an
arbitrary block $B$ of weight $2$, there exists a sequence $B_{0},\ldots
,B_{l}$ of blocks of weight $2$ such that $\mathcal{P}_{B_{0}}=\emptyset$,
$B_{l}=B$ and for $i=1,\ldots,l$, the blocks $B_{i-1}$ and $B_{i}$ form a
$[2:m]$ pair for some $m$.

By Scopes \cite{Scopes-SGBDT}, there exists a bijection
\[
\Phi=\Phi_{B,\bar{B}}:\Lambda_{B}\iso\Lambda_{\bar{B}}%
\]
such that
\begin{itemize}
\item $\Phi(\lambda)$ is $p$-regular if and only if $\lambda$ is $p$-regular,
\item $\Phi(\lambda)$ and $\lambda$ have the same shorthand notation, except
in the following cases when $m=1$:%
\begin{align*}
\Phi\left(  \left\langle t,t\right\rangle \right)   &  =\left\langle
s\right\rangle ,\\
\Phi\left(  \left\langle s,t\right\rangle \right)   &  =\left\langle
t,t\right\rangle ,\\
\Phi\left(  \left\langle s\right\rangle \right)   &  =\left\langle
s,t\right\rangle .
\end{align*}
\end{itemize}
In case $m=1$, we are extending Scopes' Definition 3.4 in \cite{Scopes-SGBDT} by taking,
in her notation, $\Phi(\alpha)=\bar{\alpha}$, $\Phi(\beta)=\bar{\gamma}$, and
$\Phi(\gamma)=\bar{\beta}$.

We now produce the promised parametrization of simple modules in blocks of
weight $2$.

\begin{proposition}\label{parametrise}
\begin{enumerate}
\item Let $B$ be a block of weight $2$. Then the map%
\[
\lambda_{B}:\mathcal{S\rightarrow}\Lambda_{B}%
\]
defined by%
\[
\lambda_{B}(u,v)=%
\begin{cases}%
\left\langle u+1,v\right\rangle  &   \text{if }    (u,v)\notin
\mathcal{S}_{B}\text{ and }(u+1,v)\notin\mathcal{S}_{B}\\
\left\langle v,v\right\rangle  &   \text{if }    (u,v)\notin\mathcal{S}%
_{B}\text{ and }(u+1,v)\in\mathcal{S}_{B}\\
\left\langle u,v+1\right\rangle    & \text{if }    (u,v)\in\mathcal{S}%
_{B}\text{ and }(u,v+1)\in\mathcal{S}_{B}\\
\left\langle u\right\rangle  &   \text{if }    (u,v)\in\mathcal{S}_{B}\text{
and }(u,v+1)\notin\mathcal{S}_{B}%
\end{cases}
\]
is an bijection of $\mathcal{S}$ onto the set of $p$-regular partitions in
$\Lambda_{B}$.

\item If $B$ and $\bar{B}$ form a $[2:m]$ pair, then%
\[
\Phi_{B,\bar{B}}\circ\lambda_{B}=\lambda_{\bar{B}}.
\]

\end{enumerate}
\end{proposition}

\begin{proof}
First suppose that $\mathcal{P}_{B}=\emptyset$. Then $\lambda_{B}(u,u)=\langle
u\rangle$ and $\lambda_{B}(u,v)=\langle u+1,v\rangle$ if $u<v$. Thus
$\lambda_{B}$ is a bijection onto the set of partitions in $\Lambda_{B}$ whose
shorthand labels do not involve $0$, which are precisely the $p$-regular ones.
This proves statement (1) in this special case.

For an arbitrary block $B$, there is a sequence of blocks starting at one with
empty pyramid and ending at $B$ such that each successive pair of blocks forms
a $[2:m]$ pair for some $m$. Hence in order to prove both statements in
general it suffices to show that statement (2)\ holds for a fixed $[2:m]$-pair
of blocks $B$ and $\bar{B}$ under the assumption that statement (1) holds for
$B$. This is clearly true if $m\geq2$, because then $\mathcal{S}%
_{B}=\mathcal{S}_{\bar{B}}$ and $\Phi$ preserves shorthand labels. So let us
suppose that $B$ and $\bar{B}$ form a $[2:1]$ pair. We have $\mathcal{P}%
_{\bar{B}}=\mathcal{P}_{B}\cup\left\{  \left(  s,t\right)  \right\}  $ for
some $0\leq s<t\leq p-1$. Note that $(s+1,t),(s,t-1)\in\mathcal{S}_{B}$ and
$\left(  s-1,t\right)  ,(s,t+1)\notin\mathcal{\mathcal{S}}_{B}$. Hence%
\begin{align*}
\Phi(\lambda_{B}(s,t))  &  =\Phi(\left\langle t,t\right\rangle )=\left\langle
s\right\rangle =\lambda_{\bar{B}}(s,t),\\
\Phi(\lambda_{B}(s-1,t))  &  =\Phi\left(  \langle s,t\rangle\right)
=\left\langle t,t\right\rangle =\lambda_{\bar{B}}\left(  s-1,t\right)  ,\\
\Phi(\lambda_{B}(s,t-1))  &  =\Phi(\left\langle s\right\rangle )=\left\langle
s,t\right\rangle =\lambda_{\bar{B}}(s,t-1).
\end{align*}
Remembering our assumption that statement (1) holds for $B$, we also see that
for any $(u,v)\in\mathcal{S}\setminus\left\{  (s,t),\left(  s-1,t\right)
,\left(  s,t-1\right)  \right\}  $, the shorthand labels for $\Phi(\lambda
_{B}(u,v))$ and $\lambda_{B}(u,v)$ are the same and therefore that
$\Phi(\lambda_{B}(u,v))=\lambda_{\bar{B}}(u,v)$.
\end{proof}

\begin{example}
We take $p=7$ and $\kappa=(12,6,6,1,1,1,1)$ as in Example~\ref{weight2exampleA}.
The graph in Figure~\ref{blockfigure1} records the bijection of Proposition~\ref{parametrise}. Its vertices are in bijection with the $p$-regular partitions in $\Lambda_B$, each of which has a shorthand label as well as a label by an element of $\mathcal{S}$, via $\lambda_B$. The shorthand labels are placed to the right of the vertices, and the $\mathcal{S}$-labels to the left, in boldface. The subset $\mathcal{P}$ is indicated by black vertices and the pyramid $\mathcal{P}_B$ by square black vertices.  Two vertices are connected by an edge if an only if there exists a nonsplit extension of one of the corresponding simple $B$-modules by the other; in fact, by replacing each edge by a pair of directed edges in opposite directions one obtains the `extension quiver' of $B$.

\begin{figure}[h]
\[
\xymatrix
@C=4.5ex@R=2.625ex
@M=0.3EX
{
 && \umv  \ar@{-}[dd] \ar@{-}[ddrr] 
\ar@{}[l]|<(.16){\bf_{11}} \ar@{}[r]|<(.18){_{12}}
&& \umv \ar@{-}[ddll]|!{[dd];[ll]}\hole \ar@{-}[dd] \ar@{-}[ddrr] 
\ar@{}[l]|<(.16){\bf_{22}} \ar@{}[r]|<(.18){_{23}}
&& \umv \ar@{-}[ddll]|!{[dd];[ll]}\hole \ar@{-}[dd]  \ar@{-}[ddrr] 
 \ar@{}[l]|<(.16){\bf_{33}} \ar@{}[r]|<(.18){_{34}}
&& \umv \ar@{-}[ddll]|!{[dd];[ll]}\hole \ar@{-}[dd] \ar@{-}[dddr] 
\ar@{}[l]|<(.16){\bf_{44}} \ar@{}[r]|<(.18){_{4}}
&& \umv \ar@{-}[ddll] |!{[dddl];[ll]}\hole \ar@{-}[dddl] \ar@{-}[dddr] 
\ar@{}[l]|<(.16){\bf_{55}} \ar@{}[r]|<(.18){_{5}}
&& \umv \ar@{-}[dddl]
\ar@{}[l]|<(.16){\bf_{66}} \ar@{}[r]|<(.18){_{6}}
 &&  \\
&&&&&&&&&&&&&&
\\
&& \smv \ar@{-}[dr] 
\ar@{}[l]|<(.16){\bf_{01}} \ar@{}[r]|<(.18){_{02}}
&& \smv \ar@{-}[dl] \ar@{-}[dd] \ar@{-}[dr]
\ar@{}[l]|<(.16){\bf_{12}} \ar@{}[r]|<(.18){_{13}}
 && \smv \ar@{-}[dl] \ar@{-}[dr] 
 \ar@{}[l]|<(.16){\bf_{23}} \ar@{}[r]|<(.18){_{24}}
&& \smv \ar@{-}[dl] \ar@{-}[dd]
\ar@{}[l]|<(.16){\bf_{34}} \ar@{}[r]|<(.18){_{3}}
 && \mv \ar@{-}[dl] \ar@{-}[dr]
\ar@{}[l]|<(.16){\bf_{45}} \ar@{}[r]|<(.18){_{55}}   
&& \mv \ar@{-}[dl]
\ar@{}[l]|<(.16){\bf_{56}} \ar@{}[r]|<(.18){_{66}}
 &&\\
&&& \smv  
\ar@{}[l]|<(.16){\bf_{02}} \ar@{}[r]|<(.18){_{0}}
&& \smv \ar@{-}[dd] \ar@{-}[dr]
\ar@{}[l]|<(.16){\bf_{13}} \ar@{}[r]|<(.18){_{14}}
 && \smv \ar@{-}[dl] \ar@{-}[dd]  
\ar@{}[l]|<(.16){\bf_{24}} \ar@{}[r]|<(.18){_{2}} 
&& \mv \ar@{-}[dl] \ar@{-}[dr]
\ar@{}[l]|<(.16){\bf_{35}} \ar@{}[r]|<(.18){_{45}}
 && \mv \ar@{-}[dl] 
 \ar@{}[l]|<(.16){\bf_{46}} \ar@{}[r]|<(.18){_{56}} 
&& & \\
&& 
&& \mv \ar@{-}[dr]
\ar@{}[l]|<(.16){\bf_{03}} \ar@{}[r]|<(.18){_{33}}
 && \smv \ar@{-}[dd] 
 \ar@{}[l]|<(.16){\bf_{14}} \ar@{}[r]|<(.18){_{1}} 
&& \mv \ar@{-}[dl] \ar@{-}[dr]
\ar@{}[l]|<(.16){\bf_{25}} \ar@{}[r]|<(.18){_{35}}
 && \mv \ar@{-}[dl]  
 \ar@{}[l]|<(.16){\bf_{36}} \ar@{}[r]|<(.18){_{46}}
&&  && \\
&  && 
&& \mv  \ar@{-}[dr] 
\ar@{}[l]|<(.16){\bf_{04}} \ar@{}[r]|<(.18){_{44}}
&& \mv \ar@{-}[dl] \ar@{-}[dr]  
\ar@{}[l]|<(.16){\bf_{15}} \ar@{}[r]|<(.18){_{25}}
&& \mv \ar@{-}[dl] 
\ar@{}[l]|<(.16){\bf_{26}} \ar@{}[r]|<(.18){_{36}}&& 
&&  & \\
 &&  && 
 && \mv  \ar@{-}[dr]  
 \ar@{}[l]|<(.16){\bf_{05}} \ar@{}[r]|<(.18){_{15}}
&& \mv \ar@{-}[dl] 
\ar@{}[l]|<(.16){\bf_{16}} \ar@{}[r]|<(.18){_{26}} && 
&&  && \\
&  && 
&&  && \mv 
\ar@{}[l]|<(.16){\bf_{06}} \ar@{}[r]|<(.18){_{16}}
&&  &&  
&&  & \\
}
\]
\caption{Extension quiver of $B$ when $p=7$ and $\kappa=(12,6,6,1,1,1,1)$}
\label{blockfigure1}
\end{figure}
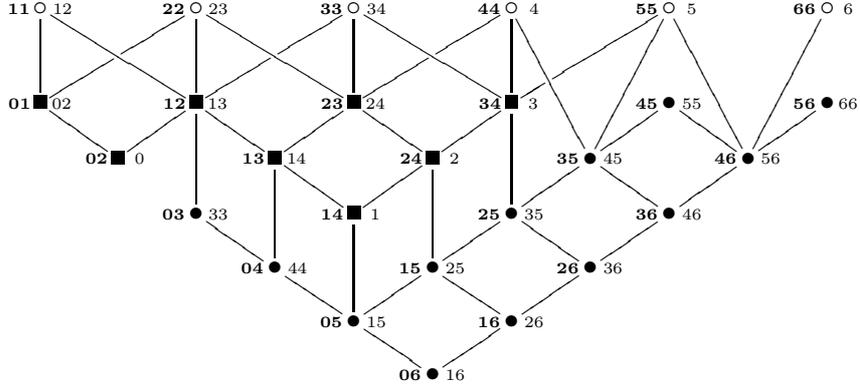
\FloatBarrier

\end{example}

\subsection{Gradings}

We expect that all blocks of weight $w<p$ should have gradings compatible with
radical filtrations. This is easy to verify when $w<2$, and has been proved by Peach 
for $w=2$.

\begin{theorem}[Peach \cite{Peach}]
Let $B$ be a block of weight $2$. Then there exists a grading
$B=\oplus_{i=0}^4 B_i$ such that $\operatorname{rad}^j(B)=\oplus_{i=j}^4 B_i$.
\end{theorem}

We may assume that the idempotents $f_\lambda$ introduced in \S\ref{subsection:blocks}
are in $B_0$.
\subsection{Morita and derived equivalences}

There is a strong relationship between the module categories of blocks in a
Scopes pair. Suppose that $B$ and $\bar{B}$ form a $[2:m]$ pair of blocks.
Scopes \cite{Scopes-CMMEBSG} proves that, if $m\geq2$, there is an equivalence%
\[
F^{\prime}:B\mMod\iso\bar{B}\mMod
\]
such that
\[
F^{\prime}(Bf_{\lambda})\cong\bar{B}f_{\Phi(\lambda)}%
\]
\newline for all $p$-regular $\lambda\in\Lambda_{B}$. Her results are true in
greater generality, for blocks of arbitrary weight. Rickard \cite{Rickard-MSRI} 
(see also \cite{Chuang-Rouquier}) built on the ideas of Scopes, proving the existence of some derived
equivalences between blocks. A special case of Rickard's result (see \cite{Chuang-DCSBSGCB}
and \cite[\S5]{Peach}) states that if $B$ and $\bar{B}$ form a $[2:1]$
pair, then there exists an equivalence%
\[
F^{\prime}:D^{b}(B\mMod)\iso D^{b}
(\bar{B}\mMod),
\]
such that, in the notation of the analysis of $[2:1]$ pairs in \S\ref{subsection:weight2blocks},
\[
F^{\prime}(Bf_{\langle t,t\rangle}\langle 3 \rangle [1])\cong\bar{B}f_{\Phi(\left\langle
t,t\right\rangle )}%
\]
and
\[
F^{\prime}(\operatorname{cone}(P'_\lambda\overset{\zeta_{\lambda
}^{\prime}}{\longrightarrow}Bf_{\lambda}))\cong Bf_{\Phi(\lambda)}%
\]
for all $p$-regular $\lambda\in\Lambda_{B}$ apart from $\langle t,t\rangle$,
where $\zeta_{\lambda}^{\prime}$ is a projective cover of the smallest submodule $M$ of
$Bf_{\lambda}$ such that $\operatorname{Hom}_B(Bf_{\langle t,t \rangle}\langle n \rangle,
Bf_{\lambda}/M)=0$ for all $n\in\mathbb{Z}$.

We will only make use of truncated versions of these equivalences. Let%
\[
f=\sum_{(u,v)\in\mathcal{P}}f_{\lambda_{B}(u,v)}\in B\quad\text{and }\quad
\bar{f}=\sum_{(u,v)\in\mathcal{P}}f_{\lambda_{\bar{B}}(u,v)}\in\bar{B}.
\]
If $m\geq2$, then by Proposition~\ref{parametrise} the Morita equivalence $F'$ above induces an equivalence%
\[
F=F_{B,\bar{B}}:fBf\mMod\iso\bar{f}
\bar{B}\bar{f}\mMod
\]
such that%
\[
F(fBf_{\lambda_{B}(u,v)})\cong\bar{f}\bar{B}f_{\lambda_{\bar{B}(u,v)}}%
\]
for all $(u,v)\in\mathcal{P}$. If $m=1$, then $\langle t,t\rangle
=\lambda_B(s,t)$ and, by Proposition~\ref{parametrise}, we have 
$\lambda\in\lambda_B(\mathcal{P})$ if and only if
$\Phi(\lambda)\in\lambda_{\bar{B}}(\mathcal{P})$.
Hence the equivalence $F^{\prime}:D^{b}%
(B\mMod)\iso D^{b}(\bar{B}\mMod)$ induces an
equivalence
\[
F=F_{B,\bar{B}}:D^{b}(fBf\mMod)\iso D^{b}(\bar{f}\bar{B}\bar{f}\mMod)
\]
such that
\[
F(fBf_{\lambda_{B}(s,t)}\langle 3 \rangle [1])\cong\bar{f}\bar{B}f_{\lambda_{\bar{B}(s,t)}}%
\]
and
\[
F(\operatorname{cone}(P_{(u,v)}\overset{\zeta_{(u,v)}%
}{\longrightarrow}fBf_{\lambda_{B}(u,v)}))\cong\bar{f}\bar{B}f_{\lambda
_{\bar{B}}(u,v)}%
\]
for $(u,v)\in\mathcal{P}\setminus\{(s,t)\}$, where $\zeta_{(u,v)}$ is a projective
cover of the smallest submodule $M$ of
$fBf_{\lambda_{B}(u,v)}$ such that $\operatorname{Hom}_{fBf}(fBf_{\lambda_B(s,t)}\langle n \rangle,
fBf_{\lambda_B(u,v)}/M)=0$ for all $n\in\mathbb{Z}$.

\subsection{Main result}

We are now ready to state and prove our result linking Cubist algebras and
blocks of symmetric groups. Let $B$ be a block of weight $2$. Define%
\[
x_{B}:\left\{  \left(  u,v\right)  \in\mathbb{Z}^{2}\mid  u < v\right\}
\longrightarrow\mathbb{Z}^{3}%
\]
by 
$$x_{B}(u,v)=
\begin{cases}
\left(  -u-1,v,0\right)  & \text{if } (u,v)\in\mathbb{Z}^{2}
\setminus\mathcal{\mathcal{P}}_{B},  \\
(-u,1+v,1) & \text{if } \left(  u,v\right)  \in\mathcal{\mathcal{P}}_{B}.
\end{cases}
$$
 Then
\[
\mathcal{X}_{B}=\operatorname{Im}(x_{B})\cup\{(i,j,1)\in\mathbb{Z}^{3}\mid 
i+j\leq 1\}
\]
is a Cubist subset of $\mathbb{Z}^{3}$. Indeed, 
$$\mathcal{X}_B^-=\mathbb{Z}\times\mathbb{Z}\times\mathbb{Z}_{\leq 0}
\:\cup\:\{(i,j,1)\in\mathbb{Z}^{3}\mid 
i+j\leq 1 \text{ or } (-i,j-1)\in\mathcal{P}\}$$ is an ideal in $\mathbb{Z}^3$
such that
 $\mathcal{X}_B=
\mathcal{X}_B^-\setminus\mathcal{X}_B^-[-1]$.

\begin{example}
As in earlier examples, we take $p=7$ and $\kappa=(12,6,6,1,1,1,1)$.
Figure~\ref{cubistforblockfigure} shows part of $\mathcal{X}_B$ realised in the plane as a rhombus tiling. The image of $\mathcal{P}$ under $x_B$ is indicated by black vertices and that of $\mathcal{P}_B$ by square black vertices.
Compare with Figure~\ref{blockfigure1}.
\FloatBarrier
\begin{figure}[h]
\[
\xymatrix@C=4.5ex@R=2.625ex@M=0.3EX{
& \umv \ar@{-}[dl] \ar@{-}[dr] && \umv \ar@{-}[dl] \ar@{-}[dr] && 
\umv \ar@{-}[dl] \ar@{-}[dr] && \umv \ar@{-}[dl] \ar@{-}[dr]  
&& \umv \ar@{-}[dl] \ar@{-}[dr] && \umv \ar@{-}[dl] \ar@{-}[dr]  
&& \umv \ar@{-}[dl] \ar@{-}[dr] & \\
\umv \ar@{-}[dr] \ar@{-}[dd] && \umv \ar@{-}[dl] \ar@{-}[dr] 
&& \umv \ar@{-}[dl] \ar@{-}[dr] && \umv \ar@{-}[dl] \ar@{-}[dr] 
&& \umv \ar@{-}[dl] \ar@{-}[dr] && \umv \ar@{-}[dl] \ar@{-}[dd] \ar@{-}[dr] 
&& \umv \ar@{-}[dl] \ar@{-}[dd] \ar@{-}[dr] &&\umv \ar@{-}[dl] \ar@{-}[dd] \\
& \umv \ar@{-}[dd] \ar@{-}[dr] && \umv \ar@{-}[dl] \ar@{-}[dr] 
&& \umv \ar@{-}[dl] \ar@{-}[dr] && \umv \ar@{-}[dl] \ar@{-}[dr] 
&& \umv \ar@{-}[dl] \ar@{-}[dd] && \umv \ar@{-}[dd] 
&& \umv \ar@{-}[dd] & \\
\umv \ar@{-}[dr] && \smv \ar@{-}[dd] \ar@{-}[dr] 
\ar@{}[l]|<(.16){\bf_{01}} 
&& \smv \ar@{-}[dl] \ar@{-}[dd] \ar@{-}[dr]
\ar@{}[l]|<(.16){\bf_{12}} 
 && \smv \ar@{-}[dl] \ar@{-}[dr] 
 \ar@{}[l]|<(.16){\bf_{23}} 
&& \smv \ar@{-}[dl] \ar@{-}[dd] 
\ar@{}[l]|<(.16){\bf_{34}} 
&& \mv \ar@{-}[dl] \ar@{-}[dr]  
\ar@{}[l]|<(.16){\bf_{45}} 
&& \mv \ar@{-}[dl] \ar@{-}[dr]
\ar@{}[l]|<(.16){\bf_{56}} 
 &&\umv \ar@{-}[dl] \\
& \umv \ar@{-}[dl] \ar@{-}[dr] && \smv \ar@{-}[dd] 
\ar@{}[l]|<(.16){\bf_{02}} 
&& \smv \ar@{-}[dd] \ar@{-}[dr] 
\ar@{}[l]|<(.16){\bf_{13}} 
&& \smv \ar@{-}[dl] \ar@{-}[dd]  
\ar@{}[l]|<(.16){\bf_{24}} 
&& \mv \ar@{-}[dl] \ar@{-}[dr] 
\ar@{}[l]|<(.16){\bf_{35}} 
&& \mv \ar@{-}[dl] \ar@{-}[dr]  
\ar@{}[l]|<(.16){\bf_{46}}
&& \umv \ar@{-}[dl] \ar@{-}[dr] & \\
\umv \ar@{-}[dr] && \umv \ar@{-}[dl] \ar@{-}[dr]  
&& \mv \ar@{-}[dl] \ar@{-}[dr] 
\ar@{}[l]|<(.16){\bf_{03}} 
&& \smv \ar@{-}[dd]  
\ar@{}[l]|<(.16){\bf_{14}} 
&& \mv \ar@{-}[dl] \ar@{-}[dr] 
\ar@{}[l]|<(.16){\bf_{25}}
&& \mv \ar@{-}[dl] \ar@{-}[dr]  
\ar@{}[l]|<(.16){\bf_{36}} 
&& \umv \ar@{-}[dl] \ar@{-}[dr] && \umv \ar@{-}[dl] \\
& \umv \ar@{-}[dl] \ar@{-}[dr] && \umv \ar@{-}[dl] \ar@{-}[dr] 
&& \mv \ar@{-}[dl] \ar@{-}[dr]
\ar@{}[l]|<(.16){\bf_{04}} 
 && \mv \ar@{-}[dl] \ar@{-}[dr]  
 \ar@{}[l]|<(.16){\bf_{15}}
&& \mv \ar@{-}[dl] \ar@{-}[dr]
\ar@{}[l]|<(.16){\bf_{26}} 
 && \umv \ar@{-}[dl] \ar@{-}[dr]  
&& \umv \ar@{-}[dl] \ar@{-}[dr] & \\
\umv \ar@{-}[dr] && \umv \ar@{-}[dl] \ar@{-}[dr] && 
\umv \ar@{-}[dl] \ar@{-}[dr] && \mv \ar@{-}[dl] \ar@{-}[dr]  
\ar@{}[l]|<(.16){\bf_{05}} 
&& \mv \ar@{-}[dl] \ar@{-}[dr] 
\ar@{}[l]|<(.16){\bf_{16}} 
&& \umv \ar@{-}[dl] \ar@{-}[dr]  
&& \umv \ar@{-}[dl] \ar@{-}[dr] &&\umv \ar@{-}[dl] \\
& \umv \ar@{-}[dl] \ar@{-}[dr] && \umv \ar@{-}[dl] \ar@{-}[dr]  
&& \umv \ar@{-}[dl] \ar@{-}[dr] && \mv \ar@{-}[dl]
\ar@{}[l]|<(.16){\bf_{06}} 
 \ar@{-}[dr]  
&& \umv \ar@{-}[dl] \ar@{-}[dr] && \umv \ar@{-}[dl] \ar@{-}[dr]  
&& \umv \ar@{-}[dl] \ar@{-}[dr]  & \\
\umv && \umv 
&& \umv && \umv 
&& \umv && \umv 
&& \umv && \umv \\
}
\]
\caption{Part of the Cubist subset $\calx_B$ when $p=7$ and 
$\kappa=(12,6,6,1,1,1,1)$}
\label{cubistforblockfigure}
\end{figure}
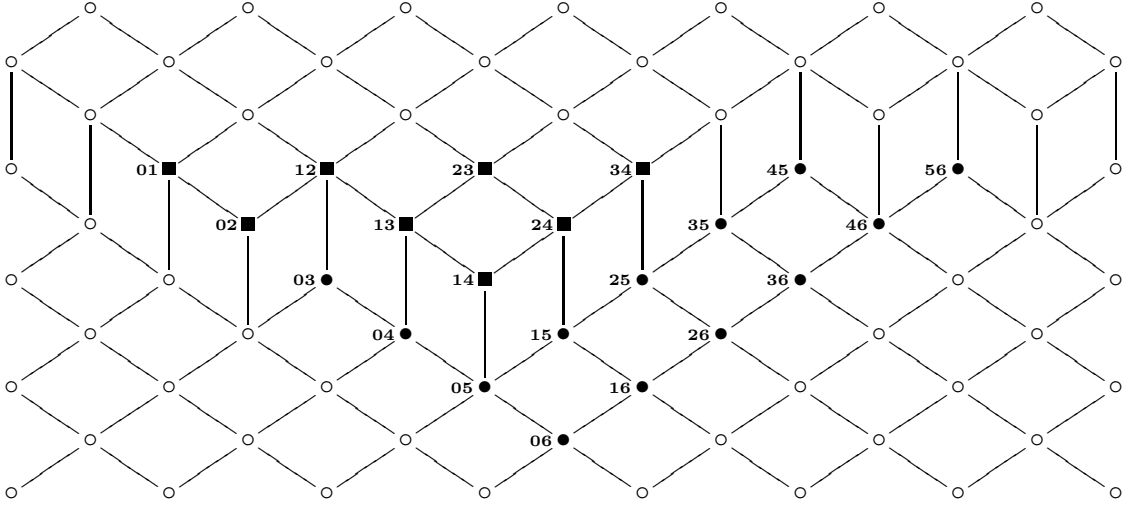
\FloatBarrier
\end{example}

 Let $U_{\mathcal{X}_{B}}$ be the
Cubist algebra corresponding to $\mathcal{X}_B$ and put $e=\sum_{ x \in x_{B}(\mathcal{P)}%
}e_{ x }\in U_{\mathcal{X}_{B}}$.

\begin{theorem} \label{blockiscube}
We have an equivalence%
\[
F_{B}:eU_{\mathcal{X}_{B}}e\mMod\iso fBf\mMod
\]
such that for all $(u,v)\in\mathcal{\mathcal{P}}$,%
\[
F_{B}(eU_{\mathcal{X}_{B}}e_{x_{B}(u,v)})\cong fBf_{\lambda_{B}(u,v)}.
\]

\end{theorem}

\begin{proof}
The case $\mathcal{P}_{B}=\emptyset$ is handled in \S\ref{rouquier}.

Now suppose that $B$ and $\bar{B}$ form a $[2:m]$ pair and assume that
the theorem is known to hold for the block $B$. We need to prove that it holds
for $\bar{B}$.

Suppose that $m\geq2$. Then $\mathcal{P}_{B}=\mathcal{P}_{\bar{B}}$, and hence
$x_{B}=x_{\bar{B}}$ and $\mathcal{X}_{B}=\mathcal{X}_{\bar{B}}$. So we simply
define $F_{\bar{B}}=F_{B,\bar{B}}\circ F_{B}$.

Suppose that $m=1$. We have $\mathcal{P}_{\bar{B}}=\mathcal{P}_{B}\cup\left\{
\left(  s,t\right)  \right\}  $, where $0\leq s<t\leq p-1$. Thus
$x_{B}(u,v)=x_{\bar{B}}(u,v)$ for $(u,v)\in\mathcal{P}\setminus\left\{
s,t\right\}  $, and $x_{\bar{B}}(s,t)=x_{B}(s,t)[1]$, and $\mathcal{X}%
_{\bar{B}}=\left(  \mathcal{X}_{B}\setminus\{(x_{B}(s,t)\}\right)
\cup\left\{  x_{B}\left(  s,t\right)  [1]\right\}  $. Let $  z  =x_{B}(s,t)$. By
Theorem~\ref{mutation theorem} and Lemma~\ref{basic} we have an equivalence%
\[
G^{\prime}:D^{b}\left(U_{\mathcal{X}_{B}}\mMod\right)
\iso D^{b}\left( U_{\mathcal{X}_{\bar{B}}}\mMod\right)
\]
such that%
\[
G^{\prime}(U_{\mathcal{X}_{B}}e_{  z  }\langle 3 \rangle [1])\cong U_{\mathcal{X}_{\bar{B}}}%
e_{  z  [1]}.
\]
and%

\[
G^{\prime}(\operatorname{cone}(Q'_x\overset{\xi
_{ x }^{\prime}}{\longrightarrow}U_{\mathcal{X}_{B}}e_{ x }))\cong U_{\mathcal{X}%
_{\bar{B}}}e_{ x }%
\]
for $ x \in\mathcal{X}_{B}\setminus\{  z  \}$, where $\xi_{ x }^{\prime}$ is a
projective cover of the smallest submodule $M$ of $U_{\mathcal{X}_{B}}e_{ x }$
such that $\operatorname{Hom}_{U_{\mathcal{X}_{B}}}(U_{\mathcal{X}_{B}}e_{  z  }\langle n \rangle,
U_{\mathcal{X}_{B}}e_{ x }/M)=0$ for all $n\in\mathbb{Z}$.
We know that $Q'_x$ is either $0$ or isomorphic to $U_{\mathcal{X}_{B}}e_{  z  }\langle n \rangle$
for some $n\in\mathbb{Z}$.
 
Since $(s,t)\in\mathcal{P}$, we have $ee_{  z  }\neq0$, which implies the
existence of an equivalence%
\[
G:D^{b}\left( eU_{\mathcal{X}_{B}}e\mMod\right)  \iso
D^{b}\left( \bar{e}U_{\mathcal{X}_{\bar{B}}}\bar
{e}\mMod\right)
\]
such that%
\[
G(eU_{\mathcal{X}_{B}}e_{  z  }\langle 3 \rangle[1])\cong\bar{e}U_{\mathcal{X}_{\bar{B}}%
}e_{  z  [1]}.
\]
and%
\[
G(\operatorname{cone}(Q_x\overset{\xi
_{ x }}{\longrightarrow}eU_{\mathcal{X}_{B}}e_{ x }))\cong \bar{e}U_{\mathcal{X}%
_{\bar{B}}}e_{ x }%
\]
for $ x \in x_B(\mathcal{P})\setminus\{  z  \}$, where $\xi_{ x }$ is a
projective cover of the smallest submodule $M$ of $eU_{\mathcal{X}_{B}}e_{ x }$
such that $\operatorname{Hom}_{eU_{\mathcal{X}_{B}}}(eU_{\mathcal{X}_{B}}e_{  z  }\langle n \rangle,
eU_{\mathcal{X}_{B}}e_{ x }/M)=0$ for all $n\in\mathbb{Z}$.
 Put%
\[
F_{\bar{B}}=F_{B,\bar{B}}\circ RF_{B}\circ H:D^{b}(\bar{e}U_{\mathcal{X}%
_{\bar{B}}}\bar{e})\iso D^{b}(\bar{f}\bar{B}\bar{f}),
\]
where $H$ is a quasi-inverse to $G$. Then we have
\begin{align*}
F_{\bar{B}}(\bar{e}U_{\mathcal{X}_{\bar{B}}}e_{x_{\bar{B}}(s,t)})  &  \cong
F_{B,\bar{B}}(RF_{B}(eU_{\mathcal{X}_{B}}e_{x_{B}(s,t)}[1])\\
&  \cong F_{B,\bar{B}}(fBf_{\lambda_{B}(s,t)}[1])\\
&  \cong\bar{f}\bar{B}\bar{f}_{\lambda_{B}(s,t)},
\end{align*}
and for $\left(  u,v\right)  \in\mathcal{P}\setminus\left\{  (s,t)\right\}  $,%
\begin{align*}
F_{\bar{B}}(\bar{e}U_{\mathcal{X}_{\bar{B}}}e_{x_{\bar{B}}(u,v)})  &  \cong
F_{B,\bar{B}}(RF_{B}(\operatorname{cone}(Q_x\overset{\xi_{ x }}{\longrightarrow}U_{\mathcal{X}_{B}}e_{_{x_{B}(u,v)}%
})))\\
&  \cong F_{B,\bar{B}}(\operatorname{cone}(F_B(Q_x)
\overset{F_{B}(\xi_{ x })}{\longrightarrow}Be_{\lambda_{B}(u,v)}))\\
&  \cong\bar{f}\bar{B}\bar{f}_{\lambda_{B}(u,v)}.
\end{align*}
It follows that $F_{\bar{B}}$ restricts to an equivalence
$\bar{e}U_{\mathcal{X}_{\bar{B}}}\bar{e}\mMod\iso\bar{f}%
\bar{B}\bar{f}\mMod$.
\end{proof}

\subsection{Rouquier blocks} \label{rouquier}
Now we assume that $\mathcal{P}_{B}=\emptyset$.  In this case $B$ is known as a \emph{Rouquier
block} of weight $2$ and is particularly well understood.
We shall prove that there is an isomorphism
$eU_{\mathcal{X}_B}e\iso fBf$ sending $e_{x_B(u,v)}$ to $f_{\lambda_B(u,v)}$ for all $(u,v)\in\mathcal{P}$,
and thus complete the proof of Theorem~\ref{blockiscube}.

 We have%
\[
\mathcal{X}=\mathcal{X}_{B}=\mathcal{X}_{0}\cup\mathcal{X}_{1},
\]
where%
\[
\mathcal{X}_{0}=\operatorname{Im}(x_{B})=\{(i,j,0)\mid  i+j\geq0\}
\]
and%
\[
\mathcal{X}_{1}=\{(i,j,1)\mid  i+j\leq 1\}.
\]
We shall in fact prove that the algebra
$\sum_{  x  , x' \in\mathcal{X}_0}e_{  x  }U_{\mathcal{X}}e_{ x' }$
is isomorphic to a truncation of an infinite-dimensional wreath product. Then
a further truncation will yield the desired isomorphism. 


\subsubsection{Reformulation in terms of wreath products}
Let $A$ be the path algebra of the infinite quiver%
\[\cdots
\overset{\gamma_{-2}}{\underset{\delta_{-1}}{\rightleftarrows}}\cdot
_{g_-1}\overset{\gamma_{-1}}{\underset{\delta_{0}}{\rightleftarrows}}\cdot
_{g_0}\overset{\gamma_{0}}{\underset{\delta_{1}}{\rightleftarrows}}\cdot
_{g_1}\overset{\gamma_{1}}{\underset{\delta_{2}}{\rightleftarrows}}%
\cdots
\]
modulo the relations
$\gamma_{i}\gamma_{i+1}=0,\;
\delta_{i}\delta_{i-1}=0,\;
\gamma_{i}\delta_{i+1}+\delta_{i}\gamma_{i-1}=0,\,$ for $i\in\mathbb{Z}$;
this is a graded algebra with each arrow in degree $1$.
(We saw in Remark~\ref{smallrank} that this algebra is isomorphic to $U_\mathcal{Y}$, for any 
Cubist subset $\mathcal{Y}\subset\mathbb{Z}^2$.)
 
We now form the wreath product $A\wr\mathfrak{S}_{2}=A\otimes A \otimes k \mathfrak{S}_2$, graded with
$k\mathfrak{S}_2$ in degree $0$.
Let $\sigma$ be the nonidentity element of $\mathfrak{S}_2$. 
The following elements of $A\wr\mathfrak{S}_{2}$ are a complete set of orthogonal idempotents:

\begin{align*}
g_{ij} & =  g_i\otimes g_j \otimes 1,   \quad (i,j\in\mathbb{Z}, i<j), \\
g^\pm_{ii}&=g_i\otimes g_i \otimes (1\pm\sigma)/2, \quad  (i\in\mathbb{Z}).
\end{align*}

Now observe that $(g_1+\ldots+g_p)A(g_1+\ldots+g_p)$
 is a Brauer tree algebra of a line with
$p-1$ simple modules, so is Morita equivalent to the principal block of $k\mathfrak{S}_p$. It follows
that $g'(A\wr\mathfrak{S}_{2})g'$ is Morita equivalent to the principal block of $k\mathfrak{S}_p\wr\mathfrak{S}_2$,
where $g'=\sum_{1\leq i<j \leq p-1} g_{ij} + \sum_{1\leq i \leq p-1} g^+_{ii}+\sum_{1\leq i \leq p-1} g^-_{ii}$.

On the other hand there is a Morita equivalence between the block $B$ and the principal block of $k\mathfrak{S}_p\wr\mathfrak{S}_2$ (see \cite{Chuang-DCSBSGCB} and \cite{Chuang-Kessar}), where the
correspondence between simple modules is known explicitly. Passing to truncated algebras, we obtain an isomorphism
$fBf \iso g(A\wr\mathfrak{S}_{2})g$ that sends $f_{\langle i,j \rangle}$ to $g_{ij}$ for $1\leq i \leq j \leq p-1$,
where $g_{ii}$ is defined to be $g^+_{ii}$ if $i$ is even and
$g^-_{ii}$ if $i$ is odd, and $g=\sum_{1\leq i \leq j \leq p-1} g_{ij}$.

For $(u,v)\in\mathcal{P}$, we have $\lambda_B(u,v)=\langle u+1,v \rangle$
and $x_B(u,v)=(-u-1,v,0)$. So to obtain the desired isomorphism between
$eU_\mathcal{X}e$ and $fBf$, it suffices to get an isomorphism
$eU_\mathcal{X}e\iso g(A\wr\mathfrak{S}_{2})g$ sending $e_{(-i,j,0)}$
to $g_{ij}$ for $1\leq i \leq j\leq p-1$.

It will be more convenient to prove a stronger statement. Let
$$\tilde{e}U_{\mathcal{X}}\tilde{e}:=\bigoplus_{x,x'\in\mathcal{X}_0} e_x U_{\mathcal{X}} e_{x'}
\qquad\text{and}\qquad  \tilde{g}(A\wr\mathfrak{S}_{2})\tilde{g}:=
 \bigoplus_{i,j,i',j'\in\mathbb{Z},\:i\leq j,\:i'\leq j'} g_{ij}(A\wr\mathfrak{S}_{2})g_{i'j'}.$$
We will show that there is an isomorphism 
$\tilde{e}U_\mathcal{X}\tilde{e}\iso \tilde{g}(A\wr\mathfrak{S}_{2})\tilde{g}$ sending $e_{(-i,j,0)}$
to $g_{ij}$ for $i,j\in\mathbb{Z},\: i\leq j$.

\begin{figure}[h]
\[
\includegraphics[height=2.5in]{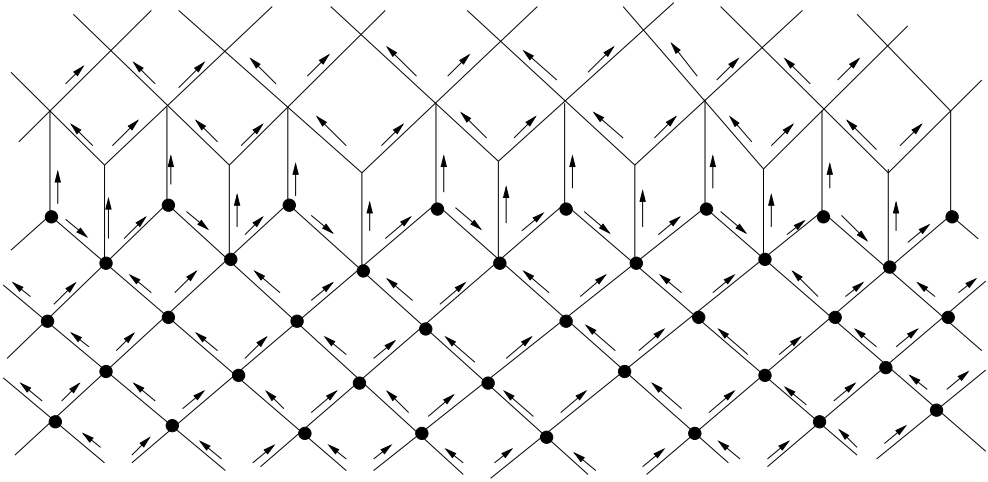}
\]
\caption{$\mathcal{X}_B$ in case $\mathcal{P}_B = \emptyset$.}
\label{Rouquier}
\end{figure}
\FloatBarrier

\subsubsection{A presentation for $\tilde{e}U_{\mathcal{X}}\tilde{e}$} \label{subsubsection:relations}
We first observe that $\tilde{e}U_{\mathcal{X}}\tilde{e}$ is a Koszul algebra, since
$\mathcal{X}$ is a Cubist subset and
$\mathcal{X}_{0}$ is an ideal in $(\mathcal{X},\preceq)$,
as in Figure~\ref{Rouquier}. In particular it
is quadratic algebra.

So using the alternative
presentation of $U_3$ given in Remark~\ref{Peachsigns}
we deduce a presentation of $\tilde{e}U_{\mathcal{X}}\tilde{e}$ by quiver and relations. The
quiver has vertices
\[
\{e_{i,j}\mid  i,j\in\mathbb{Z},~i+j\geq0\},
\]
and arrows
\begin{align*}
\{\alpha_{i,j,1},\alpha_{i,j,2}  &  \mid  i+j\geq0\},\\
\{\beta_{i,j,1},\beta_{i,j,2}  &  \mid  i+j\geq1\}.
\end{align*}
The arrows $\alpha_{i,j,1}$ and $\alpha_{i,j,2}$ are directed from $e_{i,j}$
to $e_{i+1,j}$ and from $e_{i,j}$ to $e_{i,j+1}$, and the arrows
$\beta_{i,j,1}$ and $\beta_{i,j,2}$ are directed from $e_{i,j}$ to $e_{i-1,j}$
and from $e_{i,j}$ to $e_{i,j-1}$. Then $\tilde{e}U_{\mathcal{X}}\tilde{e}$ is isomorphic to
the path algebra of this quiver, modulo relations%
\begin{align}
	\begin{split}
		\alpha_{i,j,1}\alpha_{i+1,j,1}  & =0\quad(i+j\geq0),\\
		\alpha_{i,j,2}\alpha_{i,j+1,2}  &  =0\quad(i+j\geq0),\\
		\beta_{i,j,1}\beta_{i-1,j,1}  &  =0\quad(i+j\geq2),\\
		\beta_{i,j,2}\beta_{i,j-1,2}  &  =0\quad(i+j\geq2);
	\end{split} \tag{0}\\
	\nonumber\\
	\begin{split}
		\alpha_{i,j,1}\alpha_{i+1,j,2}  &  =\alpha_{i,j,2}\alpha_{i,j+1,1}
		\quad(i+j\geq0),\\
		\beta_{i,j,1}\beta_{i-1,j,2}  &  =\beta_{i,j,2}\beta_{i,j-1,1}\quad
		(i+j\geq2),\\
		\alpha_{i,j,1}\beta_{i+1,j,2}  &  =\beta_{i,j,2}\alpha_{i,j-1,1}\quad
		(i+j\geq1),\\
		\alpha_{i,j,2}\beta_{i,j+1,1}  &  =\beta_{i,j,1}\alpha_{i-1,j,2}\quad
		(i+j\geq1),\\
		\alpha_{-j,j,1}\beta_{-j+1,j,2}  &  =0\quad(j\in\mathbb{Z}),\\
		\alpha_{-j,j,2}\beta_{-j,j+1,1}  &  =0\quad(j\in\mathbb{Z});
	\end{split} \tag{1} \\
	\nonumber\\
	\begin{split}
		\alpha_{i,j,1}\beta_{i+1,j,1}  &  =\beta_{i,j,1}\alpha_{i-1,j,1}\quad
		(i+j\geq2),\\
		\alpha_{i,j,2}\beta_{i,j+1,2}  &  =\beta_{i,j,2}\alpha_{i,j-1,2}\quad
		(i+j\geq2),\\
		\alpha_{-j,j+1,1}\beta_{-j+1,j+1,1}-\beta_{-j,j+1,1}\alpha_{-j-1,j+1,1}  &
		=\beta_{-j,j+1,2}\alpha_{-j,j,2}-\alpha_{-j,j+1,2}\beta_{-j,j+2,2}\quad
		(j\in\mathbb{Z)},\\
		\alpha_{-j,j,1}\beta_{-j+1,j,1}  &  =\alpha_{-j,j,2}\beta_{-j,j+1,1}\quad
		(j\in\mathbb{Z}).
	\end{split} \tag{2}
\end{align}
The three groups of relations come from \ref{U1}, \ref{U2'} and \ref{U3'},
respectively.

\subsubsection{Construction of the isomorphism} 
We define a homomorphism%
\[
\tilde{e}U_{\mathcal{X}}\tilde{e}\rightarrow \tilde{g}(A\wr\mathfrak{S}_{2})\tilde{g}%
\]%
by
\begin{align*}
e_{i,j}  &  \mapsto g_{-i,j},\\
\alpha_{i,j,1}  &  \mapsto%
\begin{cases}
\delta_{-i}\otimes g_{j}\otimes1 &   \text{if }i+j\geq1,\\
\delta_{j}\otimes g_{j}\otimes1+(-1)^{j}g_{j}\otimes\delta_{j}\otimes\sigma 
& \text{if }i+j=0,
\end{cases}
\\
\alpha_{i,j,2}  &  \mapsto%
\begin{cases}
g_{-i}\otimes\gamma_{j}\otimes1 &   \text{if }i+j\geq1,\\
g_{j}\otimes\gamma_{j}\otimes1+(-1)^{j}\gamma_{j}\otimes g_{j}\otimes\sigma &
 \text{if }i+j=0,
\end{cases}
\\
\beta_{i,j,1}  &  \mapsto%
\begin{cases}
\gamma_{-i}\otimes g_{j}\otimes1   & \text{if }i+j\geq2,\\
\frac{1}{2}\left(  \gamma_{j-1}\otimes g_{j}\otimes1+(-1)^{j}\gamma
_{j-1}\otimes g_{j}\otimes\sigma\right)  &   \text{if }i+j=1,
\end{cases}
\\
\beta_{i,j,2}  &  \mapsto%
\begin{cases}
g_{-i}\otimes\delta_{j}\otimes1   & \text{if }i+j\geq2,\\
\frac{1}{2}\left( g_{j-1}\otimes\delta_{j}\otimes1+(-1)^{j-1}g_{j-1}%
\otimes\delta_{j}\otimes\sigma\right)    & \text{if }i+j=1.
\end{cases}
\end{align*}

This yields a homomorphism: that the images of the generators satisfy the
relations stated in \S\ref{subsubsection:relations} is a straightforward calculation in $A\wr\mathfrak{S}_2$.

The graded Cartan matrix $C(q)$ for $\tilde{e}U_{\mathcal{X}}\tilde{e}$ is just a submatrix
of that of $U_{\mathcal{X}}$, which can be calculated using
Proposition~\ref{alternativeformula}. We get%
\[
C(q)_{(i,j),(i^{\prime},j^{\prime})}=%
\begin{cases}
1+q^{2}+q^{4}   & \text{if }(i,j)=(i^{\prime},j^{\prime})\text{ and }i+j=0,\\
1+3q^{2}+q^{4}   & \text{if }(i,j)=(i^{\prime},j^{\prime})\text{ and
}i+j=1,\\
1+2q^{2}+q^{4}   & \text{if }(i,j)=(i^{\prime},j^{\prime})\text{ and }%
i+j\geq2,\\
q+q^{3} &   \text{if }| i-i^{\prime}| +| j-j^{\prime}| =1,\\
q^{2} &   \text{if }\left\vert i-i^{\prime}\right\vert =1\text{ and
}| j-j^{\prime}| =1,\\
0  & \text{otherwise.}%
\end{cases}
\]
Replacing $(i,j)$ by $(-i,j)$ and $(i',j')$ by $(-i',j')$,
we obtain the same answers for the graded Cartan matrix of $\tilde{g}(A\wr\mathfrak{S}
_{2})\tilde{g}$, by direct calculation (see, e.g., Proposition 7.1
of \cite{CTGW}). So to show that our homomorphism is an isomorphism as desired, it suffices
to demonstrate that it is surjective.  
This can be done by making the following observations:
\begin{itemize}
\item The degree $0$ and degree $1$ components of $\tilde{g}(A\wr\mathfrak{S}_2)\tilde{g}$ are contained in the image of our homomorphism, so it is enough to show that they generate $\tilde{g}(A\wr\mathfrak{S}_2)\tilde{g}$ as an algebra.
\item The algebra $A\wr\mathfrak{S}_2$ is generated by its degree $0$ and $1$ components.
\item Define $h_{ii}$ to be $g^-_{ii}$ if $i$ is odd and $g^+_{ii}$ if $i$ is even. Then 
$h_{ii}(A\wr\mathfrak{S}_2)_1 g_{i'j'}\neq 0$ if and only if $g_{i'j'}(A\wr\mathfrak{S}_2)_1 h_{ii}\neq0$
if and only if $(i',j')=(i-1,i)$ or $(i',j')=(i,i+1)$.
\item The product maps 
$$g_{i-1,i}\tilde{g}\left(A\wr\mathfrak{S}_2\right)_1\tilde{g}
\otimes\tilde{g}\left(A\wr\mathfrak{S}_2\right)_1\tilde{g}g_{i-1,i}\rightarrow g_{i-1,i}\left(A\wr\mathfrak{S}_2\right)_2 g_{i-1,i}$$
and
$$g_{i-1,i}\tilde{g}\left(A\wr\mathfrak{S}_2\right)_1\tilde{g}
\otimes\tilde{g}\left(A\wr\mathfrak{S}_2\right)_1\tilde{g}g_{i,i+1}\rightarrow g_{i-1,i}\left(A\wr\mathfrak{S}_2\right)_2 g_{i,i+1}$$
and
$$g_{i,i+1}\tilde{g}\left(A\wr\mathfrak{S}_2\right)_1\tilde{g}
\otimes\tilde{g}\left(A\wr\mathfrak{S}_2\right)_1\tilde{g}g_{i-1,i}\rightarrow g_{i,i+1}\left(A\wr\mathfrak{S}_2\right)_2 g_{i-1,i}$$
are surjective.
\end{itemize}

\section{Open questions}

Let $r=2$. The affine
Lie algebra $gl_\infty$ is associated to the infinite 
Brauer line $U_{\mathcal X}$, via the Cartan matrix $C_{U_{\mathcal X}}(-1)$.
Via a construction of Ringel-Hall type, 
the positive part of $gl_\infty$ can be obtained as a Lie algebra of constructible functions on 
the set of indecomposable $V_{\mathcal X}$-modules on which arrows act nilpotently
(G. Lusztig, H. Nakajima).
D. Joyce has shown analogously how to associate a Lie algebra ${\mathcal L}({\mathcal A})$
to any abelian category ${\mathcal A}$. 
Are the examples ${\mathcal L}(V_{\mathcal X} \Nil)$ of any distinguished interest, when $r>2$ ?
Can one define analogues of the full Lie algebra $gl_\infty$ in this setting, 
and not only its positive part ?

\bigskip

Can one deform the Cubist algebras in an interesting way ?   
In case $r=2$, PBW deformations of $V_{\mathcal X}$ have been introduced 
by Crawley-Boevey (``deformed preprojective algebras''), 
and deformations of $V_{\mathcal X} \wr \Sigma_n$ by Etingof, and Ginzburg (``symplectic reflection algebras'').

Let $p$ be a prime number. Let $(K, {\mathcal O}, k)$ be a $p$-modular system.
In case $r=2$, there is a polynomial deformation $\tilde{U}_{\mathcal X} = {\mathcal O}[z] \otimes U_{\mathcal X}$ of $U_{\mathcal X}$ 
defined over ${\mathcal O}$, such that $\tilde{U}_{\mathcal X}/(z-p)$ is the Green order
associated to an infinite line, defined over ${\mathcal O}$. Can one make analogous constructions 
for $r>2$ ? 

Is it possible to deform the Cubist algebras, whilst
preserving all their homological structure (including decomposition matrices, etc.) ? 

\bigskip

Classify all symmetric algebras with highest weight module categories, whose Loewy length is $\leq 5$.

\bigskip

In case $r=2$, the preprojective algebra $V_{\mathcal X}$ 
is closely related to the hereditary algebra $kA_\infty$ with linear quiver of type $A_\infty$. 
There are analogues of this algebra, and its Koszul dual, in case $r>2$.
Indeed, the algebras 
$${\mathcal U_X} = \bigoplus_{  z  , x  \in {\mathcal X},i \in \mathbb{Z}} 
Ext^i_{V_{\mathcal X} \mMod}(\Delta_{V_{\mathcal X}}  (    z  ), \Delta_{V_{\mathcal X}}( x )<i>),$$
$${\mathcal V_X} = \bigoplus_{  z  , x  \in {\mathcal X},i \in \mathbb{Z}} 
Ext^i_{U_{\mathcal X} \mMod}(\Delta_{U_{\mathcal X}}  (    z  ), \Delta_{U_{\mathcal X}}( x )<i>),$$
are Koszul dual to each other, with Cartan matrices $D_{U_{\mathcal X}}(q)$,
$D_{V_{\mathcal X}}(q)$.
One can associate such a pair of algebras
to any dual pair of standard Koszul algebras. 
Are these of any interest ?

\bibliographystyle{amsplain}
\bibliography{cubistch}

\end{document}